\documentclass[a4paper,10pt]{article}
\usepackage{amsmath}
\usepackage{mathtools,xparse,enumerate,amssymb,bm,amsthm}
\usepackage{paralist}
\usepackage{graphics} 
\usepackage[backend=bibtex,style=numeric,maxnames=5,firstinits]{biblatex}
\usepackage[a4paper, headheight = 30pt, left = 38mm, top=30mm, right=43mm, bottom=35mm]{geometry}

\newtheorem{theorem}{Theorem}[section]

\newtheorem{lemma}[theorem]{Lemma}

\theoremstyle{definition}

\newtheorem{annahme}[theorem]{Assumptions}

\newtheorem{remark}[theorem]{Remark}
\newtheorem*{notation}{Notation}

\numberwithin{equation}{section}
\DeclareMathAlphabet\mathbfcal{OMS}{cmsy}{b}{n}
\DeclareMathOperator*{\esssup}{ess\,sup}

\DeclarePairedDelimiter{\norm}{\lVert}{\rVert}
\NewDocumentCommand{\norml}{ s O{} m m }{%
	\IfBooleanTF{#1}{\norm*{#4}}{\norm[#2]{#4}}_{L^{#3}}%
}
\NewDocumentCommand{\normL}{ s O{} m m }{%
	\IfBooleanTF{#1}{\norm*{#4}}{\norm[#2]{#4}}_{\L^{#3}}%
}
\NewDocumentCommand{\normw}{ s O{} m m m }{%
	\IfBooleanTF{#1}{\norm*{#4}}{\norm[#2]{#5}}_{W^{#3,#4}}%
}
\NewDocumentCommand{\normW}{ s O{} m m m }{%
	\IfBooleanTF{#1}{\norm*{#4}}{\norm[#2]{#5}}_{\W^{#3,#4}}%
}
\NewDocumentCommand{\normh}{ s O{} m m }{
	\IfBooleanTF{#1}{\norm*{#4}}{\norm[#2]{#4}}_{H^{#3}}
}
\NewDocumentCommand{\normH}{ s O{} m m }{%
	\IfBooleanTF{#1}{\norm*{#4}}{\norm[#2]{#4}}_{\H^{#3}}%
}
\newcommand{\inndual}[2]{\left \langle #1\mkern2mu{,}\mkern2mu #2 \right \rangle}
\newcommand{\inn}[2]{\left( #1\mkern2mu{,}\mkern2mu #2 \right)}
\newcommand{\R}{\mathbb{R}}

\newcommand{\mc}{\mkern2mu{,}\mkern2mu}

\def\R{\mathbb R}

\def\n{\mathbf n}
\def\nV{{\mathbf n}_V}
\def\a{\mathbf a}
\def\b{\mathbf b}

\def\L{\mathbf L}
\def\H{\mathbf H}
\def\v{\mathbf{v}}

\def\B{\mathbf{B}}

\def\T{\mathbf{T}}

\def\I{\mathbf{I}}
\def\D{{\mathbf{D}}}

\def\W{{\mathbf W}}

\def\m{{\mathbf{m}}}

\def\oldSigma{\partial\Omega\times(0,T)}

\def\intT{\int_{0}^{T}}

\def\intO{\int_\Omega}

\def\ddt{\frac{\mathrm d}{\mathrm dt}}

\def\dy{\;\mathrm dy}
\def\dz{\;\mathrm dz}
\def\dt{\;\mathrm dt}

\def\dx{\; \mathrm d \mathcal{L}^{d}}
\def\dH{\; \mathrm d \mathcal{H}^{d-1}}
\def\textdx{\mathrm d \mathcal{L}^{d}}
\def\textdH{\mathrm d \mathcal{H}^{d-1}}

\def\del{\partial}

\def\delxk{\partial_{x_k}}
\def\delt{\partial_{t}}

\def\d{{\mathrm{d}}}

\def\grad{\nabla}
\def\laplace{\Delta}
\def\divergence{\textnormal{div}}

\def\splus{\hspace{-1pt}+\hspace{-1pt}}
\def\sminus{\hspace{-2pt}-\hspace{-2pt}}
\bibliography{EGN_arxiv}
\begin{document}
	
	\title{Cahn--Hilliard--Brinkman systems for tumour growth}
	\author{Matthias Ebenbeck \footnotemark[1] \and Harald Garcke \footnotemark[2]\and Robert N\"urnberg\footnotemark[3]}
	\date{}
	
	\renewcommand{\thefootnote}{\fnsymbol{footnote}}
	\footnotetext[1]{Fakult\"at f\"ur Mathematik, Universit\"at Regensburg, 93040 Regensburg, Germany ({\tt matthias.ebenbeck@mathematik.uni-regensburg.de}).}
	\footnotetext[1]{Fakult\"at f\"ur Mathematik, Universit\"at Regensburg, 93040 Regensburg, Germany ({\tt harald.garcke@mathematik.uni-regensburg.de}).}
	\footnotetext[3]{Department of Mathematics, Imperial College London, London, SW7 2AZ, UK
		({\tt robert.nurnberg@imperial.ac.uk}).}
	
	\maketitle
	
	\begin{abstract}
		A phase field model for tumour growth is introduced that is based on a
		Brinkman law for convective velocity fields. The model couples a
		convective Cahn--Hilliard equation for the evolution of the tumour to
		a reaction-diffusion-advection equation for a nutrient and to a
		Brinkman--Stokes type law for the fluid velocity. The model is derived
		from basic thermodynamical principles, sharp interface limits are
		derived by matched asymptotics and an existence theory is presented
		for the case of a mobility which degenerates in one phase leading to a
		degenerate parabolic equation of fourth order. Finally numerical results describe
		qualitative features of the solutions and illustrate
		instabilities in certain situations. 
	\end{abstract}
	
	

	\section{Introduction}
	
	Classical continuum models for tumour growth use free boundary
	problems to describe the growth of the tumour. These models go back to the
	seminal work of Greenspan, \cite{Greenspan}, who modelled the tissue as a
	porous medium and used Darcy's law for the convective velocity
	field. This modelling approach was subsequently  further developed by many
	authors, see  \cite{AmbrosiPreziosi2, ByrneChaplain2} and the reviews
	\cite{BellomoLiMaini,Friedman2,RooseChapmanMaini}.  
	Later also Stokes flow has been used to model velocities in
	tumour growth \cite{FranksKing2,FranksKing, Friedman,
		FriedmanHu}. This is justified, since typically tissue does not have the
	characteristics of a porous medium. As tumours might undergo
	morphological instabilities like fingering or folding, see,
	e.\,g., \cite{CristiniLowengrubNie,CristiniEtAl}, free boundary
	problems in a classical formulation have their limitations, because
	changes in topology have to be dealt with.
	
	To overcome these difficulties, it has turned out that diffuse interface
	models, where the sharp
	interface is replaced by a narrow transition layer and the tumour is
	treated as a collection of
	cells, are a good alternative modelling strategy to describe the evolution and interactions
	of different species. In
	contrast to free boundary problems, there is no need to explicitly track
	the interface, or to enforce
	complicated boundary conditions across the interface, see, e.\,g.,
	\cite{WiseLowengrubFrieboesCristini}.
	Moreover, tissue interfaces
	may be more realistically represented by the diffuse interface framework,
	since phase boundaries
	between tissues may not be well delineated, see \cite{FrieboesEtAl}. These models are
	typically based on a
	multiphase approach, on balance laws for the single constituents, like
	mass and momentum
	balance, on constitutive laws and on thermodynamic principles. Several
	additional variables
	describing the extracellular matrix (ECM), growth factors or inhibitors
	can be incorporated into
	these models, and biological mechanisms like chemotaxis, apoptosis or
	necrosis and effects of
	stress, plasticity or viscoelasticity can be included, see
	\cite{CristiniLowengrub, 
		EylesKingStyles, GarckeLamNuernbergSitka,GarckeLamSignori1, GarckeLamSitkaStyles,  
		HawkinsZeeKristofferOdenTinsley, OdenHawkinsPrudhomme}. 
	
	In most of the earlier phase field models in the literature, flow velocity is modelled by Darcy's law, see
	\cite{GarckeLam1, GarckeLamSitkaStyles, JiangWZ15,LowengrubTitiZhao, WiseLowengrubFrieboesCristini}.
	However, often tissue cannot be modelled as a porous medium, see \cite{FranksKing2, FranksKing,Friedman}, and hence models based on Stokes or Brinkman flow
	have been suggested, see  \cite{EbenbeckGarcke, EbenbeckGarcke2, FritzEtAl}. It is the goal of this work to derive these models systematically
	using thermodynamic principles, and to give several examples of constitutive laws which are relevant for applications. In these models, cell adhesion
	is modelled with the help of a Ginzburg--Landau energy, see also \cite{CristiniLowengrub}, and the resulting equation for the growth of the tumour turns out to be a convective
	Cahn--Hilliard equation with sources related to proliferation (cell growth) and apoptosis (controlled cell death). In phase field models, the interface between
	the tumour and the healthy region is modelled with the help of a diffuse interface, which has a thickness that is proportional
	to a small positive parameter $\epsilon$. A further goal of this paper is to derive  sharp interface problems in the  limit as $\epsilon$ tends to zero. Here we use the
	method of formally matched asymptotic expansions to analyse the limit. In applications to tumour growth, the mobility in the Cahn--Hilliard equation typically degenerates
	in one phase (see, e.\,g., \cite{CristiniLiLowengrubWise,HawkinsZeeKristofferOdenTinsley,WiseLowengrubFrieboesCristini}), and the resulting Cahn--Hilliard equation is a degenerate Cahn--Hilliard equation, which is notoriously difficult to analyse. 
	Using entropy-like estimates, we will show existence of weak solutions, which is non-standard due to source terms in the Cahn--Hilliard equation, see also
	\cite{AgostiAntoniettiCiarlettaGrasselliVerani,FrigeriLamRocca, PerthamePoulain} 
	for similar results.
	The Brinkman model has Darcy's law and Stokes flow as singular limits. In
	numerical simulations we will analyse these limiting behaviours, as well
	as several qualitative features of the model, which include in
	particular several unstable growing fronts.
	It will turn out that for certain situations in which chemotaxis is
	present, unstable fronts appear, and we will also show that certain wave
	lengths are more unstable than others.
	
	Following this introduction, we first of all derive the governing
	equations. 
	In Section 3, we will discuss several additional modelling aspects like,
	for example, specific forms of source terms, pressure reformulations, a
	general energy inequality, boundary conditions and
	non-dimensionalisation arguments.
	Then we will use the method of formally matched asymptotics to derive
	some sharp interface models for tumour growth, which are related to free
	boundary problems that have been studied earlier in the literature. In Section 5, we present analytical results for a model with one-sided degenerate mobility and singular potential.
	In Section 6, we will show numerical simulations which give further insights into the model and the influence of different parameters. Finally, we want to fix the notation for this work:
	\begin{notation}
		We denote by $\Omega\subset\R^d$, $d=2,3$, a bounded domain with boundary $\del\Omega$ and outer unit normal $\n$, and by $T>0$ a fixed final time. We denote $Q\coloneqq \Omega\times (0,T)$. 
		For a (real) Banach space $X$ we denote by $\norm{\cdot}_X$ its norm, by $X^*$ the dual space, and by $\langle \,\cdot\mc\cdot\, \rangle_X$ the duality pairing between $X^*$ and $X$. By $(\,\cdot\mc\cdot\,)$ we denote the $L^2$ inner product in $\Omega$. We define the scalar product of two matrices by
		\begin{equation*}
		\mathbf{A}\colon \B\coloneqq \sum_{j,k=1}^{d}a_{jk}b_{jk}\quad\text{for } \mathbf{A},\B\in\R^{d\times d},
		\end{equation*}
		and the divergence of a matrix-valued function $\mathbf{A}\colon \R^d\to \R^{d\times d}$ by
		\begin{equation*}
		\divergence(\mathbf{A})\coloneqq \left(\sum_{k=1}^{d}\delxk a_{jk}(x)\right)_{j=1}^d.
		\end{equation*}
		For the standard Lebesgue and Sobolev spaces with $1\leq p\leq \infty$, $k>0$,  we use the notation $L^p\coloneqq L^p(\Omega)$ and $W^{k,p}\coloneqq W^{k,p}(\Omega)$ with norms $\norml{p}{\cdot}$ and $\normw{k}{p}{\cdot}$, respectively.  In the case $p=2$ we use $H^k\coloneqq W^{k,2}$ and the norm $\normh{k}{\cdot}$. We will denote the Lebesgue spaces on the boundary by $L^p(\del\Omega)$ with corresponding norm $\norm{\cdot}_{L^p(\del\Omega)}$. We denote the space $W_0^{k,p}$ as the completion of $C_0^{\infty}(\Omega)$ with respect to the $W^{k,p}$-norm and we set $H_0^k\coloneqq W_0^{k,2}$. By $\mathbf{L}^p$, $\mathbf{W}^{k,p}$, $\mathbf{H}^k$, $\mathbf{L}^p(\del\Omega)$, $\W_0^{k,p}$ and $\H_0^k$ we will denote the corresponding spaces of vector valued and matrix valued functions. We denote the $\L^2$ inner product of two vectors $\a,\b\in \L^2$ or two matrices $\mathbf{A},\mathbf{B}\in \L^2$ by $\inn{\a}{\b}\coloneqq \sum_{i=1}^{d}\inn{a_i}{b_i}$ and $\inn{\mathbf{A}}{\mathbf{B}}\coloneqq \sum_{j,k=1}^{d}\inn{a_{jk}}{b_{jk}}$, respectively.
		For Bochner spaces we use the notation $L^p(X)\coloneqq L^p(0,T;X)$ for a Banach space $X$ with $p\in [1,\infty]$. We define
		\begin{equation*}
		\norm{\cdot}_{A\cap B}\coloneqq \norm{\cdot}_A + \norm{\cdot}_B
		\end{equation*}
		for two or more Banach spaces $A$ and $B$.
		Moreover, we introduce the function spaces
		\begin{align*}
		&L_0^2\coloneqq \{w\in L^2\colon \inn{w}{1}=0\},\quad (H^1)_0^*\coloneqq \left\{f\in (H^1)^*\colon \inndual{f}{1}_{H^1} =0\right\} ,\\
		&H_N^2\coloneqq \left\{w\in H^2\colon \grad w\cdot\n = 0 \text{ on } \del\Omega\right\}.
		\end{align*}
		For problems related to the Stokes equation we define
		\begin{equation} \label{eq:boldV}
		\bm{\mathcal{V}}\coloneqq \left\{ \v\in C_0^{\infty}(\Omega;\R^d)\,\colon \,\divergence(\v)=0\right\},\quad \H\coloneqq \overline{\bm{\mathcal{V}}}^{\,\L^2}, \quad\mathbf{V}\coloneqq \overline{\bm{\mathcal{V}}}^{\,\H^1}.
		\end{equation}
		
	\end{notation}
	
	\section{Derivation of the model}
	Using basic thermodynamic principles and the Lagrange multiplier method of Liu and M\"uller, we will derive a general Cahn--Hilliard--Brinkman model for tumour growth including effects like, for example, diffusion, chemotaxis, active transport,  proliferation and apoptosis. This model will serve as the basis for this work, and several variants of this model will be analysed later. We use basic ideas of continuum mechanics, see, e.\,g., \cite{EckGarckeKnabner,GurtinFriedAnand}, and allow for a partial mixing of two components, see also \cite{AbelsGarckeGrun,GarckeLamSitkaStyles}. \\ [1ex]
	Let us consider a mixture consisting of tumour and healthy cells. We denote the first and second component as the healthy and tumour tissues, respectively. Furthermore, we introduce $\rho_i$, $i=1,2$, (actual mass of the component matter per volume in the mixture) and $\bar{\rho}_i$, $i=1,2$ (mass density of a pure component $i$). The mass density of the mixture is denoted by $\rho\coloneqq \rho_1+\rho_2$. We define 
	\begin{equation*}
	u_i = \frac{\rho_i}{\bar{\rho}_i}
	\end{equation*}
	as the volume fraction of component $i$ and
	\begin{equation*}
	c_i = \frac{\rho_i}{\rho}
	\end{equation*}
	as the mass concentration of the $i$-th component, and we note that $c_1+c_2=1$. 
	Physically we expect $\rho_i\in [0,\bar{\rho}_i]$ and thus $u_i\in [0,1]$. By $\v _i$, $i=1,2$, we denote the velocity of component $i$ and we make the following assumptions on our model.
	\begin{enumerate}
		\item[(i)] The excess volume due to mixing of the components is zero, i.\,e.,
		\begin{equation}
		\label{3_DER_1} u_1+u_2 = 1.
		\end{equation}
		\item[(ii)] We allow for mass exchange between the two components. Growth of the tumour is represented by mass transfer of healthy to tumour tissue and vice versa.
		\item[(iii)] We choose a volume-averaged mixture velocity, i.\,e.,
		\begin{equation}
		\label{3_DER_2}\v \coloneqq u_1\v _1+u_2\v _2.
		\end{equation}
		\item[(iv)] We assume the existence of a general chemical species acting as a nutrient for the tumour, like, for example, oxygen or glucose. The concentration of this species is denoted by $\sigma$ and it is transported by the velocity $\v $ and a diffusive flux $\textbf{J}_{\sigma}$.
	\end{enumerate}
	We remark that the choice of the mixture velocity is in contrast to \cite{LowengrubTruskinovsky}, where a barycentric/mass-averaged mixture velocity $\tilde{\v }\coloneqq  c_1\v _1+c_2\v _2$ was used, leading to a more complicated expression for the continuity equation.
	
	\subsection{Balance laws}
	We now study the balance laws for mass and momentum.
	\subsubsection{Balance of mass}
	The mass balance law in its local form for the two components is given by
	\begin{align}
	\label{3_DER_3a}\delt\rho_i + \divergence(\rho_i\v_i) = \Gamma_i,\quad i=1,2,
	\end{align}
	with source or sink terms $\Gamma_i$, $i=1,2$. 
	Dividing \eqref{3_DER_3a} by $\bar{\rho}_i$, $i=1,2$, we obtain the identities
	\begin{equation}
	\label{3_DER_4}\delt u_i + \divergence(u_i\v _i)=\frac{\Gamma_i}{\bar{\rho}_i},\quad i=1,2.
	\end{equation}
	Using \eqref{3_DER_1}, \eqref{3_DER_2} and \eqref{3_DER_4} yields
	\begin{equation}
	\label{3_DER_5}\divergence(\v ) = \divergence(u_1\v _1)+\divergence(u_2\v _2) = \sum_{i=1}^{2}\left(\frac{\Gamma_i}{\bar{\rho}_i}-\delt u_i\right) = \frac{\Gamma_1}{\bar{\rho}_1}+\frac{\Gamma_2}{\bar{\rho}_2}\eqqcolon \Gamma_{\v}.
	\end{equation}
	We introduce the fluxes
	\begin{equation*}
	\textbf{J}_i\coloneqq \rho_i(\v _i-\v ),\quad \mathbfcal{J}\coloneqq \textbf{J}_1+\textbf{J}_2,\quad\textbf{J}\coloneqq -\frac{1}{\bar{\rho}_1}\textbf{J}_1+\frac{1}{\bar{\rho}_2}\textbf{J}_2,
	\end{equation*}
	where $\textbf{J}_i$ describes the remaining diffusive flux after subtracting the flux resulting from transport along the mixture velocity. Using the identity 
	\begin{equation*}
	\mathbfcal{J} + \rho\v  = \textbf{J}_1+\textbf{J}_2 + \rho\v  = \rho_1\v _1+\rho_2\v _2
	\end{equation*}
	in conjunction with \eqref{3_DER_3a}, the equation for the mixture density reads
	\begin{equation}
	\label{3_DER_6}\delt\rho + \divergence(\rho_1\v _1+\rho_2\v _2) = \delt\rho + \divergence(\rho\v +\mathbfcal{J}) = \Gamma_1+\Gamma_2.
	\end{equation}
	In particular, we see that the flux of the mixture is decomposed into one part representing mathematical transport along the mixture velocity, and another part describing additional fluxes. In some models it is assumed that there is no gain or loss of mass locally, which is the case if $\Gamma_1 = -\Gamma_2$ in \eqref{3_DER_6}.
	From now on we denote by $\varphi\coloneqq u_2-u_1$ the difference in volume fractions of the two components. Recalling $\rho_i=\bar{\rho}_iu_i$ and using the identity
	\begin{equation*}
	\divergence(u_i\v _i) = \divergence\left(\frac{\rho_i}{\bar{\rho}_i}\v _i\right) = \divergence\left(\frac{\rho_i}{\bar{\rho}_i}(\v _i-\v +\v )\right) = \frac{1}{\bar{\rho}_i}\divergence(\textbf{J}_i) + \divergence(u_i\v ),
	\end{equation*}
	from \eqref{3_DER_4} we obtain
	\begin{equation*}
	\delt u_i + \frac{1}{\bar{\rho}_i}\divergence(\textbf{J}_i)+\divergence(u_i\v ) = \frac{\Gamma_i}{\bar{\rho_i}}.
	\end{equation*}
	Subtracting the equation for $u_1$ from the equation for $u_2$ yields
	\begin{equation}
	\label{3_DER_7}\delt\varphi + \divergence(\varphi\v )+\divergence(\textbf{J}) = \frac{\Gamma_2}{\bar{\rho}_2}-\frac{\Gamma_1}{\bar{\rho}_1}\eqqcolon\Gamma_{\varphi}.
	\end{equation}
	For the nutrient we postulate the balance law
	\begin{equation}
	\label{3_DER_8}\delt\sigma + \divergence(\sigma\v ) +\divergence\mathbf{J}_{\sigma} = -\Gamma_{\sigma},
	\end{equation}
	where $\Gamma_{\sigma}$ is a term related to sources or sinks, $\sigma\v $ models transport by the volume-averaged velocity and $\mathbf{J}_{\sigma}$ represents other transport mechanisms. 
	\subsubsection{Balance of linear momentum:}
	We make the following assumptions for our model.
	\begin{enumerate}
		\item[(i)] As in \cite{AbelsGarckeGrun}, we assume that
		the mixture with volume-averaged velocity $\v $ satisfies the balance law of linear momentum of continuum mechanics.
		\item[(ii)] We assume that inertial forces are negligible, which can be justified as the Reynolds number for biological processes like tumour growth is usually very small. Since gravity plays no role in our model of interest, and since other body forces are difficult to imagine, we neglect body forces.
		\item[(iii)]Surface forces are represented by a stress tensor $\T$, and we assume an additional source $\m$ in the momentum balance equation, which could for example represent momentum supply due to interaction forces in a porous medium, see, e.\,g., \cite{SrinivasanRajagopal}.
		\item[(iv)] We assume that the stress tensor $\T$ is symmetric, isotropic and can depend on $\grad\v$, $\varphi$, $\sigma$ and $\grad\varphi$.
	\end{enumerate}
	With all these assumptions, the balance of linear momentum takes the form
	\begin{equation}
	\label{3_DER_9}\divergence(\T) + \m = \mathbf 0,
	\end{equation}
	where $\T$ and $\m$ have to be specified by constitutive assumptions.

	\subsection{Energy inequality and the Lagrange multiplier method.}
	In an isothermal situation, the second law of thermodynamics is formulated as an energy inequality, see, e.\,g., \cite{EckGarckeKnabner, Gurtin2}.
	Thus the specific form of the stress tensor and the fluxes for $\varphi$ and  $\sigma$ depend on the choice of a suitable system energy. 
	Since we have neglected inertia effects in the momentum balance law, we assume that there is no contribution of kinetic energy. For a model including inertia effects we refer to \cite{AbelsGarckeGrun}, where the authors deduce a Navier--Stokes--Cahn--Hilliard system. We postulate a free energy of the form
	\begin{equation*} 
	e = \hat{e}(\varphi,\grad \varphi,\sigma).
	\end{equation*}
	We denote by $V(t)\subset \Omega$ an arbitrary volume which is transported with the fluid velocity. A discussion of the situation when source terms are present can be found in, e.\,g., \cite[Chap. 62]{GurtinFriedAnand}. Using the second law of thermodynamics in an isothermal situation, the following energy inequality has to hold
	\begin{align}
	\nonumber\underbrace{\frac{\d}{\d t}\int_{V(t)}e(\varphi,\grad \varphi,\sigma)\dx}_{\begin{array}{l}\text{Change of}\\ \text{energy}\end{array}}&\leq \underbrace{-\int_{\del  V(t)}\mathbf{J}_e\cdot\nV\dH}_{\begin{array}{l}\text{Energy flux across}\\ \text{the boundary}\end{array}} + \underbrace{\int_{\del  V(t)}(\T \nV)\cdot\v \dH}_{\begin{array}{l}\text{Working due to}\\ \text{macroscopic stresses}\end{array}}  \\
	\label{3_DER_10}&\quad +\underbrace{\int_{V(t)}c_{\v}\Gamma_{\v} + c_{\varphi}\Gamma_{\varphi} + c_{\sigma}(-\Gamma_{\sigma}) \dx}_{\begin{array}{c}\text{Supply of energy}\end{array}},
	\end{align}
	where $\nV$ is the outer unit normal to $\del V(t)$, $\mathbf{J}_e$ is an energy flux yet to be determined, and $\textdx$ and $\textdH$ denote integration with respect to the Lebesgue measure and the $(d-1)$-dimensional Hausdorff measure in $\R^d$, respectively. Moreover, $c_{\v}$, $c_{\varphi}$ and $c_{\sigma}$ are unknown multipliers which have to be specified. We observe that the second boundary term describes working due to the macroscopic stresses, see, e.\,g., \cite{AbelsGarckeGrun,EckGarckeKnabner,GurtinFriedAnand}.\newline 
	We introduce the material derivative of a function $f$ by
	\begin{equation*}
	\delt^{\bullet}f \coloneqq \delt f+\grad  f\cdot \v.
	\end{equation*}
	Following the arguments in, e.\,g., \cite{AbelsGarckeGrun,GarckeLamSitkaStyles}, we now apply the Lagrange multiplier method of Liu and M\"uller, which has been developed in \cite{Liu}. More precisely, we introduce Lagrange multipliers $\lambda_{\v}$, $\lambda_{\varphi}$ and $\lambda_{\sigma}$ for the equations \eqref{3_DER_5}, \eqref{3_DER_7} and \eqref{3_DER_8}. The following identity can be easily verified upon using the momentum balance equation:
	\begin{equation*}
	-\int_{\del  V(t)}(\T \nV)\cdot\v \dH = -\int_{V(t)}\divergence(\T )\cdot\v  + \T \colon\grad  \v \dx = \int_{V(t)} \m\cdot\v - \T\colon\grad\v\dx.
	\end{equation*}
	Therefore, using Reynold's transport theorem, see \cite{EckGarckeKnabner,GurtinFriedAnand}, \eqref{3_DER_10} and the identity
	\begin{align*}
	&\delt^{\bullet}e= \frac{\del  e}{\del \varphi}\delt^{\bullet}\varphi + \frac{\del  e}{\del \grad \varphi}\delt^{\bullet}(\grad \varphi)+\frac{\del  e}{\del \sigma}\delt^{\bullet}\sigma,
	\end{align*}
	the following local dissipation inequality has to be fulfilled for arbitrary values of $(\varphi,\sigma,\grad\varphi,\grad\sigma,\v,\Gamma_{\v}, \Gamma_{\varphi},\Gamma_{\sigma},\delt^{\bullet}\varphi,\delt^{\bullet}\sigma)$ 
	\begin{align*}
	-\mathcal{D}_{\text{iss}}&\coloneqq \delt^{\bullet}e + e\divergence(\v) + \divergence(\mathbf{J}_e)-\T \colon\grad \v + \m \cdot\v -c_{\v}\Gamma_{\v} - c_{\varphi}\Gamma_{\varphi} + c_{\sigma}\Gamma_{\sigma}\\
	&\quad -\lambda_{\v}(\divergence(\v)-\Gamma_{\v}) \\
	&\quad -\lambda_{\varphi}(\delt^{\bullet}\varphi+\varphi\divergence(\v )+\divergence(\mathbf{J}_{\varphi})-\Gamma_{\varphi})\\
	&\quad - \lambda_{\sigma}(\delt^{\bullet}\sigma +\sigma\divergence(\v)+\divergence(\mathbf{J}_{\sigma})+\Gamma_{\sigma})\leq 0.
	\end{align*}
	Using the identity
	\begin{equation*}
	\del _{x_j}(\delt^{\bullet}\varphi) = \delt\del _{x_j}\varphi +\v \cdot\grad (\del _{x_j}\varphi)+\del _{x_j}\v \cdot\grad \varphi = \delt^{\bullet}(\del _{x_j}\varphi)+\del _{x_j}\v \cdot\grad \varphi
	\end{equation*}
	we calculate
	\begin{equation*}
	\divergence\left(\delt^{\bullet}\varphi\frac{\del  e}{\del \grad \varphi}\right) = \delt^{\bullet}\varphi\divergence\left(\frac{\del  e}{\del \grad \varphi}\right) + \delt^{\bullet}(\grad \varphi)\cdot\frac{\del  e}{\del \grad \varphi}+\grad \v \colon\left(\grad \varphi\otimes \frac{\del  e}{\del \grad \varphi}\right).
	\end{equation*}
	Therefore, we can rewrite $-\mathcal{D}_{\text{iss}}$ as 
	\begin{align}
	\nonumber-\mathcal{D}_{\text{iss}} &= \divergence\left(\mathbf{J}_e-\lambda_{\varphi}\mathbf{J}_{\varphi}-\lambda_{\sigma}\mathbf{J}_{\sigma}+\delt^{\bullet}\varphi\frac{\del  e}{\del \grad \varphi}\right)+\grad \lambda_{\varphi}\cdot\mathbf{J}_{\varphi}+\grad \lambda_{\sigma}\cdot\mathbf{J}_{\sigma}\\
	\nonumber &\quad +\delt^{\bullet}\varphi\left(\frac{\del  e}{\del \varphi}-\divergence\left(\frac{\del  e}{\del \grad \varphi}\right)-\lambda_{\varphi}\right) + \delt^{\bullet}\sigma\left(\frac{\del  e}{\del \sigma}-\lambda_{\sigma}\right)\\
	\nonumber &\quad -\left(\T +\left(\grad \varphi\otimes \frac{\del  e}{\del \grad \varphi}\right)\right)\colon \grad \v  + \m \cdot\v \\
	\nonumber&\quad +(c_{\sigma}-\lambda_{\sigma})\Gamma_{\sigma}+(\lambda_{\v}-c_{\v})\Gamma_{\v} + (\lambda_{\varphi}-c_{\varphi})\Gamma_{\varphi}\\
	\label{3_DER_12}&\quad +\left(e-\lambda_{\varphi}\varphi-\lambda_{\sigma}\sigma-\lambda_{\v}\right)\divergence(\v)\leq 0.
	\end{align}
	Finally, we define the chemical potential as
	\begin{equation*}
	\mu\coloneqq \frac{\del e}{\del \varphi}-\divergence\left(\frac{\del e}{\del\grad \varphi}\right).
	\end{equation*} 
	\subsection{Constitutive assumptions:}
	To fulfil \eqref{3_DER_12}, we can argue as in, e.\,g., \cite{AbelsGarckeGrun,GarckeLamSitkaStyles}, and we make the following constitutive assumptions
	\begin{subequations}
		\begin{align}
		\label{3_DER_13a} \mathbf{J}_e &= \lambda_{\sigma}\mathbf{J}_{\sigma}+\lambda_{\varphi}\mathbf{J}_{\varphi}-\delt^{\bullet}\varphi\frac{\del  e}{\del \grad \varphi},\quad c_{\v}=\lambda_{\v},\\
		\label{3_DER_13b}c_{\varphi} &=\lambda_{\varphi}= \frac{\del  e}{\del \varphi}-\divergence\left(\frac{\del  e}{\del \grad \varphi}\right)=\mu,\quad c_{\sigma}=\lambda_{\sigma} = \frac{\del  e}{\del \sigma},\\
		\label{3_DER_13c}\mathbf{J}_{\varphi}&=-m(\varphi)\grad \mu,\quad \mathbf{J}_{\sigma}=-n(\varphi)\grad  \left(\frac{\del  e}{\del \sigma}\right),
		\end{align}
	\end{subequations}
	where $m(\varphi)$ and $n(\varphi)$ are non-negative mobilities corresponding to a generalised Fick's law (see \cite{AbelsGarckeGrun}). In principle, $m(\cdot)$ and $n(\cdot)$ could also depend on additional variables like $\mu$ and $\sigma$. With these choices
	\eqref{3_DER_12} simplifies to
	\begin{align}
	\label{3_DER_14} \hspace{-1pt}-\left(\T +\left(\grad \varphi\otimes \frac{\del  e}{\del \grad \varphi}\right)\right)\colon \grad \v  + \m \cdot\v +\left(e-\lambda_{\varphi}\varphi-\lambda_{\sigma}\sigma-\lambda_{\v}\right)\divergence(\v)\leq 0.
	\end{align}
	We now introduce the unknown pressure $p$ and we rewrite the stress tensor as 
	\begin{equation}
	\label{3_ST_REF_1}\T = \mathbf{S}-p\I,\quad\text{i.\,e.}\quad \mathbf{S} = \T+p\I.
	\end{equation}
	An easy calculation yields the identity
	\begin{equation*}
	\left(\grad \varphi\otimes\frac{\del  e}{\del \grad \varphi}\right)\colon \frac{1}{2}(\grad \v -(\grad \v )^{\intercal}) = \frac{1}{2}\left(\grad \varphi\otimes\frac{\del  e}{\del \grad \varphi}-\frac{\del  e}{\del \grad \varphi}\otimes\grad \varphi\right)\colon \frac{1}{2}(\grad \v -(\grad\v)^{\intercal}).
	\end{equation*}
	Since the skew symmetric part of $\grad \v $ can attain arbitrary values (see, e.\,g., \cite{AbelsGarckeGrun}), and by the symmetry of $\T$, cf. 2.\,1.\,2.\,\textnormal{(iv)}, we conclude from \eqref{3_DER_14} that
	\begin{equation*}
	\grad \varphi\otimes\frac{\del  e}{\del \grad \varphi}=\frac{\del  e}{\del \grad \varphi}\otimes\grad \varphi ,
	\end{equation*}
	which implies
	\begin{equation*}
	\left|\frac{\del  e}{\del \grad \varphi}\right|^2|\grad \varphi|^2 = \left(\grad \varphi\cdot\frac{\del  e}{\del \grad \varphi}\right)^2.
	\end{equation*}
	The last identity yields
	\begin{equation*} 
	\frac{\del  e}{\del \grad \varphi}(\varphi,\grad \varphi,\sigma) = a(\varphi,\grad \varphi,\sigma)\grad \varphi
	\end{equation*}
	for some real valued function $a(\varphi,\grad \varphi,\sigma)$.
	By the symmetry of $\mathbf{S}$ and using $\I\colon\D\v = \textnormal{tr}(\D\v)$, we obtain $(\mathbf{S}-p\I)\colon\grad \v = \mathbf{S}\colon\D\v -p\divergence(\v)$. 
	Together with \eqref{3_ST_REF_1}, this implies
	\begin{equation*}
	\T\colon \grad\v = \mathbf{S}\colon\D\v-p \,\divergence(\v).
	\end{equation*}
	This identity allows us to rewrite \eqref{3_DER_14} as
	\begin{align*}
	-\left(\mathbf{S}+\left(\grad \varphi\otimes a(\varphi,\grad \varphi,\sigma)\grad \varphi\right)\right)\colon \D\v + \m \cdot\v  +\left(e-\lambda_{\varphi}\varphi-\lambda_{\sigma}\sigma+p-\lambda_{\v}\right)\divergence(\v)\leq 0.
	\end{align*}
	In order to control the mass exchange term we set
	\begin{equation*}
	\lambda_{\v} \coloneqq e-\lambda_{\varphi}\varphi-\lambda_{\sigma}\sigma+p,
	\end{equation*}
	and therefore it remains to fulfil the inequality
	\begin{equation*}
	\big(\mathbf{S}+\left(\grad \varphi\otimes a(\varphi,\grad \varphi,\sigma)\grad \varphi\right)\big)\colon \D\v - \m \cdot\v \geq 0.
	\end{equation*}
	Similar as in, e.\,g., \cite{AbelsGarckeGrun}, and motivated by Newton's linear rheological law, we make the constitutive assumption
	\begin{equation*}
	\mathbf{S}+\grad \varphi\otimes a(\varphi,\grad \varphi,\sigma)\grad \varphi = 2\eta(\varphi)\D\v  +  \lambda(\varphi)\divergence(\v )\I,
	\end{equation*}
	where $\eta(\cdot)$ and $\lambda(\cdot)$ are non-negative functions referred to as shear and bulk viscosities. 
	This means that, on account of the last identity, the dissipation inequality \eqref{3_DER_12} holds provided
	\begin{equation*}
	- \m \cdot\v \geq 0.
	\end{equation*}
	A typical choice, see, e.\,g., \cite{OdenHawkinsPrudhomme,SrinivasanRajagopal}, is
	\begin{equation*}
	\m \coloneqq -\nu(\varphi)\v,
	\end{equation*}
	where $\nu(\cdot)$ represents the permeability and is also referred to as ``drag'' coefficient function.\\ [1ex]
	The energy flux $\mathbf{J}_e$ in \eqref{3_DER_13a} is chosen such that the divergence term in \eqref{3_DER_12} vanishes. It contains classical terms like $\mu\mathbf{J}_{\varphi}$ and $\frac{\del  e}{\del \sigma}\mathbf{J}_{\sigma}$, which describe energy flux due to mass diffusion, and the non-classical term $\delt^{\bullet}\varphi\frac{\del e}{\del \grad \varphi}$ describing working due to microscopic stresses. For more details see, e.\,g., \cite{AbelsGarckeGrun,GarckeLamSitkaStyles}.
	Collecting the results above, we arrive at the following dissipation inequality
	\begin{equation*}
	\mathcal{D}_{\text{iss}} = 2\eta(\varphi)|\D\v |^2 + \lambda(\varphi)(\divergence(\v ))^2 +\nu(\varphi)|\v |^2+ m(\varphi)|\grad \mu|^2 + n(\varphi)\left|\grad \frac{\del  e}{\del \sigma}\right|^2 \geq 0.
	\end{equation*}
	Hence dissipation is produced by the following processes: viscosity effects, changes in volume, dissipation at the pores of the mixture due to the flow, and diffusive transport induced by $\grad \mu$ and $\grad \frac{\del  e}{\del \sigma}$. 
	\subsection{The model equations:}
	{From} now on we assume a general energy of the form
	\begin{equation*}
	e(\varphi,\grad \varphi,\sigma) = f(\varphi,\grad \varphi) + N(\varphi,\sigma).
	\end{equation*}
	The first term accounts for adhesion energy of the diffuse interface, whereas the second term represents the energy contribution due to the presence of the nutrient and the interaction between the tumour tissue and the nutrients. For more details regarding the second energy term, we refer to \cite{GarckeLamSitkaStyles,HawkinsZeeKristofferOdenTinsley}. Furthermore, we assume that $f$ is of Ginzburg--Landau type, that is,
	\begin{equation*}
	f(\varphi,\grad \varphi) = \frac{\beta}{\epsilon}\psi(\varphi) + \frac{\beta\epsilon}{2}|\grad \varphi|^2,
	\end{equation*}
	where $\psi$ is a potential with minima at $\pm 1$, typically the classical double-well potential, and the parameter $\beta>0$ is a cell-cell
	adhesion parameter and $\epsilon>0$ is related to  the interfacial thickness.\newline
	With this choice we calculate
	\begin{equation*}
	\frac{\del  e}{\del \varphi} = \frac{\beta}{\epsilon}\psi'(\varphi) + N_{,\varphi}, \quad \frac{\del  e}{\del \grad \varphi} = \beta\epsilon\grad \varphi,\quad a(\varphi,\grad \varphi,\sigma) = \beta\epsilon,\quad\frac{\del  e}{\del \sigma} = N_{,\sigma},
	\end{equation*}
	where $N_{,\varphi}$ and $N_{,\sigma}$ denote the partial derivatives of $N(\varphi,\sigma)$ with respect to $\varphi$ and $\sigma$, respectively. \newline In the following we use the relation \eqref{3_ST_REF_1}. Recalling \eqref{3_DER_5}, \eqref{3_DER_7}-\eqref{3_DER_9} and using the constitutive assumptions, we obtain the following general Cahn--Hilliard--Brinkman model for tumour growth
	\begin{subequations}\label{3_MEQ}
		\begin{align}
		\label{3_MEQ_1}\divergence(\v )&=\Gamma_{\v },\\
		\label{3_MEQ_2}-\divergence(2\eta(\varphi)\D\v +\lambda(\varphi)\divergence(\v )\I) + \nu(\varphi)\v  + \grad  p &= -\divergence(\beta\epsilon \grad \varphi\otimes\grad \varphi),\\
		\label{3_MEQ_3}\delt\varphi + \divergence(\varphi\v ) &= \divergence(m(\varphi)\grad \mu)+\Gamma_{\varphi},\\
		\label{3_MEQ_4}\mu&= \tfrac{\beta}{\epsilon}\psi'(\varphi)-\beta\epsilon\laplace\varphi + N_{,\varphi},\\
		\label{3_MEQ_5}\delt\sigma + \divergence(\sigma\v ) &= \divergence(n(\varphi)\grad  N_{,\sigma}) - \Gamma_{\sigma},
		\end{align}
	\end{subequations}
	where
	\begin{equation*}
	\Gamma_{\v} = \frac{\Gamma_2}{\bar{\rho}_2}+\frac{\Gamma_1}{\bar{\rho}_1},\quad \Gamma_{\varphi} = \frac{\Gamma_2}{\bar{\rho}_2}-\frac{\Gamma_1}{\bar{\rho}_1}.
	\end{equation*}
	\section{Further aspects of modelling}
	\subsection{Specific source terms}
	We now outline specific choices of source terms that are commonly used in the literature.
	\begin{enumerate}
		\item[(i)] In some cases it is meaningful to assume no gain or loss of mass locally (see \eqref{3_DER_6}), and in this case we demand that
		\begin{equation*}
		\Gamma_2 = -\Gamma_1\eqqcolon \Gamma.
		\end{equation*}
		Then, there is a close relation between the source terms $\Gamma_{\v}$ and $\Gamma_{\varphi}$, given by
		\begin{equation}
		\label{3_SOURTE_TERM_1}\Gamma_{\varphi} = \frac{\Gamma_2}{\bar{\rho}_2}-\frac{\Gamma_1}{\bar{\rho}_1} = \left(\frac{1}{\bar{\rho}_1}+\frac{1}{\bar{\rho}_2}\right)\Gamma,\qquad \Gamma_{\v} = \frac{\Gamma_2}{\bar{\rho}_2}+\frac{\Gamma_1}{\bar{\rho}_1} = \left(\frac{1}{\bar{\rho}_2}-\frac{1}{\bar{\rho}_1}\right)\Gamma.
		\end{equation}
		In the following we set 
		\begin{equation}
		\label{3_SOURTE_TERM_2}\alpha\coloneqq \frac{1}{\bar{\rho}_2}-\frac{1}{\bar{\rho}_1},\qquad \beta\coloneqq \frac{1}{\bar{\rho}_1}+\frac{1}{\bar{\rho}_2}.
		\end{equation}
		\item[(ii)] A possible assumption for the source terms is linear kinetics (see, e.\,g., \cite{GarckeLam1,GarckeLamSitkaStyles}), and in this case one chooses
		\begin{equation}
		\label{3_SOURTE_TERM_3}\Gamma\coloneqq (\mathcal{P}\sigma-\mathcal{A})h(\varphi),\qquad \Gamma_{\sigma} = \mathcal{C}\sigma h(\varphi),
		\end{equation}
		where $\mathcal{P}$, $\mathcal{A}$ and $\mathcal{C}$ are non-negative constants related to proliferation, apoptosis and consumption. The function $h(\cdot)$ interpolates linearly between $h(-1)=0$ and $h(1)=1$ and can be extended constant outside of the interval $[-1,1]$. We refer to \cite{GarckeLamSitkaStyles} for the motivation of these specific source terms.
		\item[(iii)] Other authors use linear phenomenological laws for chemical reactions. For example, in \cite{HawkinsZeeKristofferOdenTinsley} it was suggested to take
		\begin{equation*}
		\Gamma_{\varphi} = \Gamma_{\sigma} = P(\varphi)(N_{,\sigma}-\mu)
		\end{equation*}
		for a non-negative proliferation function $P(\cdot)$. These kind of source terms have, e.\,g., been studied in \cite{ColliGilardiHilhorst,FrigeriGrasselliRocca}. In \cite{HawkinsZeeKristofferOdenTinsley} it has been proposed to take
		\begin{equation*}
		P(\varphi) = \begin{cases} 
		\delta P_0 (1+\varphi)&\text{if }\varphi\geq -1,\\
		0&\text{elsewhere}
		\end{cases}
		\end{equation*}
		for positive constants $\delta$ and $P_0$, where $\delta$ is usually very small. In contrast, the authors in \cite{HilhorstKampmannNguyenZee} considered a proliferation function given by
		\begin{equation*}
		P(\varphi) = \begin{cases}
		2\epsilon^{-1}P_0\sqrt{\psi(\varphi)}&\text{if }\varphi\in [-1,1],\\
		0&\text{elsewhere}.
		\end{cases}
		\end{equation*}
		\item[(iv)] Taking $\Gamma_1 = 0$ and $\Gamma = \Gamma_2$ one obtains
		\begin{equation*}
		\Gamma_{\varphi} = \Gamma_{\v} = \frac{1}{\bar{\rho}_2}\Gamma.
		\end{equation*}
		This choice will be of importance when deriving the formal asymptotic sharp interface limit for a mobility of the form $m(\varphi) = m_0\epsilon$ with a positive constant $m_0$, where source terms of the  form \eqref{3_SOURTE_TERM_1} with $\Gamma$ as in \eqref{3_SOURTE_TERM_3} do not fulfil a corresponding compatibility condition.
	\end{enumerate}

	\subsection{Boundary and initial conditions}
	We prescribe homogeneous Neumann boundary conditions for the phase field variable, the chemical potential and the stress tensor, i.\,e.,
	\begin{subequations}\label{3_CHB_BIC}
		\begin{alignat}{3}
		\label{3_CHB_BC_1}&\grad\varphi\cdot\n = \grad\mu\cdot\n = 0&&\qquad\text{a.\,e. on }\oldSigma,\\
		\label{3_CHB_BC_2}&\T(\v,p)\n = \mathbf{0}&&\qquad\text{a.\,e. on }\oldSigma.
		\end{alignat}
		For the nutrient we may prescribe Robin-type boundary conditions of the form
		\begin{equation}
		\label{3_CHB_BC_3}n(\varphi)\grad N_{,\sigma}\cdot\n = K(\sigma_{\infty}-\sigma)\qquad\text{a.\,e.  on }\oldSigma
		\end{equation}
		for a constant $K\geq 0$ referred to as the boundary permeability, and $\sigma_{\infty}$ denoting a given nutrient supply at the boundary. We may see $\sigma_{\infty}$ as a far-field nutrient level outside of $\Omega$, and recalling \eqref{3_DER_13c} we can rewrite \eqref{3_CHB_BC_3} as
		\begin{equation*}
		\mathbf{J}_{\sigma}\cdot\n = K(\sigma-\sigma_{\infty}).
		\end{equation*}
		Thus we see that there is nutrient outflow if $\sigma>\sigma_{\infty}$, i.\,e., the nutrient concentration on the boundary is higher than the far-field nutrient level, and inflow if $\sigma_{\infty}>\sigma$. The rate of inflow or outflow depends on the boundary permeability $K$.
		Finally, we impose the initial conditions
		\begin{equation}
		\label{3_CHB_IC_3}\varphi(0) = \varphi_0,\quad\sigma(0)= \sigma_0\quad\text{a.\,e.\ in }\Omega
		\end{equation}
	\end{subequations}
	with prescribed functions $\varphi_0$, $\sigma_0$.
	The Robin boundary condition \eqref{3_CHB_BC_3} can be interpreted as an interpolation between Neumann and Dirichlet boundary conditions. Indeed, the case $K=0$, that means no boundary permeability, corresponds to the Neumann type boundary condition
	\begin{equation*}
	n(\varphi)\grad N_{,\sigma}\cdot\n=0\quad\text{a.\,e. on }\oldSigma,
	\end{equation*}
	whereas formally sending $K\to\infty$ gives a Dirichlet boundary condition of the form
	\begin{equation*}
	\sigma=\sigma_{\infty}\quad\text{a.\,e. on }\oldSigma.
	\end{equation*}
	\subsection{Specific form of the nutrient energy}
	For the rest of this paper we consider a nutrient energy density of the form
	\begin{equation}
	\label{3_NUTRIENT_ENERGY}N(\varphi,\sigma)\coloneqq \frac{\chi_{\sigma}}{2}|\sigma|^2 + \chi_{\varphi}\sigma(1-\varphi)
	\end{equation}
	for positive constants $\chi_{\sigma}$ and $\chi_{\varphi}$ referred to as the nutrient diffusion and chemotaxis parameter, respectively. \newline The first term characterises energy effects due to the presence of the nutrient, i.\,e., a high concentration of nutrients leads to a high energy of the system. The second term accounts for chemotaxis effects, i.\,e., tumour cells move towards regions of high nutrient concentration. We refer to \cite{GarckeLamSitkaStyles,HawkinsZeeKristofferOdenTinsley} for more details regarding this form of the nutrient energy. Using \eqref{3_NUTRIENT_ENERGY} we compute
	\begin{equation*}
	N_{,\sigma} = \chi_{\sigma}\sigma + \chi_{\varphi}(1-\varphi),\qquad N_{,\varphi} = -\chi_{\varphi}\sigma.
	\end{equation*}
	Therefore, the fluxes $\mathbf{J}_{\varphi}$ and $\mathbf{J}_{\sigma}$ are given by 
	\begin{equation*}
	\mathbf{J}_{\varphi} = -m(\varphi)\grad\left(\tfrac{\beta}{\epsilon}\psi'(\varphi)-\beta\epsilon\laplace\varphi-\chi_{\varphi}\sigma\right),\qquad \mathbf{J}_{\sigma} = -n(\varphi)\grad\left(\chi_{\sigma}\sigma-\chi_{\varphi}\varphi\right).
	\end{equation*}
	There are two non-standard contributions in the definition of $\mathbf{J}_{\varphi}$ and $\mathbf{J}_{\sigma}$. The term $m(\varphi)\grad(\chi_{\varphi}\sigma)$ drives the tumour cells towards regions of high nutrient concentrations and is referred to as chemotaxis. \newline Moreover, we encounter a term of the form $n(\varphi)\grad(\chi_{\varphi}\varphi)$ driving the nutrients towards regions with higher tumour concentrations. This effect is called active transport and seems to be counter-intuitive at first glance. However, it can be observed for malign tumours in, e.\,g., the avascular growth phase. Indeed, to overcome nutrient limitations, some tumours express more glucose transporters to provide an increasing glucose transport through the cell membrane. We remark that this term is only active on the interface and we refer to \cite{GarckeLamSitkaStyles} for more details.\\ [1ex]
	In general we can decouple chemotaxis and active transport mechanisms by introducing the scaled mobility
	\begin{equation}
	\label{3_DECOUPLE_AT_CHEMO_1} \mathcal{D}(\varphi)\coloneqq \chi_\sigma n(\varphi),
	\end{equation}
	and setting $\chi = \frac{\chi_{\varphi}}{\chi_\sigma}$.
	Then, the fluxes can be rewritten as
	\begin{equation*} 
	\mathbf{J}_{\varphi} = -m(\varphi)\grad\left(\tfrac{\beta}{\epsilon}\psi'(\varphi)-\beta\epsilon\laplace\varphi-\chi_{\varphi}\sigma\right),\qquad \mathbf{J}_{\sigma} = -\mathcal{D}(\varphi)\grad\left(\sigma-\chi\varphi\right).
	\end{equation*}
	By formally sending $\chi\to 0$ we can switch off active transport while preserving the chemotaxis mechanism.
	
	\section{Formally matched asymptotics}
	In the following we formally derive the sharp interface limit of the system
	\begin{subequations}\label{3_basic_equation}
		\begin{align}
		\label{3_basic_equation_1a}\hspace{-6pt}\divergence(\v) &= \bar{\rho}_2^{-1}\Gamma_2(\varphi,\sigma,\mu) + \bar{\rho}_1^{-1}\Gamma_1(\varphi,\sigma,\mu),\\
		\label{3_basic_equation_1b}\hspace{-6pt}-\divergence(\T(\varphi,\v,p)) +\nu(\varphi)\v &= (\mu +\chi_{\varphi}\sigma)\grad\varphi,\\
		\label{3_basic_equation_1c}\hspace{-6pt}\delt\varphi + \divergence(\varphi\v) &= \divergence(m(\varphi)\grad\mu)+\bar{\rho}_2^{-1}\Gamma_2(\varphi,\sigma,\mu) - \bar{\rho}_1^{-1}\Gamma_1(\varphi,\sigma,\mu),\\
		\label{3_basic_equation_1d}\hspace{-6pt}\mu &= \tfrac{\beta}{\epsilon}\psi'(\varphi)-\beta\epsilon\Delta\varphi-\chi_{\varphi}\sigma,\\
		\label{3_basic_equation_1e}\hspace{-6pt}\delt\sigma + \divergence(\sigma\v)&= \divergence(n(\varphi)(\chi_{\sigma}\grad\sigma-\chi_{\varphi}\grad\varphi)) - \Gamma_{\sigma}(\varphi,\sigma,\mu),
		\end{align}
	\end{subequations}
	where 
	\begin{equation*}
	\T(\varphi,\v,p)\coloneqq 2\eta(\varphi)\D\v + \lambda(\varphi)\divergence(\v)\I-p\I.
	\end{equation*}
	The adhesion term $(\mu+\chi_{\varphi}\sigma)\grad\varphi$ in \eqref{3_basic_equation_1b} follows from a reformulation of the pressure. In fact, the term $-\divergence(\beta\epsilon\grad\varphi\otimes\grad\varphi)$ in \eqref{3_MEQ_2} is up to a gradient equal to $(\mu+\chi_{\varphi}\sigma)\grad\varphi$ and the gradient term can be absorbed into the pressure, see \cite{GarckeLamSitkaStyles} for details. We will focus on the double-well potential given by
	\begin{equation*}
	\psi(\varphi) = \frac{1}{4}(1-\varphi^2)^2,
	\end{equation*}
	and satisfying
	\begin{equation*}
	\psi'(\varphi) = \varphi^3-\varphi ,\quad \psi''(\varphi) = 3\varphi^2-1.
	\end{equation*}
	Moreover, we assume that $\eta(\cdot)$, $\lambda(\cdot)$, $\nu(\cdot)$ are smooth with $\eta(\cdot)$, $\nu(\cdot)$ positive and $\lambda(\cdot)$ non-negative. For the mobility $m(\cdot)$ we consider the following three cases:
	\begin{equation}
	\label{3_mobility}m(\varphi)=\begin{cases}
	m_0&\text{Case (i)},\\
	\epsilon m_0&\text{Case (ii)},\\
	\tfrac{m_1}{2}(1+\varphi)^2&\text{Case (iii)}.
	\end{cases}
	\end{equation}
	
	\subsection{Outer Expansion}
	\subsubsection{Assumptions}

	We make the following assumptions (compare \cite{GarckeLamSitkaStyles}).
	\begin{enumerate}
		\item[(i)] For any $\epsilon>0$ small enough, there exists a family $(\varphi_{\epsilon},\mu_{\epsilon},\sigma_{\epsilon}, \v_{\epsilon},p_{\epsilon})_{\epsilon>0}$ of solutions to \eqref{3_basic_equation_1a}-\eqref{3_basic_equation_1e} which are sufficiently smooth.
		\item[(ii)] We assume that 
		\begin{equation*}
		\Sigma(\epsilon)\coloneqq \{(x,t)\in  \Omega\times[0,T]\colon\, \varphi_{\epsilon}(x,t)=0\}
		\end{equation*}
		are evolving hypersurfaces (see, e.\,g., \cite[Def. 23]{BarrettGarckeNuernberg}) that do not intersect with $\del\Omega$ and we define
		\begin{equation*}
		\Sigma(\epsilon,t)\coloneqq \{x\in\Omega\colon\, \varphi_{\epsilon}(x,t)=0\}.
		\end{equation*}
		We assume that for every $\epsilon>0$ small enough, and for each time $t\in[0,T]$, the domain $\Omega$ can be divided into two open subdomains
		\begin{equation*}
		\Omega_+(\epsilon,t)\coloneqq \{x\in\Omega\colon\, \varphi_{\epsilon}(x,t)>0\}, \quad \Omega_-(\epsilon,t)\coloneqq \{x\in\Omega\colon\, \varphi_{\epsilon}(x,t)<0\}
		\end{equation*}
		separated by $\Sigma(\epsilon,t)$ such that $\Omega_{+}(\epsilon,t)$ is enclosed by $\Sigma(\epsilon,t)$. Thus, for all $\epsilon>0$ small enough and all $t\in[0,T]$ it holds that 
		\begin{equation*}
		\Omega = \Omega_+(\epsilon,t)\,\cup\, \Sigma(\epsilon,t)\,\cup\, \Omega_-(\epsilon,t),\quad \Sigma(\epsilon,t)=\del\Omega_+(\epsilon,t),\quad \Omega_+(\epsilon,t)=\Omega\,\backslash\, \overline{\Omega_-(\epsilon,t)}.
		\end{equation*}
		We show a sketch of the typical situation in Figure \ref{3_FIG_Typical_situation}.
		\begin{figure}[!h]
			\centering
			\includegraphics[angle=-0,width=0.55\textwidth]{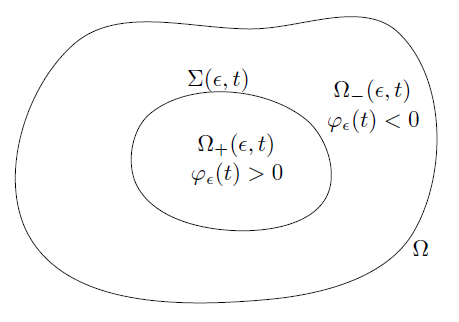}
			\caption{Typical situation for the formal asymptotic analysis.}
			\label{3_FIG_Typical_situation}
		\end{figure}
		\item[(iii)] We assume that $(\varphi_{\epsilon}, \v_{\epsilon},p_{\epsilon},\mu_{\epsilon},\sigma_{\epsilon})_{\epsilon>0}$ have an asymptotic expansion in $\epsilon$ in the bulk regions away from $\Sigma(\epsilon)$ (\textbf{outer expansion}), and another expansion in the interfacial region close to $\Sigma(\epsilon)$ (\textbf{inner expansion}).
		\item[(iv)] The zero level sets of $\varphi_{\epsilon}$ depend smoothly on $t$ and $\epsilon$  and converge as $\epsilon\to 0$ to a limiting evolving hypersurface $\Sigma(0)$ which evolves with normal velocity $\mathcal{V}$.
	\end{enumerate}
	From now on we will often drop the dependence on the time variable $t$. We use the notation $\eqref{3_basic_equation_1d}_O^{a}$ and $\eqref{3_basic_equation_1d}_I^{a}$ for the terms resulting from the order $a$ outer and inner expansions of $\eqref{3_basic_equation_1d}$, respectively.  
	\subsubsection{Expansion to leading order} We assume that $f_{\epsilon}\in\{\varphi_{\epsilon},\mu_{\epsilon},\sigma_{\epsilon}, \v_{\epsilon},p_{\epsilon}\}$ can be expanded by
	\begin{equation*}
	f_{\epsilon} = f_0 + \epsilon f_1 + \epsilon^2 f_2+ \dots\,.
	\end{equation*}
	Then, to leading order, $\eqref{3_basic_equation_1d}_O^{-1}$ yields
	\begin{equation}
	\label{3_outer_expansion_d_minus_1} -\beta\psi'(\varphi_0) = 0.
	\end{equation} 
	Stable solutions of \eqref{3_outer_expansion_d_minus_1} are the minima of $\psi(\cdot)$, and they are given by $\varphi_0 = \pm 1$. Consequently, we define
	\begin{equation*}
	\Omega_T \coloneqq \{x\in\Omega\colon \varphi_0(x)=1\},\quad \Omega_H \coloneqq \{x\in\Omega\colon \varphi_0(x)=-1\}.
	\end{equation*}
	The typical situation for $\Omega_T$ and $\Omega_H$ is shown in Figure \ref{3_FIG_TYP_SIT_TUM_HOST_REGIONS}. 
\begin{figure}[!h]
	\centering
	\includegraphics[angle=-0,width=0.55\textwidth]{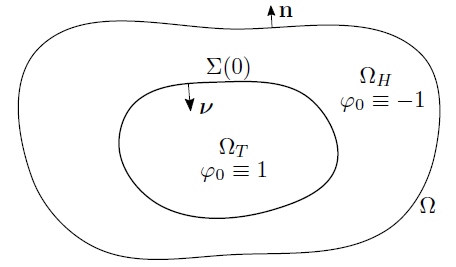}
	\caption{The tumour and healthy regions $\Omega_T$ and $\Omega_H$.}
	\label{3_FIG_TYP_SIT_TUM_HOST_REGIONS}
\end{figure}
	
	\noindent Since $\grad\varphi_0 =\mathbf{0}$, $\delt\varphi_0 =0$ in $\Omega_T$ and $\Omega_H$, we obtain for the equations to zeroth order that
	\begin{subequations}
		\begin{align}
		\label{3_outer_equations_1a}\divergence(\v_0) &= \tfrac{1}{\bar{\rho}_2}\Gamma_2(\varphi_0,\sigma_0,\mu_0) + \tfrac{1}{\bar{\rho}_1}\Gamma_1(\varphi_0,\sigma_0,\mu_0),\\
		\label{3_outer_equations_1b}-\divergence(\T(\varphi_0,\v_0,p_0))  +\nu(\varphi_0)\v_0&= 0,\\
		\label{3_outer_equations_1c}-\divergence (m(\varphi_0)\grad\mu_0) &=\tfrac{1}{\bar{\rho}_2}\Gamma_2(\varphi_0,\sigma_0,\mu_0)(1-\varphi_0) \nonumber \\ & \quad - \tfrac{1}{\bar{\rho}_1}\Gamma_1(\varphi_0,\sigma_0,\mu_0)(1+\varphi_0),\\
		\label{3_outer_equations_1d} \delt \sigma_0 + \divergence(\sigma_0\v_0)&= \divergence (n(\varphi_0)\chi_{\sigma}\grad \sigma_0) + \Gamma_{\sigma}(\varphi_0,\sigma_0,\mu_0),
		\end{align}
	\end{subequations}
	where 
	\begin{equation*}
	\T(\varphi_0,\v_0,p_0)= 2\eta(\varphi_0)\D\v_0 + \lambda(\varphi_0)\divergence(\v_0)\I-p_0\I.
	\end{equation*}
	We now analyse the three different cases for \eqref{3_basic_equation_1c} according to the mobilities introduced in \eqref{3_mobility}.\\ [1ex]
	\textit{Case (i)} ($m(\varphi)=m_0$): In this case we obtain
	\begin{subequations}
		\begin{equation}
		\label{3_outer_equations_1c_i}-m_0\laplace\mu_0 = \bar{\rho}_2^{-1}\Gamma_2(\varphi_0,\sigma_0,\mu_0)(1-\varphi_0)- \bar{\rho}_1^{-1}\Gamma_1(\varphi_0,\sigma_0,\mu_0)(1+\varphi_0).
		\end{equation}
		\textit{Case (ii)} ($m(\varphi)=\epsilon m_0$): The mobility is rescaled and the chemical potential does not contribute to the equations at zeroth order. Indeed, we have
		\begin{equation}
		\label{3_outer_equations_1c_ii}\bar{\rho}_2^{-1}\Gamma_2(\varphi_0,\sigma_0,\mu_0)(1-\varphi_0) = \bar{\rho}_1^{-1}\Gamma_1(\varphi_0,\sigma_0,\mu_0)(1+\varphi_0).
		\end{equation}
		\textit{Case (iii)} ($m(\varphi)=\tfrac{m_1}{2}(1+\varphi)^2$): The degenerate mobility case leads to
		\begin{equation}
		\label{3_outer_equations_1c_iii} -\divergence(\tfrac{m_1}{2}(1+\varphi_0)^2\,\grad\mu_0) = \bar{\rho}_2^{-1}\Gamma_2(\varphi_0,\sigma_0,\mu_0)(1-\varphi_0)- \bar{\rho}_1^{-1}\Gamma_1(\varphi_0,\sigma_0,\mu_0)(1+\varphi_0).
		\end{equation}
	\end{subequations}
	\begin{remark}\begin{enumerate}
			\item[(i)] In order to fulfil \eqref{3_outer_equations_1c_ii} we have to assume that
			\begin{equation}
			\label{3_conditions_on_Gamma}\Gamma_1(1,\sigma_0,\mu_0) = 0\quad\text{and}\quad \Gamma_2(-1,\sigma_0,\mu_0) = 0.
			\end{equation}
			Furthermore, we observe that for general source terms the chemical potential $\mu_0$ appears on the right hand side of \eqref{3_outer_equations_1a} although the bulk equations for $\mu_0$ remain undetermined. Therefore, it is reasonable to assume that the source terms are either independent of $\mu$, i.\,e., 
			\begin{equation}
			\label{3_conditions_on_Gamma_a}\Gamma_1 = \Gamma_1(\varphi,\sigma),\quad \Gamma_2 = \Gamma_2(\varphi,\sigma),
			\end{equation}
			or we may ask for
			\begin{equation}
			\label{3_conditions_on_Gamma_b}\Gamma_1(\pm 1,\sigma,\mu) = 0,\quad \Gamma_2(\pm 1,\sigma,\mu) = 0.
			\end{equation}
			To fulfil \eqref{3_conditions_on_Gamma} and \eqref{3_conditions_on_Gamma_a} we could choose
			\begin{equation*}
			\Gamma_1\equiv 0,\quad \Gamma_2(\varphi,\sigma)\coloneqq \frac{\bar{\rho}_2}{2}\left(\frac{1}{\bar{\rho}_2}-\frac{1}{\bar{\rho}_1}\right)(\mathcal{P}\sigma-\mathcal{A})(1+\varphi),
			\end{equation*}
			where $\mathcal{P}$ and $\mathcal{A}$ are non-negative constants related to proliferation and apoptosis, respectively. In this case the source terms in \eqref{3_basic_equation_1a}, \eqref{3_basic_equation_1c} coincide and are of the form
			\begin{equation*}
			\Gamma_{\varphi}(\varphi,\sigma) = \Gamma_{\v}(\varphi,\sigma) = \frac{\alpha}{2}(\mathcal{P}\sigma-\mathcal{A})(1+\varphi),
			\end{equation*}
			where 
			\begin{equation*}
			\alpha\coloneqq \frac{1}{\bar{\rho}_2}-\frac{1}{\bar{\rho}_1}.
			\end{equation*}
			Equation \eqref{3_conditions_on_Gamma} can be interpreted as follows:
			\begin{enumerate}
				\item[$\bullet$] in the pure tumour phases, there can be no growth of healthy cells,
				\item[$\bullet$] in regions of unmixed healthy tissue, there is no spontaneous growth of tumour cells.
			\end{enumerate}
			In a situation where we assume no gain or loss of mass locally, i.\,e., $\Gamma_2 = -\Gamma_1$, condition \eqref{3_conditions_on_Gamma} implies that
			\begin{equation*}
			\Gamma_1(\pm 1,\sigma_0,\mu_0) = \Gamma_2(\pm 1,\sigma_0,\mu_0)=0,
			\end{equation*} 
			which coincides with \eqref{3_conditions_on_Gamma_b}.
			Hence death and growth are restricted to the interfacial region and we may choose, for example,
			\begin{equation*}
			\Gamma_1(\varphi,\sigma,\mu) = \gamma_1(\varphi,\sigma,\mu)(1-\varphi^2)_{+}
			\end{equation*}
			for a function $\gamma_1$ to be specified. Alternatively we could use phenomenological laws to describe growth and death by choosing
			\begin{equation*}
			\Gamma_2 = -\Gamma_1 = P_1(\varphi)(\chi_{\sigma}\sigma + \chi_{\varphi}(1-\varphi)-\mu),
			\end{equation*} 
			where $P_1(\cdot)$ is a proliferation function satisfying $P_1(\pm 1)=0$. For instance, we could take $P_1(\varphi)=\frac{1}{4}(1-\varphi^2)^2$.
			\item [(ii)]
			In the healthy region \eqref{3_outer_equations_1c_iii} simplifies to
			\begin{equation*}
			0 = 2\bar{\rho}_2^{-1} \Gamma_2(-1,\sigma_0,\mu_0).
			\end{equation*}
			This is a compatibility for the source term $\Gamma_2$.
			For similar reasons as before, we can assume that either the source terms are independent of $\mu$ or
			\begin{equation*}
			\Gamma_1(-1,\sigma,\mu) = \Gamma_2(-1,\sigma,\mu)=0.
			\end{equation*}
			Reasonable choices are
			\begin{equation*}
			\Gamma_2(\varphi,\sigma) = \gamma_2(\varphi,\sigma)(1+\varphi)_{+}
			\end{equation*}
			for some function $\gamma_2$, or
			\begin{equation*}
			\frac{\Gamma_2}{\bar{\rho}_2} = -\frac{\Gamma_1}{\bar{\rho}_1} = P_2(\varphi)(\chi_{\sigma}\sigma + \chi_{\varphi}(1-\varphi)-\mu),
			\end{equation*}
			where $P_2(\varphi)=p_0(1+\varphi)_{+}$. This can be interpreted as a scaled zero excess of total mass and we have
			\begin{equation*}
			\Gamma_{\varphi} = 2P_2(\varphi)(\chi_{\sigma}\sigma + \chi_{\varphi}(1-\varphi)-\mu),\quad\Gamma_{\v}=0.
			\end{equation*}
			If the mobility was degenerate in both phases we would obtain the same condition as in \eqref{3_conditions_on_Gamma}.
			\item[(iii)] Similar conditions have to hold for the source term $\Gamma_{\sigma}$. From now on we assume that the source terms are independent of $\mu$.
		\end{enumerate}
	\end{remark}
	\subsection{Inner Expansion}
	\subsubsection{New Coordinates and matching conditions} This subsection uses ideas presented in \cite{AbelsGarckeGrun} and \cite{GarckeStinner}. We denote by $\Sigma(0)$ the smooth evolving interface which is assumed to be the limit of the zero level sets $\Sigma(\epsilon)$ of $\varphi_{\epsilon}$ as $\epsilon\to 0$ (see, e.\,g., \cite{GarckeStinner} for details). 
	We now introduce new coordinates in a neighbourhood of $\Sigma(0)$. To this end, we choose a time interval $I\subset\R$ and a spatial parameter domain $U\subset\R^{d-1}$, and we define a local parametrisation of $\Sigma(0)$ by
	\begin{equation*}
	\gamma\colon U\times I\to \R^d.
	\end{equation*}
	By $\boldsymbol{\nu}$ we denote the unit normal to $\Sigma(0)$ pointing into the tumour region. Close to $\gamma(U\times I)$ we consider the signed distance function $d(x,t)$ of a point $x$ to $\Sigma(0,t)$ with $d(x,t)>0$ if $x\in \Omega_T$ and $d(x,t)<0$ if $x\in \Omega_H$. We introduce a local parametrisation of $\R^d\times I$ near $\gamma(U\times I)$ using the rescaled distance $z = \frac{d}{\epsilon}$ by
	\begin{equation*}
	G^{\epsilon}(s,z,t)\coloneqq (\gamma(s,t)+\epsilon z\boldsymbol{\nu}(s,t),t)
	\end{equation*}
	with $s\in U\subset\R^{d-1}$. 
We show a sketch of the situation in Figure \ref{3_FIG_Schem_sketch_inner_reg}.
\begin{figure}[!h]
	\centering
	\includegraphics[angle=-0,width=0.55\textwidth]{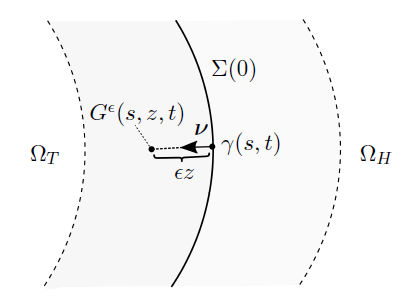}
	\caption{Schematic sketch of the inner region close to $\Sigma(0)$.}
	\label{3_FIG_Schem_sketch_inner_reg}
\end{figure}
	
	\noindent The (scalar) normal velocity is given by
	\begin{equation*}
	\mathcal{V} = \delt\gamma\cdot\boldsymbol{\nu},
	\end{equation*}
	and we observe that $(G^{\epsilon})^{-1}(x,t)\eqqcolon (s,z,t)(x,t)$ fulfils
	\begin{equation*}
	\delt z = \frac{1}{\epsilon}\delt d = -\frac{1}{\epsilon}\mathcal{V}.
	\end{equation*}
	In particular, it holds that $\boldsymbol{\nu}(x,t) = \grad d(x,t)$ on $\Sigma(0,t)$.\\ [1ex] Let $b(x,t)$ be a scalar function and define $B(s(x,t),z(x,t),t)=b(x,t)$. Then, in the new coordinate system, we obtain
	\begin{equation*}
	\ddt b(x,t) = \delt B + \del_z B\delt z+\grad_s B\cdot\delt s = -\frac{1}{\epsilon}\mathcal{V}\del_z B + \text{h.\,o.\,t.}\,.
	\end{equation*}
	For the gradient of $b$ we have
	\begin{equation*}
	\grad_x b = \grad_{\Sigma_{\epsilon z}} B + \frac{1}{\epsilon}\del_z B\boldsymbol{\nu},
	\end{equation*}
	where $\grad_{\Sigma_{\epsilon z}}$ is the surface gradient on $\Sigma_{\epsilon z}\coloneqq \{\gamma(s)+\epsilon z\boldsymbol{\nu}\colon s\in U\}$.\newline
	For a vector quantity $\mathbf{j}(x,t)=\mathbf{J}(s(x,t),z(x,t),t)$ we obtain
	\begin{equation*}
	\grad_x\cdot \mathbf{j} = \frac{1}{\epsilon}\del_z\mathbf{J}\cdot\boldsymbol{\nu}+\divergence_{\Sigma_{\epsilon z}}\mathbf{J}
	\end{equation*}
	with $\divergence_{\Sigma_{\epsilon z}}$ being the surface divergence on $\Sigma_{\epsilon z}$. Furthermore, it holds
	\begin{equation*}
	\Delta_x b(x,t) = \frac{1}{\epsilon^2}\del_{zz}B-\frac{1}{\epsilon}\kappa\del_z B + \text{h.\,o.\,t.}\,,
	\end{equation*}
	where $\kappa$ is the mean curvature of $\Sigma(0)$. In addition, we have
	\begin{align*}
	\grad_{\Sigma_{\epsilon z}}B(s,z) &= \grad_{\Sigma(0)}B(s,z) + \text{h.\,o.\,t.}\,,\\
	\divergence_{\Sigma_{\epsilon z}} \mathbf{J}(s,z) &= \divergence_{\Sigma(0)}\mathbf{J}(s,z) + \text{h.\,o.\,t.}\,,\\
	\laplace_{\Sigma_{\epsilon z}}B(s,z) &= \laplace_{\Sigma(0)}B(s,z) + \text{h.\,o.\,t.}\,.
	\end{align*}
	Summarising all the identities deduced so far yields
	\begin{subequations}\label{3_new_coord_ident}
		\begin{align}
		\label{3_new_coord_ident_1a}\frac{\d}{\d t}b(x,t) &= -\frac{1}{\epsilon}\mathcal{V}\del_z B + \text{h.\,o.\,t.}\,,\\
		\label{3_new_coord_ident_1b}\grad_x b(x,t) &=  \frac{1}{\epsilon}\del_zB\boldsymbol{\nu} +\grad_{\Sigma(0)}B + \text{h.\,o.\,t.}\,,\\
		\label{3_new_coord_ident_1c}\Delta_x b(x,t) &=  \frac{1}{\epsilon^2}\del_{zz}B- \frac{1}{\epsilon}\kappa\del_{z}B + \text{h.\,o.\,t.}\,,\\
		\label{3_new_coord_ident_1d}\divergence_x\mathbf{j}(x,t) &=  \frac{1}{\epsilon}\del_z\mathbf{J}\cdot\boldsymbol{\nu} +\divergence_{\Sigma(0)}\mathbf{J} +\text{h.\,o.\,t.}\,.
		\end{align}
		Using \eqref{3_new_coord_ident_1b}-\eqref{3_new_coord_ident_1c} component-wise we obtain
		\begin{align}
		\label{3_new_coord_ident_1e}\grad_x \mathbf{j} &=  \frac{1}{\epsilon}\del_z\mathbf{J}\otimes\boldsymbol{\nu} +\grad_{\Sigma(0)}\mathbf{J} + \text{h.\,o.\,t.}\,,\\
		\label{3_new_coord_ident_1f}\Delta_x\mathbf{j} &= \frac{1}{\epsilon^2}\del_{zz}\mathbf{J} - \frac{1}{\epsilon}\kappa \del_z\mathbf{J} + \text{h.\,o.\,t.}\,.
		\end{align}
	\end{subequations}
	We denote the variables $\varphi_{\epsilon}$, $\mu_{\epsilon}$, $\sigma_{\epsilon}$, $\v_{\epsilon}$, $p_{\epsilon}$, in the new coordinate system by $\Phi_{\epsilon}$, $\Xi_{\epsilon}$, $C_{\epsilon}$, $\mathbf{V}_{\epsilon}$, $ P_{\epsilon}$, and we assume the following inner expansion
	\begin{equation*}
	F_{\epsilon}(s,z) = F_0(s,z)+\epsilon F_1(s,z) +\epsilon^2 F_2(s,z)+ \dots
	\end{equation*}
	for $F_{\epsilon}\in \{\Phi_{\epsilon},\Xi_{\epsilon},C_{\epsilon},\mathbf{V}_{\epsilon},P_{\epsilon}\}$.
	The assumption that the zero level sets of $\varphi_{\epsilon}$ converge to $\Sigma(0)$ implies
	\begin{equation*}
	\Phi_0(s,z=0,t) =0.
	\end{equation*}
	We will employ the matching conditions (see \cite{GarckeLamSitkaStyles})
	\begin{subequations}
		\begin{align}
		\label{3_matching_cond_1a}\lim_{z\to\pm \infty}F_0(s,z,t) &= f_0^{\pm}(x,t),\\
		\label{3_matching_cond_1b}\lim_{z\to\pm \infty} \del_z F_0(s,z,t) &=0,\\
		\label{3_matching_cond_1c}\lim_{z\to\pm \infty}\del_z F_1(s,z,t) &= \grad f_0^{\pm}(x,t)\cdot\boldsymbol{\nu},
		\end{align}
	\end{subequations}
	where
	\begin{equation*}
	f_0^{\pm}(x,t)\coloneqq \lim_{\delta\searrow 0}f_0(x\pm \delta\boldsymbol{\nu},t)\quad\text{for }x\in\Sigma(0,t).
	\end{equation*}
	Moreover, we introduce the notation
	\begin{equation*}
	[f]_H^T\coloneqq \lim_{\delta\searrow 0}f(x+ \delta\boldsymbol{\nu},t)-\lim_{\delta\searrow 0}f(x- \delta\boldsymbol{\nu},t)\quad\text{for }x\in\Sigma(0,t)
	\end{equation*}
	to denote the jump of a quantity $f$ across the interface. 
	\subsubsection{Inner Expansion to leading order}\mbox{\hfill}\\ [1ex]
	\textbf{Step 1:} From $\eqref{3_basic_equation_1d}_I^{-1}$ we obtain
	\begin{equation}
	\label{3_IE_lead_ord_1d_eq1}\del_{zz}\Phi_0-\psi'(\Phi_0)=0.
	\end{equation}
	Since $\Phi_0(s,z=0,t)=0$ we can choose $\Phi_0$ independent of $s$ and $t$, hence, $\Phi_0$ solves
	\begin{equation}
	\label{3_IE_lead_ord_1d_eq2}\Phi_0''(z)-\psi'(\Phi_0(z)) = 0,\quad \Phi_0(0) = 0,\quad \Phi_0(\pm \infty) = \pm 1,
	\end{equation}
	where we used \eqref{3_matching_cond_1a}. The unique solution of \eqref{3_IE_lead_ord_1d_eq2} is given by
	\begin{equation*}
	\Phi_0(z) = \tanh\left(\frac{z}{\sqrt{2}}\right).
	\end{equation*}
	This solution has the property of equipartition of energy
	\begin{equation}
	\label{3_IE_lead_ord_1d_eq4}\frac{1}{2}|\Phi_0'(z)|^2 = \psi(\Phi_0(z))\quad \forall\,|z|<\infty.
	\end{equation}
	\textbf{Step 2:}
	From $\eqref{3_basic_equation_1a}_I^{-1}$ we obtain (using \eqref{3_new_coord_ident_1d})
	\begin{equation}
	\label{3_IE_lead_ord_1a_eq1}\del_z\mathbf{V}_0\cdot\boldsymbol{\nu} = 0.
	\end{equation}
	Due to $\del_z\boldsymbol{\nu} = \mathbf 0$ this implies
	\begin{equation}
	\label{3_IE_lead_ord_1a_eq1a}\del_z(\mathbf{V}_0\cdot\boldsymbol{\nu}) = 0.
	\end{equation}
	Integrating this identity gives
	\begin{equation*}
	0=\int_{-\infty}^{\infty}\del_z(\mathbf{V}_0\cdot\boldsymbol{\nu})\dz = [\mathbf{V}_0\cdot\boldsymbol{\nu}]_{-\infty}^{\infty}.
	\end{equation*}
	Hence, the matching condition \eqref{3_matching_cond_1a} yields
	\begin{equation}
	\label{3_IE_lead_ord_1a_eq2}[\v_0]_H^T\cdot\boldsymbol{\nu}\coloneqq \v_0^+\cdot\boldsymbol{\nu}-\v_0^-\cdot\boldsymbol{\nu} = 0.
	\end{equation}
	\textbf{Step 3:}
	We now analyse \eqref{3_basic_equation_1c}. The terms $\bar{\rho}_2^{-1}\Gamma_2$ and $\bar{\rho}_1^{-1}\Gamma_1$ do not contribute to leading order. We distinguish again the three cases for the mobilities:\\ [1ex]
	\textit{Case (i)} ($m(\varphi)=m_0$): Using \eqref{3_new_coord_ident},  from $\eqref{3_basic_equation_1c}_I^{-2}$ we get
	\begin{equation*}
	m_0\del_{zz}\Xi_0 = 0.
	\end{equation*}
	Upon integrating and using the matching condition \eqref{3_matching_cond_1b} we obtain
	\begin{equation*}
	\del_{z}\Xi_0 = 0\quad \forall\, |z|<\infty.
	\end{equation*}
	Integrating again from $-\infty$ to $\infty$ and using the matching condition \eqref{3_matching_cond_1a}, yields
	\begin{equation*}
	[\mu_0]_H^T = 0.
	\end{equation*}
	\textit{Case (ii)} ($m(\varphi)=\epsilon m_0$): Using \eqref{3_new_coord_ident}  we obtain from $\eqref{3_basic_equation_1c}_I^{-1}$ that
	\begin{equation}
	\label{3_IE_high_ord_1c_eq4}-\mathcal{V}\Phi_0' + \del_z(\Phi_0\mathbf{V}_0)\cdot\boldsymbol{\nu} = \del_z(m_0\del_z\Xi_0).
	\end{equation}
	Integrating this identity and using $\del_z\mathcal{V}=0$, $\del_z\boldsymbol{\nu}=\mathbf{0}$ in conjunction with \eqref{3_IE_lead_ord_1a_eq1a} and \eqref{3_matching_cond_1b} gives
	\begin{equation*}
	2(-\mathcal{V}+\v_0\cdot\boldsymbol{\nu}) = 0.
	\end{equation*}
	In particular, we obtain from \eqref{3_IE_lead_ord_1a_eq1a}-\eqref{3_IE_lead_ord_1a_eq2} and \eqref{3_IE_high_ord_1c_eq4} that
	\begin{equation*}
	m_0\del_{zz}\Xi_0 = (-\mathcal{V}+\v_0\cdot\boldsymbol{\nu})\Phi_0' = 0,
	\end{equation*}
	which together with the matching condition \eqref{3_matching_cond_1b} implies that $\del_z\Xi_0 = 0$ for all $|z|<\infty$.
	Hence, we obtain that $\Xi_0$ is independent of $z$.\\ [1ex]
	\textit{Case (iii)} $\big(m(\varphi)=\tfrac{m_1}{2}(1+\varphi)^2\big)$: With similar arguments as above we obtain from $\eqref{3_basic_equation_1c}_I^{-2}$ that
	\begin{equation*}
	\tfrac{m_1}{2}\del_{z}((1+\Phi_0)^2\,\del_{z}\Xi_0)=0.
	\end{equation*}
	Integrating this inequality in time from $-\infty$ to $z$ with $|z|<\infty$ and using the matching condition \eqref{3_matching_cond_1b} gives
	\begin{equation*}
	\tfrac{m_1}{2}(1+\Phi_0)^2\del_{z}\Xi_0(s,z,t) = 0\quad \forall\, |z|<\infty.
	\end{equation*}
	Since $|\Phi_0(z)|<1$ for $|z|<\infty$, this implies that 
	\begin{equation*}
	\del_{z}\Xi_0(s,z,t) = 0\quad \forall\,|z|<\infty,
	\end{equation*}
	and therefore $\Xi_0$ is independent of $z$. 
	\\[1ex]
	\textbf{Step 4:} Using $\del_{z}\boldsymbol{\nu}=\mathbf{0}$ and applying similar calculations as for \eqref{3_basic_equation_1c}, from $\eqref{3_basic_equation_1e}_I^{-2}$  we obtain
	\begin{equation*}
	\del_z(n(\Phi_0)\chi_{\sigma}\del_z C_0)-\del_z(n(\Phi_0)\chi_{\varphi}\del_z \Phi_0) = 0.
	\end{equation*}
	Integrating this identity from $-\infty$ to $z$ with $|z|<\infty$ and using \eqref{3_matching_cond_1b} yields
	\begin{equation*}
	n(\Phi_0)(\chi_{\sigma}\del_z C_0-\chi_{\varphi}\Phi_0'(z))=0\quad\forall\,|z|<\infty.
	\end{equation*}
	Since $n(\Phi_0)>0$, this means
	\begin{equation}
	\label{3_IE_lead_ord_1e_eq8}\chi_{\sigma}\del_z C_0(s,z,t)=\chi_{\varphi}\Phi_0'(z)\quad\forall\,|z|<\infty.
	\end{equation}
	Upon integrating and using \eqref{3_matching_cond_1a} we see that
	\begin{equation*}
	[\sigma_0]_H^T=[C_0(s,z,t)]_{-\infty}^{+\infty} = \int_{-\infty}^{\infty}\del_z C_0(s,z,t)\dz = \frac{\chi_{\varphi}}{\chi_{\sigma}} \int_{-\infty}^{\infty}\Phi_0'(z)\dz = 2\frac{\chi_{\varphi}}{\chi_{\sigma}}.
	\end{equation*}
	\textbf{Step 5:} Finally, we analyse \eqref{3_basic_equation_1b} and we define $\mathcal{E}(\mathbf{A}) = \frac{1}{2}(\mathbf{A}+\mathbf{A}^{\intercal})$ for a square matrix $\mathbf{A}$. Using \eqref{3_new_coord_ident_1b}, \eqref{3_new_coord_ident_1e} and \eqref{3_IE_lead_ord_1a_eq1}, with similar arguments as in \cite{AbelsGarckeGrun} we obtain from $\eqref{3_basic_equation_1b}_I^{-2}$ that
	\begin{equation}
	\label{3_IE_lead_ord_1b_eq1_3}\del_z (2\eta(\Phi_0)\mathcal{E}(\del_z\mathbf{V}_0\otimes\boldsymbol{\nu})\boldsymbol{\nu}) = \mathbf 0.
	\end{equation}
	Due to \eqref{3_IE_lead_ord_1a_eq1a} we have
	\begin{equation*}
	(\boldsymbol{\nu}\otimes \del_z\mathbf{V}_0)\boldsymbol{\nu} = (\del_z\mathbf{V}_0\cdot\boldsymbol{\nu})\boldsymbol{\nu} = \mathbf 0.
	\end{equation*}
	Together with \eqref{3_IE_lead_ord_1b_eq1_3} and the identity $(\del_z\mathbf{V}_0\otimes \boldsymbol{\nu})\boldsymbol{\nu} = \del_z\mathbf{V}_0$, this implies
	\begin{equation*}
	\del_{z}(\eta(\Phi_0)\del_z\mathbf{V}_0) = \mathbf 0.
	\end{equation*}
	Integrating from $-\infty$ to $z$ with $|z|<\infty$, using the matching condition \eqref{3_matching_cond_1b} and the positivity of $\eta(\cdot)$, this gives
	\begin{equation}
	\label{3_IE_lead_ord_1b_eq1_5}\del_z\mathbf{V}_0 = \mathbf 0\quad \forall \,|z|<\infty.
	\end{equation}
	Once more integrating and using the matching condition \eqref{3_matching_cond_1a} yields
	\begin{equation}
	\label{3_IE_lead_ord_1b_eq1_6}[\v_0]_H^T = \mathbf 0.
	\end{equation}
	\subsubsection{Inner Expansion to higher order}
	We will now expand the equations in the inner regions to the next highest order.\\[1ex]
	\textbf{Step 1:} From $\eqref{3_basic_equation_1d}_I^0$, we obtain
	\begin{equation*}
	\beta\Phi_1\psi''(\Phi_0) + \beta\kappa\Phi_0' - \beta\del_{zz}\Phi_1 - \chi_{\varphi}C_0 = \Xi_0.
	\end{equation*}
	Multiplying by $\Phi_0'$ and integrating from $-\infty$ to $+\infty$ yields
	\begin{equation}
	\label{3_IE_high_ord_1d_eq2}\int_{-\infty}^{\infty}\Xi_0(s,t)\Phi_0'(z)\dz = \int_{-\infty}^{\infty}\beta(\psi'(\Phi_0))'\Phi_1 - \beta\del_{zz}\Phi_1\Phi_0' + \beta\kappa|\Phi_0'|^2-\chi_{\varphi}C_0\Phi_0'\dz.
	\end{equation}
	Using \eqref{3_matching_cond_1a}-\eqref{3_matching_cond_1b}, \eqref{3_IE_lead_ord_1d_eq1} and $\psi'(\pm 1)=0$, integration by parts gives
	\begin{align}
	\nonumber\int_{-\infty}^{\infty}(\psi'(\Phi_0))'\Phi_1-\del_{zz}\Phi_1\Phi_0'\d z &= [\psi'(\Phi_0)\Phi_1-\del_z\Phi_1\Phi_0']_{-\infty}^{+\infty} \\
	\label{3_IE_high_ord_1d_eq3}&\quad- \int_{-\infty}^{\infty}\del_z\Phi_1(\psi'(\Phi_0)-\Phi_0'')\dz=0.
	\end{align}
	Recalling that $\Xi_0$ is independent of $z$ and applying the matching condition \eqref{3_matching_cond_1a} we have
	\begin{equation}
	\label{3_IE_high_ord_1d_eq4}\int_{-\infty}^{+\infty}\Xi_0(s,t)\Phi_0'(z)\dz = 2\mu_0.
	\end{equation}
	By the equipartition of energy \eqref{3_IE_lead_ord_1d_eq4} we compute
	\begin{align*}
	\int_{-\infty}^{\infty}|\Phi_0'(z)|^2\dz &= \int_{-\infty}^{\infty}|\Phi_0'(z)|\sqrt{2\psi(\Phi_0(z))}\dz =\int_{-1}^{1}\sqrt{2\psi(y)}\dy \\
	&= \frac{1}{\sqrt{2}}\int_{-1}^{1}(1-y^2)\dy = \frac{2\sqrt{2}}{2}\eqqcolon \tau,
	\end{align*}
	and obtain
	\begin{equation}
	\label{3_IE_high_ord_1d_eq5}\int_{-\infty}^{+\infty}\beta\kappa|\Phi_0'(z)|^2 \dz =  \beta\kappa\tau.
	\end{equation}
	Finally, by \eqref{3_IE_lead_ord_1e_eq8} we obtain
	\begin{align}
	\nonumber\int_{-\infty}^{+\infty}\chi_{\varphi}C_0\Phi_0'(z)\dz = \chi_{\sigma}\int_{-\infty}^{+\infty}\del_z C_0(s,z,t)C_0(s,z,t)\dz &= \frac{\chi_{\sigma}}{2}\int_{-\infty}^{+\infty}\del_z(|C_0|^2)\dz\\
	\label{3_IE_high_ord_1d_eq6} & = \frac{\chi_{\sigma}}{2}[|\sigma_0|^2]_H^T.
	\end{align}
	Collecting \eqref{3_IE_high_ord_1d_eq2}-\eqref{3_IE_high_ord_1d_eq6} gives
	\begin{equation}
	\label{3_IE_high_ord_1d_eq7} 2\mu_0 = \beta\kappa\tau - \frac{\chi_{\sigma}}{2}[|\sigma_0|^2]_H^T.
	\end{equation}
	This is a solvability condition for $\Phi_1$, the so-called \textbf{Gibbs--Thomas equation}.\\ [1ex] 
	\textbf{Step 2:}
	With similar arguments as above and using \eqref{3_IE_lead_ord_1e_eq8}, equation $\eqref{3_basic_equation_1e}_I^{-1}$ gives
	\begin{equation*}
	(-\mathcal{V}+\mathbf{V}_0\cdot\boldsymbol{\nu})\del_{z}C_0 = \del_z(n(\Phi_0)(\chi_{\sigma}\del_z C_1-\chi_{\varphi}\del_z\Phi_1)).
	\end{equation*}
	Employing the matching condition \eqref{3_matching_cond_1c} and $\grad\varphi_0=\mathbf{0}$ in the bulk regions together with $\del_z\mathcal{V}=0$ and \eqref{3_IE_lead_ord_1a_eq1a}, this yields
	\begin{align*}
	(-\mathcal{V}+\v_0\cdot\boldsymbol{\nu})[\sigma_0]_H^T &= \int_{-\infty}^{\infty}(-\mathcal{V}+\mathbf{V}_0\cdot\boldsymbol{\nu})\del_{z}C_0\dz\\
	& = \int_{-\infty}^{+\infty}\del_z(n(\Phi_0)(\chi_{\sigma}\del_z C_1-\chi_{\varphi}\del_z\Phi_1))\dz 
	= \chi_{\sigma}[n(\varphi_0)\grad\sigma_0]_H^T\cdot\boldsymbol{\nu}.
	\end{align*}
	\textbf{Step 3:} Similar as in \cite{AbelsGarckeGrun} we analyse \eqref{3_basic_equation_1c} only for the mobilities \eqref{3_mobility}\textnormal{(i)} and \textnormal{(iii)} since the case \eqref{3_mobility}\textnormal{(ii)} is rescaled and therefore does not contribute to the sharp interface limit.\\ [1ex]
	\textit{Case (i)} ($m(\varphi)=m_0$): Using $\del_{z}\Xi_0 = 0$ and \eqref{3_IE_lead_ord_1a_eq1}, from $\eqref{3_basic_equation_1c}_I^{-1}$ we obtain
	\begin{equation*}
	(-\mathcal{V}+\mathbf{V}_0\cdot\boldsymbol{\nu})\Phi_0' = m_0\del_{zz}\Xi_1.
	\end{equation*}
	Integrating with respect to $z$ from $-\infty$ to $\infty$, using \eqref{3_IE_lead_ord_1a_eq1a}-\eqref{3_IE_lead_ord_1a_eq2} and the matching condition \eqref{3_matching_cond_1c}, this yields
	\begin{equation}\label{3_IE_high_ord_1e_eq5}
	2(-\mathcal{V}+\v_0\cdot\boldsymbol{\nu}) = m_0[\grad\mu_0]_H^T\cdot\boldsymbol{\nu}.
	\end{equation}
	\textit{Case (iii)} $\big(m(\varphi)=m_1(1+\varphi)^2\big)$: With similar arguments as above we obtain
	\begin{equation*}
	(-\mathcal{V}+\mathbf{V}_0\cdot\boldsymbol{\nu})\Phi_0' = \tfrac{m_1}{2}\del_{z}\left( (1+\Phi_0)^2\del_{z}\Xi_1\right).
	\end{equation*}
	Using the matching conditions \eqref{3_matching_cond_1a}, \eqref{3_matching_cond_1c} and the same arguments as for \eqref{3_IE_high_ord_1e_eq5}, this entails
	\begin{equation*}
	(-\mathcal{V}+\v_0\cdot\boldsymbol{\nu}) = m_1\grad\mu_0^T\cdot\boldsymbol{\nu}.
	\end{equation*}
	\textbf{Step 4:} Finally, we consider the momentum balance equation \eqref{3_basic_equation_1b} at order $\epsilon^{-1}$. Using \eqref{3_new_coord_ident} and \eqref{3_IE_lead_ord_1b_eq1_5}, with similar arguments as above we obtain from $\eqref{3_basic_equation_1b}_I^{-1}$
	\begin{align}
	\nonumber &-\del_z\big(2\eta(\Phi_0)\mathcal{E}(\del_z\mathbf{V}_1\otimes \boldsymbol{\nu})\boldsymbol{\nu}+2\eta(\Phi_0)\mathcal{E}(\grad_{\Sigma(0)}\mathbf{V}_0)\boldsymbol{\nu}\big)\\
	\nonumber&\qquad -\del_z\big(\lambda(\Phi_0)(\del_{z}\mathbf{V}_1\cdot\boldsymbol{\nu} + \divergence_{\Sigma(0)}\mathbf{V}_0)\boldsymbol{\nu} -  P_0\boldsymbol{\nu}\big) \\
	\label{3_IE_high_ord_1b_eq4}&\quad= (\Xi_0 + \chi_{\varphi}C_0)\Phi_0'\boldsymbol{\nu}.
	\end{align}
	Since matching requires $\lim_{z\to\pm \infty}\del_z\mathbf{V}_1(z) = (\grad\v_0^{\pm})\boldsymbol{\nu}$, we conclude
	\begin{alignat*}{3}
	(\del_z\mathbf{V}_1\otimes\boldsymbol{\nu} + \grad_{\Sigma(0)}\mathbf{V}_0) &\to \grad_x\v_0&&\quad\text{for}\quad z\to\pm \infty,\\
	(\del_{z}\mathbf{V}_1\cdot\boldsymbol{\nu} + \divergence_{\Sigma(0)}\mathbf{V}_0) &\to \divergence_x\v_0&&\quad\text{for}\quad z\to\pm \infty.
	\end{alignat*}
	Integrating \eqref{3_IE_high_ord_1b_eq4} with respect to $z$ from $-\infty$ to $+\infty$ and using \eqref{3_matching_cond_1a}, this implies
	\begin{align*}
	& - [2\eta(\varphi_0)\mathcal{E}(\grad_x\v_0) +\lambda(\varphi_0)\divergence(\v_0)\I -p_0\I]_H^T\boldsymbol{\nu} \nonumber \\ & \qquad = \int_{-\infty}^{+\infty}(\Xi_0(s,t) + \chi_{\varphi}C_0(s,z,t))\Phi_0'(z)\boldsymbol{\nu}\dz.
	\end{align*}
	Together with \eqref{3_IE_high_ord_1d_eq4} and \eqref{3_IE_high_ord_1d_eq6}-\eqref{3_IE_high_ord_1d_eq7}, we end up at
	\begin{equation*}
	[\T(\varphi_0,\v_0,p_0)]_H^T\boldsymbol{\nu} = -\beta\kappa\tau\boldsymbol{\nu}.
	\end{equation*}
	\subsection{Equations of the formal sharp interface limit}
	For the reader's convenience, we summarise the sharp interface models for the different mobilities: \\ [1ex]
	\textit{Case (i)} ($m(\varphi)=m_0$) The equations in the bulk are given by
	\begin{alignat*}{3}
	-\divergence(\T(\varphi_0,\v_0,p_0)) +\nu(\varphi_0)\v_0&= 0&&\qquad\text{in }\Omega_T\cup\Omega_H, \\
	\divergence (\v_0^T) &= \bar{\rho}_2^{-1}\Gamma_2(1,\sigma_0^T) + \bar{\rho}_1^{-1}\Gamma_1(1,\sigma_0^T)&&\qquad\text{in }\Omega_T,\\
	\divergence (\v_0^H) &= \bar{\rho}_2^{-1}\Gamma_2(-1,\sigma_0^H) + \bar{\rho}_1^{-1}\Gamma_1(-1,\sigma_0^H)&&\qquad\text{in }\Omega_H,\\
	-m_0\laplace \mu_0^T &= -2\bar{\rho}_1^{-1}\Gamma_1(1,\sigma_0^T) &&\qquad\text{in }\Omega_T,\\
	-m_0\laplace \mu_0^H &= 2\bar{\rho}_2^{-1}\Gamma_2(-1,\sigma_0^H)&&\qquad\text{in }\Omega_H,\\
	\delt \sigma_0^T + \divergence(\sigma_0^T\v_0^T) &= \divergence (n(1)\chi_{\sigma}\grad \sigma_0^T) - \Gamma_{\sigma}(1,\sigma_0^T)&&\qquad\text{in }\Omega_T,\\
	\delt \sigma_0^H + \divergence(\sigma_0^H\v_0^H) &= \divergence (n(-1)\chi_{\sigma}\grad \sigma_0^H) - \Gamma_{\sigma}(-1,\sigma_0^H)&&\qquad\text{in }\Omega_H.
	\end{alignat*}
	Furthermore, on $\Sigma(0)$ we have the free boundary conditions 
	\begin{align*}
	&[\v_0]_H^T = \mathbf{0},\qquad [\mu_0]_H^T = 0,\qquad [\sigma_0]_H^T = 2\tfrac{\chi_{\varphi}}{\chi_{\sigma}},\\
	&2\mu_0 = \beta\kappa\tau -\tfrac{\chi_{\sigma}}{2}[|\sigma_0|^2]_H^T,\qquad (-\mathcal{V}+\v_0\cdot\boldsymbol{\nu})[\sigma_0]_H^T= [n(\varphi_0)\grad\sigma_0]_H^T\cdot\boldsymbol{\nu},\\
	&2(-\mathcal{V}+\v_0\cdot\boldsymbol{\nu}) = m_0[\grad \mu_0]_H^T\cdot\boldsymbol{\nu},\qquad [\T(\varphi_0,\v_0,p_0)]_H^T\boldsymbol{\nu}  = -\beta\kappa\tau\boldsymbol{\nu}.
	\end{align*}
	\textit{Case (ii)} ($m(\varphi)=\epsilon m_0$) The equations in the bulk are given by
	\begin{alignat*}{3}
	-\divergence(\T(\varphi_0,\v_0,p_0)) +\nu(\varphi_0)\v_0&= 0&&\qquad\text{in }\Omega_T\cup\Omega_H, \\
	\divergence (\v_0^T) &= \bar{\rho}_2^{-1}\Gamma_2(1,\sigma_0^T) &&\qquad\text{in }\Omega_T,\\
	\divergence (\v_0^H) &=   \bar{\rho}_1^{-1}\Gamma_1(-1,\sigma_0^H)&&\qquad\text{in }\Omega_H,\\
	\delt \sigma_0^T + \divergence(\sigma_0^T\v_0^T) &= \divergence (n(1)\chi_{\sigma}\grad \sigma_0^T) - \Gamma_{\sigma}(1,\sigma_0^T)&&\qquad\text{in }\Omega_T,\\
	\delt \sigma_0^H + \divergence(\sigma_0^H\v_0^H) &= \divergence (n(-1)\chi_{\sigma}\grad \sigma_0^H) - \Gamma_{\sigma}(-1,\sigma_0^H)&&\qquad\text{in }\Omega_H.
	\end{alignat*}
	Furthermore, on $\Sigma(0)$ we have the free boundary conditions
	\begin{align*}
	&[\v_0]_H^T = \mathbf{0},\qquad [\sigma_0]_H^T = 2\tfrac{\chi_{\varphi}}{\chi_{\sigma}},\qquad 0 =  [n(\varphi_0)\grad\sigma_0]_H^T\cdot\boldsymbol{\nu},\\
	&\mathcal{V}=\v_0\cdot\boldsymbol{\nu} ,\qquad [\T(\varphi_0,\v_0,p_0)]_H^T\boldsymbol{\nu}  = -\beta\kappa\tau\boldsymbol{\nu}.
	\end{align*}
	\textit{Case (iii)} $\big(m(\varphi)=m_1(1+\varphi)^2\big)$ The equations in the bulk are given by
	\begin{alignat*}{3}
	-\divergence(\T(\varphi_0,\v_0,p_0)) +\nu(\varphi_0)\v_0&= 0&&\qquad\text{in }\Omega_T\cup\Omega_H, \\
	\divergence (\v_0^T) &= \bar{\rho}_2^{-1}\Gamma_2(1,\sigma_0^T) + \bar{\rho}_1^{-1}\Gamma_1(1,\sigma_0^T)&&\qquad\text{in }\Omega_T,\\
	\divergence (\v_0^H) &=  \bar{\rho}_1^{-1}\Gamma_1(-1,\sigma_0^H)&&\qquad\text{in }\Omega_H,\\
	-m_1\laplace \mu_0^T &=-\bar{\rho}_1^{-1}\Gamma_1(1,\sigma_0^T) &&\qquad\text{in }\Omega_T,\\
	\delt \sigma_0^T + \divergence(\sigma_0^T\v_0^T) &= \divergence (n(1)\chi_{\sigma}\grad \sigma_0^T) - \Gamma_{\sigma}(1,\sigma_0^T)&&\qquad\text{in }\Omega_T,\\
	\delt \sigma_0^H + \divergence(\sigma_0^H\v_0^H) &= \divergence (n(-1)\chi_{\sigma}\grad \sigma_0^H) - \Gamma_{\sigma}(-1,\sigma_0^H)&&\qquad\text{in }\Omega_H.
	\end{alignat*}
	Furthermore, on $\Sigma(0)$ we have the free boundary conditions
	\begin{align*}
	&[\v_0]_H^T = \mathbf{0},\qquad [\sigma_0]_H^T = 2\tfrac{\chi_{\varphi}}{\chi_{\sigma}}, \qquad 2\mu_0 = \beta\kappa\tau -\tfrac{\chi_{\sigma}}{2}[|\sigma_0|^2]_H^T,\\
	& (-\mathcal{V}+\v_0\cdot\boldsymbol{\nu})[\sigma_0]_H^T= [n(\varphi_0)\grad\sigma_0]_H^T\cdot\boldsymbol{\nu},\qquad (-\mathcal{V}+\v_0\cdot\boldsymbol{\nu}) = m_1\grad \mu_0^T\cdot\boldsymbol{\nu},\\
	& [\T(\varphi_0,\v_0,p_0)]_H^T\boldsymbol{\nu}  = -\beta\kappa\tau\boldsymbol{\nu}.
	\end{align*}
	\subsection{Specific sharp interface models}
	\subsubsection{The limit of vanishing active transport, Darcy's law and Stokes' flow}
	We consider \eqref{3_basic_equation_1a}-\eqref{3_basic_equation_1e} with quasi-static nutrients and the mobility \eqref{3_mobility}\textnormal{(ii)} along with constant viscosities and permeability. Moreover, we decouple chemotaxis and active transport according to \eqref{3_DECOUPLE_AT_CHEMO_1}, and we set
	\begin{equation*}
	\mathcal{D}(\varphi) = \frac{1+\varphi}{2} +  \mathcal{D}\frac{1-\varphi}{2}
	\end{equation*}
	for a constant $\mathcal{D}>0$. Moreover, we choose
	\begin{equation*}
	\Gamma_1\equiv 0,\qquad \Gamma_2(\varphi,\sigma)=\frac{\bar{\rho}_2}{2}\left(\frac{1}{\bar{\rho}_2}-\frac{1}{\bar{\rho}_1}\right)(\mathcal{P}\sigma-\mathcal{A})(1+\varphi),\qquad \Gamma_{\sigma}(\varphi,\sigma)=\frac{\mathcal{C}}{2}\sigma(1+\varphi).
	\end{equation*}
	This gives the following system of equations
	\begin{align*}
	\divergence(\v) &= \frac{\alpha}{2}(\mathcal{P}\sigma-\mathcal{A})(1+\varphi),\\
	-\divergence(\T(\v,p)) +\nu\v &= (\mu +\chi_{\varphi}\sigma)\grad\varphi,\\
	\delt\varphi + \grad\varphi\cdot\v &= \divergence(\epsilon m_0 \grad\mu)+\frac{\alpha}{2}(\mathcal{P}\sigma-\mathcal{A})(1-\varphi^2),\\
	\mu &= \tfrac{\beta}{\epsilon}\psi'(\varphi)-\beta\epsilon\Delta\varphi-\chi_{\varphi}\sigma,\\
	0&= \divergence(\mathcal{D}(\varphi)\grad\sigma)-\chi\divergence(\mathcal{D}(\varphi)\grad\varphi)-\mathcal{C}\sigma(1+\varphi),
	\end{align*}
	where $\T(\v,p)=2\eta\D\v + \lambda\divergence(\v)\I-p\I$. With slightly different arguments as above (see also \cite{GarckeLamSitkaStyles}) and sending $\chi\to 0$, we obtain 
	\begin{subequations}\label{3_SHARPIF_NO_AT_1}
		\begin{align}
		\label{3_SHARPIF_NO_AT_1a}-\divergence(\T(\v_0,p_0)) +\nu\v_0&= 0\qquad\text{in }\Omega_T\cup\Omega_H, \\
		\label{3_SHARPIF_NO_AT_1b}\divergence (\v_0) &=\begin{cases}
		\alpha(\mathcal{P}\sigma_0^T-\mathcal{A} )&\text{in }\Omega_T,\\
		0&\text{in }\Omega_H,
		\end{cases}\\
		\label{3_SHARPIF_NO_AT_1c}\laplace\sigma_0 &= \begin{cases}
		\mathcal{C}\sigma_0&\text{in }\Omega_T,\\
		0&\text{in }\Omega_H,
		\end{cases}
		\end{align}
		and the free boundary conditions on $\Sigma(0)$ are given by
		\begin{equation}\label{3_SHARPIF_NO_AT_1d}
		\begin{aligned}
		&[\v_0]_H^T = \mathbf{0},\qquad [\sigma_0]_H^T = 0,\qquad \grad\sigma_0^T\cdot\boldsymbol{\nu} =  \mathcal{D}\grad\sigma_0^H\cdot\boldsymbol{\nu},\\
		&\mathcal{V}=\v_0\cdot\boldsymbol{\nu},\qquad [\T(\v_0,p_0)]_H^T\boldsymbol{\nu}  = -\beta\kappa\tau\boldsymbol{\nu}.
		\end{aligned}
		\end{equation}
	\end{subequations}
	This model is a special case of the two-phase free boundary problem in \cite{ZhengWiseCristini}, where numerical simulations for \eqref{3_SHARPIF_NO_AT_1} are presented. Similar models have been studied in \cite{CristiniEtAl2}. For a one-phase model with Brinkman's law for the velocity we refer to \cite{PhamFrieboesCristiniLowengrub}.\newline
	Sending the viscosities to $0$ in \eqref{3_SHARPIF_NO_AT_1}, we can express the velocity in terms of the pressure and we obtain the following Darcy-type model 
	\begin{align*}
	-\laplace p_0 &=\begin{cases}
	\nu\,\alpha(\mathcal{P}\sigma_0^T-\mathcal{A} )&\text{in }\Omega_T,\\
	0&\text{in }\Omega_H,
	\end{cases}\\
	\laplace\sigma_0 &= \begin{cases}
	\mathcal{C}\sigma_0&\text{in }\Omega_T,\\
	0&\text{in }\Omega_H,
	\end{cases}
	\end{align*}
	where the free boundary conditions on $\Sigma(0)$ are given by
	\begin{equation*}
	[\sigma_0]_H^T = 0,\ \grad\sigma_0^T\cdot\boldsymbol{\nu}=\mathcal{D}\grad\sigma_0^H\cdot\boldsymbol{\nu},\ \tfrac{1}{\nu}[\grad p_0]_H^T\cdot\boldsymbol{\nu} = 0,\ \mathcal{V}=-\tfrac{1}{\nu}\grad p_0\cdot\boldsymbol{\nu},\ [p_0]_H^T  = -\beta\kappa\tau.
	\end{equation*}
	Similar models have been studied in, e.\,g., \cite{CristiniLowengrubNie,Greenspan,LowengrubFriboesJin,MacklinLowengrub}. We remark that the continuity condition for $\v_0$ across the interface (see \eqref{3_IE_lead_ord_1b_eq1_6}) is based on the positivity of the shear viscosity. \\[1ex]
	Sending the permeability to zero in \eqref{3_SHARPIF_NO_AT_1}, i.\,e., $\nu\to 0$, we obtain a Stokes model given by
	\begin{align*}
	-\divergence(2\eta\D\v_0 + \lambda\divergence(\v_0)\I-p_0\I) &= 0\qquad\text{in }\Omega_T\cup\Omega_H, \\
	\divergence (\v_0) &=\begin{cases}
	\alpha(\mathcal{P}\sigma_0^T-\mathcal{A} )&\text{in }\Omega_T,\\
	0&\text{in }\Omega_H,
	\end{cases}\\
	\laplace\sigma_0 &= \begin{cases}
	\mathcal{C}\sigma_0&\text{in }\Omega_T,\\
	0&\text{in }\Omega_H,
	\end{cases}
	\end{align*}
	and the free boundary conditions on $\Sigma(0)$ are given by
	\begin{equation*}
	\begin{aligned}
	&[\v_0]_H^T = \mathbf{0},\qquad [\sigma_0]_H^T = 0,\qquad \grad\sigma_0^T\cdot\boldsymbol{\nu} =\mathcal{D}\grad\sigma_0^H\cdot\boldsymbol{\nu},\\
	&\mathcal{V}=\v_0\cdot\boldsymbol{\nu},\qquad [2\eta\D\v_0 + \lambda\divergence(\v_0)\I-p_0\I]_H^T\boldsymbol{\nu}  = -\beta\kappa\tau\boldsymbol{\nu}.
	\end{aligned}
	\end{equation*}
	For similar models, we refer to \cite{FranksKing2,FranksKing,Friedman,Friedman5,Friedman2,FriedmanHu,WuCui}.\\ [1ex]
	We remark that a similar asymptotic analysis can be performed for the double obstacle potential
	\begin{equation}
	\label{3_def_double_obstacle}\psi(\varphi) \coloneqq \frac{1}{2}(1-\varphi^2)+I_{[-1,1]}(\varphi),\quad I_{[-1,1]}(\varphi)=
	\begin{cases}
	0&\text{if }|\varphi|\leq 1,\\
	+\infty&\text{elsewhere}.
	\end{cases}
	\end{equation}
	To do so one combines the arguments above with the asymptotic analysis in \cite{GarckeLamSitkaStyles}. We refer to \cite{Ebenbeck} for details.
	
	\section{Analytical results}
	
	Our aim is to analyse the following variant of \eqref{3_MEQ}
	\begin{subequations}\label{7_state_eq}
		\begin{alignat}{3}
		\label{7_state_eq_1a}\divergence(\v)  &= 0&&\qquad\text{ in }Q,\\
		\label{7_state_eq_1b}-\divergence(2\eta\D \v) + \nu\v -\grad p  &= -\epsilon\,\divergence(\grad\varphi\otimes\grad\varphi)&&\qquad\text{ in }Q,\\
		\label{7_state_eq_1c}  \del_t\varphi + \divergence(\varphi\v )&= \divergence (m(\varphi)\grad \mu)+g(\varphi,\sigma)h(\varphi)&&\qquad\text{ in }Q,\\
		\label{7_state_eq_1d} \mu&=-\epsilon\Delta \varphi +  \epsilon^{-1}\psi'(\varphi) -\chi_{\varphi}\sigma&&\qquad\text{ in }Q,\\
		\label{7_state_eq_1e}  \del_t\sigma + \divergence(\sigma\v ) &= \divergence(\chi_{\sigma}\grad\sigma-\chi_{\varphi}\grad\varphi)-f(\varphi,\sigma)h(\varphi)&&\qquad\text{ in }Q,
		\end{alignat}
	\end{subequations}
	with boundary and initial conditions of the form
	\begin{subequations}\label{7_boundary_cond}
		\begin{alignat}{3}
		\label{7_boundary_cond_1a}  \grad\varphi\cdot\n&=  \grad\mu\cdot\n = \grad\sigma\cdot\n=0 &&\qquad\text{ on }\oldSigma,\\
		\label{7_boundary_cond_1b} \v&= \mathbf{0} &&\qquad\text{ on }\oldSigma,\\
		\label{7_boundary_cond_1c} \varphi(0)&=\varphi_0,\quad\sigma(0)=\sigma_0&&\qquad\text{  in }\Omega. 
		\end{alignat}
	\end{subequations}
	The terms $h(\varphi)g(\varphi,\sigma)$ and $h(\varphi)f(\varphi,\sigma)$ act as source terms. 
	\begin{remark}
		\begin{enumerate}
			\item[(i)] We will consider a source term that satisfies $h(\varphi)=0$ for $\varphi\leq -1$ which is consistent with a mobility satisfying $m(-1)=0$ and a potential with a singularity in $\varphi = -1$. In general, it is sufficient to prescribe $h(-1)=0$ since, as discussed above, the degenerate mobility guarantees the bound $\varphi\geq -1$ a.\,e. in $Q$.
			\item[(ii)] Equation \eqref{7_state_eq_1a} holds, e.\,g., in the case of matched pure densities, i.\,e.\ $\bar\rho_1 = \bar\rho_2\eqqcolon \bar\rho$, and assuming no gain or loss of mass locally. Indeed, this gives (see \eqref{3_SOURTE_TERM_1}-\eqref{3_SOURTE_TERM_2})
			\begin{equation*}
			\Gamma_{\varphi}= \left(\frac{1}{\bar\rho_1}+\frac{1}{\bar\rho_2}\right)\Gamma = \frac{2}{\bar\rho}\,\Gamma,\qquad \Gamma_{\v} = \left(\frac{1}{\bar\rho_2}-\frac{1}{\bar\rho_1}\right)\Gamma = 0.
			\end{equation*}
			\item[(iii)] Equations \eqref{7_state_eq_1a} and \eqref{7_boundary_cond_1b} seem to be indispensable for the analysis. Indeed, the Dirichlet condition for $\v$ guarantees that 
			there is no transport 
			across the boundary of $\Omega$ which will be important for a priori estimates. Furthermore, as a consequence of \eqref{7_boundary_cond_1b} we require that $\divergence(\v)$ has zero mean for almost all $t\in (0,T)$. This is not compatible with a solution dependent source term in \eqref{7_state_eq_1a}.
			\item[(iv)] We also allow for $\nu=0$ which corresponds to the case of  Stokes flow.
		\end{enumerate}
	\end{remark}

	\subsection{Construction of approximating solutions}\label{7_Section_1}
	\begin{annahme}\label{7_ASS}
		Throughout Subsection 5.1, we make the following assumptions.
		\begin{enumerate}		
			\item[(i)]The potential $\psi\in C^2(\R)$ satisfies
			\begin{equation}
			\label{7_Assumption_on_psi}|\psi'(t)|\leq C_1(1+|t|),\quad |\psi''(t)|\leq C_2\quad \psi(t)\geq -C_3\quad\forall \,t\in\R
			\end{equation}
			with positive constants $C_1$, $C_2$ and $C_3$.
			\item[(ii)] The initial data satisfy $\varphi_0\in H^1$, $\sigma_0\in L^6$.	
			\item[(iii)] The functions $g,f\colon \R^2 \to\R$ are continuous such that
			\begin{equation}
			\label{7_Assumption_on_fg} |g(\varphi,\sigma)|\leq C_4(1+|\varphi|+|\sigma|),\qquad |f(\varphi,\sigma)|\leq C_5(1+|\varphi|+|\sigma|)\quad\forall\, \varphi,\sigma\in\R
			\end{equation}
			for positive constants $C_4$ and $C_5$. 
			\item[(iv)] The function $h\colon\R \to\R$ is continuous, non-negative and bounded such that 
			\begin{alignat*}{3}
			h(\varphi) &= 0 \quad&&\text{ if }\varphi\leq -1,\\
			C_6(1+\varphi)&\leq h(\varphi)\leq C_7(1+\varphi) \quad&&\text{ if }\varphi\in[-1,1],\\
			h(\varphi) &\leq C_8 \quad&&\text{ if }\varphi>1
			\end{alignat*}
			for positive constants $C_6$, $C_7$, $C_8$, and $C_6\leq C_7$.
			\item[(v)] For $d=2,3$, $\Omega\subset\R^d$ is a bounded domain with $C^3$-boundary.
			\item[(vi)] The constant $\eta>0$ is positive, the constants $\lambda\geq 0$, $\nu\geq 0$ are non-negative.	
		\end{enumerate}
	\end{annahme}
	
	\begin{remark}
		From Assumptions \ref{7_ASS}\textnormal{(iv)}, it follows that $h$ behaves like $(1+\varphi)_{\text{+}}\coloneqq \max(0,1+\varphi)$ near $\varphi = -1$. A typical example is given by
		\begin{equation*}
		h(\varphi)\coloneqq \max\left(0, \min\left(\frac{1}{2}(1+\varphi), 1\right)\right). 
		\end{equation*}
		Furthermore, we observe that
		\begin{equation*}
		h(\varphi)\leq h_{\infty}\quad \forall \, \varphi\in \R,
		\end{equation*}
		where $h_{\infty}\coloneqq \max\{2C_7,C_8\}$. 
	\end{remark}
	In the following we will assume w.\,l.\,o.\,g.\ that $\psi\geq 0$, as we can always add a constant to $\psi$ without changing the equation \eqref{7_state_eq_1d}.
	For $\delta> 0$ we consider the system \eqref{7_state_eq}-\eqref{7_boundary_cond} with \eqref{7_state_eq_1b} replaced by 
	\begin{equation}
	\label{7_state_eq_1b_approx}\delta\delt\v-\divergence(2\eta\D \v) + \nu\v -\grad p  = (\mu+\chi_{\varphi}\sigma)\grad\varphi\qquad\text{in }Q,
	\end{equation}
	and \eqref{7_boundary_cond_1c} replaced by
	\begin{equation}
	\label{7_boundary_cond_1c_approx} \varphi(0)=\varphi_0,\quad\sigma(0)=\sigma_{0,\delta},\quad\v(0) = \mathbf{0}\qquad\text{in }\Omega,
	\end{equation}
	where $\sigma_{0,\delta}\in H_{N}^2$ is the unique solution of
	\begin{equation}
	\label{7_boundary_cond_1d_approx}-\delta\laplace \sigma_{0,\delta} + \sigma_{0,\delta} = \sigma_0\quad\text{in }\Omega,\qquad \grad\sigma_{0,\delta}\cdot\n = 0\quad\text{on }\del\Omega.
	\end{equation}
	
	\begin{remark}
		The modified capillary term on the right hand side of \eqref{7_state_eq_1b_approx} simplifies the a priori estimates, since the convection term in \eqref{7_state_eq_1c} and the term on the right hand side of \eqref{7_state_eq_1b_approx} cancel out within the testing procedure. This is not the case if we use $-\divergence(\epsilon(\grad\varphi\otimes \grad\varphi))$, as we do not have the formula
		\begin{equation*}
		\inn{ -\epsilon(\grad\varphi\otimes \grad\varphi)}{\grad\v} = \inn{(\mu+\chi_{\varphi}\sigma)\grad\varphi}{\v}\quad\forall\,\mathbf{u}\in\mathbf{V}
		\end{equation*}
		on the Galerkin level.
	\end{remark}
	We now prove the following lemma:
	\begin{lemma}[Existence of approximating solutions]\label{7_Theorem_non_deg_mobility_approx}
		Let $m\in C^0(\R)$ with $m_0\leq m(s)\leq M_0$ for all $s\in\R$ with positive constants $m_0$, $M_0$, and let Assumptions \ref{7_ASS} be fulfilled. Then, there exists a quadruplet $(\varphi_{\delta},\mu_{\delta}, \sigma_{\delta}, \v_{\delta})$ with the regularity	
		\begin{align*}
		\varphi_{\delta}&\in H^1((H^1)^*)\cap L^{\infty}(H^1)\cap L^2(H^3),\quad \sigma_{\delta}\in  H^1(L^2)\cap L^{\infty}(H^1)\cap L^2(H^2),\\
		\mu_{\delta}&\in L^4(L^2)\cap L^2(H^1),\quad \v_{\delta}\in H^1(L^{\frac{3}{2}})\cap L^{\infty}(L^2)\cap L^\frac{16}{5}(\mathbf{V})\cap L^{\frac{8}{5}}(\H^2),
		\end{align*}
		recall \eqref{eq:boldV},
		such that the initial conditions and equations \eqref{7_state_eq_1a}, \eqref{7_state_eq_1c}-\eqref{7_state_eq_1e}, \eqref{7_state_eq_1b_approx} and \eqref{7_boundary_cond_1a}-\eqref{7_boundary_cond_1b}, \eqref{7_boundary_cond_1c_approx} are fulfilled in the sense that
		\begin{equation*}
		\varphi_{\delta}(0)=\varphi_0, \quad \sigma_{\delta}(0) = \sigma_{0,\delta},\quad \v_{\delta}(0)=\mathbf{0} \quad \text{a.\,e.\ in } \Omega,
		\end{equation*}
		and
		\begin{subequations}\label{7_weak_formulation_eq_approx}
			\begin{align}
			\label{7_weak_formulation_eq_approx_1a} 0&= \inndual{\del_t\varphi_{\delta}}{\xi}_{H^1} +  \inn{\grad\varphi_{\delta}\cdot\v_{\delta}}{\xi} + \inn{m(\varphi_{\delta})\grad\mu_{\delta}}{\grad\xi}  -\inn{g(\varphi_{\delta},\sigma_{\delta})h(\varphi_{\delta})}{\xi},\\
			\label{7_weak_formulation_eq_approx_1d} 0&=  \inn{\delta\delt\v_{\delta}}{\mathbf{u}} +2\eta\inn{\D \v_{\delta}}{\D\mathbf{u}} + \nu\inn{\v_{\delta}}{\mathbf{u}} - \inn{(\mu_{\delta}+\chi_{\varphi}\sigma_{\delta})\grad\varphi_{\delta}}{\mathbf{u}}
			\end{align}
			for all $\xi \in H^1$, $\mathbf{u}\in \mathbf{V}$, and for a.\,e.\ $t\in (0,T)$, whereas 
			\begin{alignat}{3}
			\label{7_weak_formulation_eq_approx_1b}\hspace{-2pt}\mu_{\delta} &= -\epsilon\laplace\varphi_{\delta} + \epsilon^{-1}\psi'(\varphi_{\delta})-\chi_{\varphi}\sigma_{\delta}&&\quad\text{a.\,e.\ in }Q,\\
			\label{7_weak_formulation_eq_approx_1c}
			\hspace{-2pt}\delt\sigma_{\delta} + \grad\sigma_{\delta}\cdot\v_{\delta} &= \chi_{\sigma}\laplace\sigma_{\delta} - \chi_{\varphi}\laplace\varphi_{\delta}-f(\varphi_{\delta},\sigma_{\delta})h(\varphi_{\delta})&&\quad\text{a.\,e.\ in }Q,\\
			\label{7_weak_formulation_eq_approx_1e}\hspace{-2pt}\grad\varphi_{\delta}\cdot\n &= \grad\sigma_{\delta}\cdot\n = 0&&\quad\text{a.\,e.\ on }\oldSigma.
			\end{alignat}
		\end{subequations}
		Moreover, the estimate
		\begin{align}
		\nonumber &\norm{\varphi_{\delta}}_{H^1((H^1)^*)\cap L^{\infty}(H^1)\cap L^2(H^3)} + \norm{\sigma_{\delta}}_{H^1(L^2)\cap L^{\infty}(H^1)\cap L^2(H^2)} \\
		\label{7_EN_ID_approx_system}&\quad + \norm{\mu_{\delta}}_{L^4(L^2)\cap L^2(H^1)}+ \norm{\v_{\delta}}_{H^1(L^{\frac{3}{2}})\cap L^{\infty}(L^2)\cap L^\frac{16}{5}(\mathbf{V})\cap L^2(\W^{1,\frac{10}{3}})\cap L^{\frac{8}{5}}(\H^2)}\leq C
		\end{align}
		is satisfied for a constant $C$ independent of $(\varphi_{\delta},\mu_{\delta},\sigma_{\delta},\v_{\delta})$.
	\end{lemma}
	\begin{remark}
		With the above regularity, we can reconstruct the pressure $p_\delta\in L^{\frac{8}{3}}(L_0^2)\cap L^{\frac{8}{5}}(H^1)$ such that 
		\begin{equation*}
		\delta\delt\v-\divergence(2\eta\D \v) + \nu\v -\grad p  = (\mu+\chi_{\varphi}\sigma)\grad\varphi\quad\text{a.\,e.\ in }Q
		\end{equation*}
		and
		\begin{equation*}
		\norm{p_{\delta}}_{L^{\frac{8}{3}}(L_0^2)\cap L^{\frac{8}{5}}(H^1)}\leq C
		\end{equation*}
		holds for a constant $C$ independent of $(\varphi_{\delta},\mu_{\delta},\sigma_{\delta},\v_{\delta},p_{\delta})$, see \cite[Lem. II.2.2.2]{Sohr}.
	\end{remark}
	\begin{proof}[Proof of Lemma \ref{7_Theorem_non_deg_mobility_approx}]
		The proof is based on ideas presented in \cite{FritzEtAl} and \cite[Theorem 2.1]{GarckeLam4}. We will only present the a priori estimates on a formal level. However, they can be justified rigorously within a Galerkin scheme, see \cite{FritzEtAl} for details.\\ [1ex]
		Using $\divergence (\v_{\delta})=0$ a.\,e.\ in $\Omega$ and $\v_{\delta}=\mathbf{0}$ a.\,e.\ on $ \del\Omega$, we deduce
		\begin{align}
		\label{apriori_approx_2}- \inn{\grad\sigma_{\delta}\cdot\v_{\delta}}{\sigma_{\delta}}=  \tfrac{1}{2}\inn{\grad\left(|\sigma_{\delta}|^2\right)}{\v_{\delta}}=0, \quad  \inn{\grad\varphi_{\delta}\cdot\v_{\delta}}{\varphi_\delta} =0.
		\end{align}
		Choosing $\xi = \mu_{\delta}+\chi_{\varphi}\sigma_{\delta} + \varphi_{\delta}$ in \eqref{7_weak_formulation_eq_approx_1a}, $\mathbf{u}=\v_{\delta}$ in \eqref{7_weak_formulation_eq_approx_1d}, multiplying \eqref{7_weak_formulation_eq_approx_1b} with $-\delt\varphi_{\delta}$, \eqref{7_weak_formulation_eq_approx_1c} with $D\sigma_{\delta}$ for $D>0$ to be chosen, integrating by parts and summing the resulting identities, we arrive at
		\begin{align}
		\nonumber &\tfrac{\d}{\dt}\left(\tfrac{1}{2}\norml{2}{\varphi_{\delta}}^2+\tfrac{\epsilon}{2}\normL{2}{\grad\varphi_{\delta}}^2+\epsilon^{-1}\norml{1}{\psi(\varphi_{\delta})}+\tfrac{D}{2}\norml{2}{\sigma_{\delta}}^2 + \tfrac{\delta}{2}\normL{2}{\v_{\delta}}^2\right) \\
		\nonumber	&\qquad + \normL{2}{\sqrt{m(\varphi_{\delta})}\grad\mu_{\delta}}^2+D\chi_{\sigma}\normL{2}{\grad \sigma_{\delta}}^2+ 2\eta\normL{2}{\D\v_{\delta}}^2+\nu\normL{2}{\v_{\delta}}^2\\
		\nonumber &\quad=   \inn{g(\varphi_{\delta},\sigma_{\delta})h(\varphi_{\delta})}{\mu_{\delta}+\chi_{\varphi}\sigma_{\delta}+\varphi_\delta} -  D\inn{f(\varphi_{\delta},\sigma_{\delta})h(\varphi_{\delta})}{\sigma_{\delta}}\\
		\nonumber&\qquad + D\chi_{\varphi}\inn{\grad\varphi_{\delta}}{\grad\sigma_{\delta}}- \inn{m(\varphi_{\delta})\grad\mu_{\delta}}{\grad(\chi_{\varphi}\sigma_{\delta}+\varphi_{\delta})}\\
		\label{7_PRO_NDM_EQ_1} &\quad \eqqcolon I_1+I_2+I_3+I_4.
		\end{align}
		We now estimate the terms on the right hand side of \eqref{7_PRO_NDM_EQ_1} individually. By $C$ we denote a generic constant independent of $(\varphi_{\delta},\mu_{\delta},\sigma_{\delta},\v_{\delta})$ and we will frequently use H\"older's and Young's inequalities.\\ [1ex]
		In order to control the term involving $g$, we need a bound on $(\mu_{\delta}+\chi_{\varphi}\sigma_{\delta},1)$.
		Taking $v=1$ in $\eqref{7_weak_formulation_eq_approx_1b}_1$ and using \eqref{7_Assumption_on_psi}, we see that
		\begin{equation}
		\label{apriori_approx_3}\left|\inn{\mu_{\delta}+\chi_{\varphi}\sigma_{\delta}}{1}\right| = \left| \inn{\epsilon^{-1}\psi'(\varphi_{\delta})}{1}\right|\leq C\left(1+\norml{2}{\varphi_{\delta}}\right).
		\end{equation}
		Applying \eqref{7_Assumption_on_fg}, we obtain from Poincaré's inequality that
		\begin{align*}
		\left|I_1 \right| \leq C\left(1+\norml{2}{\varphi_{\delta}}^2 +\norml{2}{\sigma_{\delta}}^2\right)  + \tfrac{D\chi_{\sigma}}{4}\normL{2}{\grad\sigma_{\delta}}^2+\tfrac{m_0}{4}\normL{2}{\grad\mu_{\delta}}^2.
		\end{align*}
		For the term involving $f$ we infer
		\begin{align*}
		&\left|I_2\right|\leq C\left(1+\norml{2}{\sigma_{\delta}}^2+\norml{2}{\varphi_{\delta}}^2\right).
		\end{align*}
		Moreover, we obtain
		\begin{equation*}
		|I_3|\leq \tfrac{D\chi_{\sigma}}{4}\normL{2}{\grad\sigma_{\delta}}^2 + \tfrac{D\chi_{\varphi}^2}{\chi_{\sigma}}\normL{2}{\grad\varphi_{\delta}}^2.
		\end{equation*}
		For the last term on the right hand side of \eqref{7_PRO_NDM_EQ_1}, we obtain
		\begin{equation*}
		\left|I_4\right|\leq  C\normL{2}{\grad\varphi_{\delta}}^2+ \tfrac{2M_0^2\chi_{\varphi}^2}{m_0}\normL{2}{\grad\sigma_{\delta}}^2+ \tfrac{m_0}{4}\normL{2}{\grad\mu_{\delta}}^2 .
		\end{equation*}
		On account of the last four estimates and the assumptions on $m(\cdot)$, by choosing $D=\max\left( 1, \tfrac{4M_0^2\chi_{\varphi}^2+m_0}{\chi_{\sigma}m_0}\right)$ we obtain from \eqref{7_PRO_NDM_EQ_1} that
		\begin{align}
		\nonumber &\tfrac{\d}{\dt}\left(\tfrac{1}{2}\norml{2}{\varphi_{\delta}}^2+\tfrac{\epsilon}{2}\normL{2}{\grad\varphi_{\delta}}^2+\epsilon^{-1}\norml{1}{\psi(\varphi_{\delta})}+\tfrac{1}{2}\norml{2}{\sigma_{\delta}}^2 + \tfrac{\delta}{2}\normL{2}{\v_{\delta}}^2\right) \\
		\nonumber &\qquad + \tfrac{m_0}{2}\normL{2}{\grad\mu_{\delta}}^2+\tfrac{1}{2}\normL{2}{\grad \sigma_{\delta}}^2+ 2\eta\normL{2}{\D\v_{\delta}}^2+\nu\normL{2}{\v_{\delta}}^2\\
		\label{7_PRO_NDM_EQ_6} &\quad \leq   C\left( 1 + \norml{2}{\varphi_{\delta}}^2+\normL{2}{\grad\varphi_{\delta}}^2 + \norml{2}{\sigma_{\delta}}^2\right) .
		\end{align}
		Integrating \eqref{7_PRO_NDM_EQ_6} in time from $0$ to $s\in (0,T]$, using the assumptions on $\psi(\cdot)$ and the initial data along with \eqref{apriori_approx_3}, a Gronwall argument yields
		\begin{align}
		\nonumber&\esssup_{s\in (0,T]}\big(\norm{\psi(\varphi_{\delta})(s)}_{L^1} + \normh{1}{\varphi_{\delta}(s)}^2 + \norml{2}{\sigma_{\delta}(s)}^2 + \normL{2}{\v_{\delta}(s)}^2\big)\\
		\label{7_PRO_NDM_EQ_14}&\quad + \intT\normh{1}{\mu_{\delta}}^2 + \normL{2}{\grad\sigma_{\delta}}^2 + \norm{\v_{\delta}}_{\H^1}^2\dt\leq C.
		\end{align}
		\textbf{Higher order estimates.}
		Using regularity theory and interpolation arguments as in \cite{EbenbeckGarcke2}, and using the assumptions on $\psi(\cdot)$, we obtain that
		\begin{equation}
		\label{7_PRO_NDM_EQ_15}\norm{\varphi_{\delta}}_{L^4(H^2)\cap L^2(H^3)} + \norm{\mu_{\delta}}_{L^4(L^2)}\leq \tilde{C}.
		\end{equation}
		In particular, we obtain that $\mu_{\delta}$ is uniformly bounded in $L^4(L^2)$. By Gagliardo--Nirenberg's inequality and Sobolev embedding theory, we have the continuous embeddings $L^{\infty}(\L^2)\cap L^2(\H^2)\hookrightarrow L^{\frac{8}{3}}(\L^{\infty})$ and $H^1\subset L^6$. Then, it follows that $(\mu_{\delta}+\chi_{\varphi}\sigma_{\delta})\grad\varphi_{\delta}$ is bounded uniformly in $L^{\frac{8}{5}}(\L^2)\cap L^2(\L^{\frac{3}{2}})$. By classical regularity theory for the instationary Stokes equation (see, e.\,g., \cite[II.3, Cor. 4, p. 148]{GigaHandbook}), we conclude that
		\begin{equation*}
		\norm{\v_{\delta}}_{H^1(L^{\frac{3}{2}})\cap L^{\frac{8}{5}}(\H^2)}\leq C.
		\end{equation*}
		Applying Gagliardo--Nirenberg's inequality combined with \eqref{7_PRO_NDM_EQ_14} and using the last bound, it holds 
		\begin{equation}
		\label{7_PRO_NDM_EQ_17}\norm{\v_{\delta}}_{H^1(L^{\frac{3}{2}})\cap L^{\frac{16}{5}}(\mathbf{V})\cap L^2(\W^{1,\frac{10}{3}})\cap L^{\frac{8}{5}}(\H^2)}\leq C.
		\end{equation}
		Now, we derive higher order estimates for the nutrient concentration $\sigma_{\delta}$. Multiplying \eqref{7_weak_formulation_eq_approx_1c} with $-\laplace\sigma_{\delta}$ and integrating by parts, we obtain
		\begin{align}
		\label{7_PRO_NDM_EQ_18}\tfrac{\d}{\dt}\tfrac{1}{2}\normL{2}{\grad\sigma_{\delta}}^2 + \chi_{\sigma}\norml{2}{\laplace\sigma_{\delta}}^2 =  \inn{\chi_{\varphi}\laplace\varphi_{\delta} + f(\varphi_{\delta},\sigma_{\delta})h(\varphi_{\delta})+\grad\sigma_{\delta}\cdot\v_{\delta}}{\laplace\sigma_{\delta}}.
		\end{align}
		Using the assumptions on $f$, $h$ and \eqref{7_PRO_NDM_EQ_14}-\eqref{7_PRO_NDM_EQ_15} yields
		\begin{equation*}
		\left|\inn{\chi_{\varphi}\laplace\varphi_{\delta} + f(\varphi_{\delta},\sigma_{\delta})h(\varphi_{\delta})}{\laplace\sigma_{\delta}}\right|\leq C(1+\norml{2}{\laplace\varphi_{\delta}}^2) + \tfrac{\chi_{\sigma}}{4}\norml{2}{\laplace\sigma_{\delta}}^2.
		\end{equation*}
		With similar arguments and using the Sobolev embedding $\W^{1,\frac{10}{3}}\subset \L^{\infty}$, we infer
		\begin{align*}
		\left| \inn{\grad\sigma_{\delta}\cdot\v_{\delta}}{\laplace\sigma_{\delta}}\right|\leq C\normL{2}{\grad\sigma_{\delta}}^2\normW{1}{\frac{10}{3}}{\v_{\delta}}^2 + \tfrac{\chi_{\sigma}}{4}\norml{2}{\laplace\sigma_{\delta}}^2.
		\end{align*}
		Employing the last two inequalities in \eqref{7_PRO_NDM_EQ_18}, 
		integrating the resulting inequality in time from $0$ to $s\in (0,T]$, using \eqref{7_PRO_NDM_EQ_14}-\eqref{7_PRO_NDM_EQ_17} and elliptic regularity theory, a Gronwall argument yields
		\begin{equation}
		\label{7_PRO_NDM_EQ_19}\norm{\sigma_{\delta}}_{L^{\infty}(H^1)\cap L^2(H^2)}\leq C.
		\end{equation}
		\textbf{Estimates for the time derivatives and the convection terms.} 
		By \eqref{7_PRO_NDM_EQ_14}, \eqref{7_PRO_NDM_EQ_17}, the Sobolev embedding $\W^{1,\frac{10}{3}}\subset \L^{\infty}$ and H\"older's inequality, we have
		\begin{equation*}
		\norm{\grad\varphi_{\delta}\cdot\v_{\delta}}_{L^2(L^2)}\leq C\norm{\grad\varphi_{\delta}}_{L^{\infty}(\L^2)}\norm{\v_{\delta}}_{L^2(\L^{\infty})}\leq C\norm{\varphi_{\delta}}_{L^{\infty}(H^1)}\norm{\v_{\delta}}_{L^2(\W^{1,\frac{10}{3}})}\leq C,
		\end{equation*}
		and therefore
		\begin{equation}
		\label{7_PRO_NDM_EQ_20}\norm{\divergence(\varphi_{\delta}\v_{\delta})}_{L^2(L^2)}\leq C.
		\end{equation}
		Using the equation \eqref{7_weak_formulation_eq_approx_1a} for $\delt\varphi_\delta$ and \eqref{7_PRO_NDM_EQ_14}, \eqref{7_PRO_NDM_EQ_20}, we find that similar as in \cite{EbenbeckGarcke}
		\begin{equation*}
		\norm{ \delt\varphi_{\delta}}_{L^2((H^1)^*)}\leq C.
		\end{equation*}
		With exactly the same arguments as above, we obtain
		\begin{equation*}
		\norm{\divergence(\sigma_{\delta}\v_{\delta})}_{L^2(L^2)}\leq C.
		\end{equation*}
		Then, using the assumptions on $f$ and $h$, \eqref{7_PRO_NDM_EQ_14}-\eqref{7_PRO_NDM_EQ_15} and \eqref{7_PRO_NDM_EQ_19}, it follows that
		\begin{equation*}
		\norm{ \delt\sigma_{\delta}}_{L^2(L^2)}\leq C.
		\end{equation*}
		Summarising the previous estimates,we obtain \eqref{7_EN_ID_approx_system}.
		These a priori estimates are enough to pass to the limit within a Galerkin scheme. We omit the details and refer the reader to \cite{EbenbeckGarcke,EbenbeckGarcke2,FritzEtAl}.\\[1ex]
		\textbf{Reconstruction of the pressure.} By standard theory for the instationary Stokes equation (see, e.\,g., \cite[II.3, Cor. 4, p. 148]{GigaHandbook}) and using that $(\mu_{\delta}+\chi_{\varphi}\sigma_{\delta})\grad\varphi_{\delta}\in L^{\frac{8}{5}}(\L^2)\cap L^2(\L^{\frac{3}{2}})$, there exists a unique pressure $p_{\delta}\in L^{\frac{8}{5}}(H^1)\cap L^{2}(W^{1,\frac{3}{2}})$ satisfying $\inn{p_{\delta}}{1} = 0$.
	\end{proof}
	
	\subsection{The degenerate case}
	\subsubsection{Introduction of the mathematical setting}
	In the following let $\Omega\subset \R^d$, $d=2,3$, be a bounded domain with $ \del\Omega\in C^3$. 
	We assume that $\psi(\cdot)$ can be decomposed as
	\begin{equation*}
	\psi(\varphi) \coloneqq \psi^1(\varphi) + \psi^2(\varphi)
	\end{equation*}
	with functions $\psi^1$, $\psi^2$, where $\psi^2\in C^2([-1,+\infty))$ satisfies
	\begin{equation*}
	|(\psi^2)''(\varphi)|\leq C\quad\forall \,\varphi\in [1,+\infty),
	\end{equation*}
	and $\psi^1\colon (-1,+\infty)\to\R$ is convex and of the form
	\begin{equation}\label{7_ASS_2_DEG_PSI}
	(\psi^1)''(\varphi) = \max\left(0, \min\left(\tfrac{1}{2}(1+\varphi), 1\right)\right)^{-p_0}F(\varphi)\quad\text{for some }p_0\in[1,2]
	\end{equation}
	with a $C^1$-function $F\colon[-1,+\infty)\to \R_0^+$ satisfying $\norm{F}_{C^1[-1,+\infty)}\leq F_0$ for a positive constant $F_0$. Hence, $\psi$ is allowed to be singular in the convex part as $\varphi\to -1$. Without loss of generality, we assume that $(\psi^1)'(0)=(\psi^1)(0) = 0$.\\ [1ex]
	We introduce a degenerate mobility $m(\cdot)$ of the form
	\begin{equation}
	\label{7_ASS_4_DEG_MOB}m(\varphi) = \max\left(0, \min\left(\tfrac{1}{2}(1+\varphi), 1\right)\right)^{q_0}\bar{m}(\varphi)\quad \text{with } q_0\in [1,2],\, q_0\geq p_0,
	\end{equation}
	with $p_0$ as in \eqref{7_ASS_2_DEG_PSI}, and a $C^1$-function $\bar{m}\colon[-1,+\infty) \to\R$ satisfying
	\begin{align*}
	m_0\leq \bar{m}(\varphi)\leq M_0 \quad\forall \,\varphi\in[-1,+\infty),\qquad \norm{\bar{m}}_{C^1[-1,+\infty)}\leq M_1
	\end{align*}
	for positive constants $m_0$, $M_0$ and $M_1$. We extend the definition of $m(\cdot)$ to all of $\R$ by $m(\varphi)=0$ for $\varphi<-1$. \\ [1ex]
	Finally, we define the entropy like function $\Phi:(-1,+\infty) \to\R_0^+$ by
	\begin{equation*}
	\Phi''(\varphi)=\frac{1}{m(\varphi)},\quad \Phi'(0)=0,\quad \Phi(0)=0.
	\end{equation*}
	
	\subsubsection{The main theorem}
	The goal of this section is to prove the following theorem:
	\begin{theorem}[degenerate case]\label{7_THM_DEG_MOB}
		Let $\psi$ be as in Subsection 5.2.1 and let Assumptions \ref{7_ASS}, \textnormal{(ii)}-\textnormal{(vi)} be fulfilled. In addition, we assume that $\varphi_0\geq -1$ a.\,e.\ in $\Omega$ and 
		\begin{equation*}
		\inn{\psi(\varphi_0)+\Phi(\varphi_0)}{1}\leq C
		\end{equation*}
		for a positive constant $C$.
		Then, there exists a quadruplet $(\varphi,\textbf{J},\sigma,\v)$ satisfying
		\begin{enumerate}
			\item[a)]$\varphi\in H^1((H^1)^*)\cap C([0,T];L^2)\cap L^{\infty}(H^1)\cap L^2(H^2)$,
			\item[b)] $\varphi(0)=\varphi_0$ in $L^2$ and $ \grad\varphi\cdot\n = 0$ a.\,e. on $\oldSigma$,
			\item[c)] $\varphi\geq -1$ a.\,e.\ in $Q$,
			\item[d)] $\sigma\in H^1((H^1)^*)\cap C^0(L^2)\cap L^{\infty}(L^6)\cap L^2(H^1)$,
			\item[e)] $\sigma(0)=\sigma_0$ in $L^2$,
			\item[f)] $\textbf{J}\in L^2(\L^2)$,
			\item[g)] $\v \in L^2(\H^1)$,
		\end{enumerate}
		and solving
		\begin{subequations}
			\begin{align}
			\label{7_WF_DEGMOB_1}&\intT  \inndual{\del_t\varphi}{\xi}_{H^1}\dt = \intT \inn{\textbf{J}}{\grad\xi}\dt + \intT  \inn{g(\varphi,\sigma)h(\varphi)-\grad\varphi\cdot\v}{\xi}\dt,\\
			\label{7_WF_DEGMOB_2}& \inndual{\del_t\sigma}{\phi}_{H^1} =  \inn{-\chi_{\sigma}\grad\sigma+\chi_{\varphi}\grad\varphi+\sigma\v}{\grad\phi} - \inn{f(\varphi,\sigma)h(\varphi)}{\phi},\\
			\label{7_WF_DEGMOB_4} & 2\eta\inn{\D \v}{\D\mathbf{u}} + \nu\inn{\v}{\mathbf{u}}= \epsilon\inn{\grad\varphi\otimes \grad\varphi}{\grad\mathbf{u}}
			\end{align}
			for almost all $t\in (0,T)$ and all $\xi\in L^2(H^1)$, $\phi\in H^1$, $\mathbf{u}\in \mathbf{V}$, where
			\begin{equation*}
			\textbf{J} = -m(\varphi)\grad (-\epsilon\laplace\varphi + \epsilon^{-1}\psi'(\varphi)-\chi_{\varphi}\sigma)
			\end{equation*}
			holds in the sense that
			\begin{equation}
			\label{7_WF_DEGMOB_5}\intT \inn{\mathbf{J}}{\boldsymbol{\eta}}\dt = -\intT \inn{\epsilon\laplace\varphi}{\divergence(m(\varphi)\boldsymbol{\eta})} + \inn{\epsilon^{-1}(m \psi'')(\varphi)\grad\varphi-\chi_{\varphi}m(\varphi)\grad\sigma}{\boldsymbol{\eta}}\dt
			\end{equation}
		\end{subequations}
		for all $\boldsymbol{\eta}\in L^2(\H^1)\cap L^{\infty}(\L^{\infty})$ with $\boldsymbol{\eta}\cdot\n=0$ a.\,e.\ on $ \oldSigma$. Furthermore, there exists a unique pressure $p\in L^{\frac{4}{3}}(L_0^2)$ satisfying 
		\begin{equation*}
		-\grad p = -\divergence\left(2\eta\D\v - \epsilon(\grad\varphi\otimes\grad\varphi)\right)+\nu\v\quad\text{in }L^{\frac{4}{3}}(\mathbf{V}^*).
		\end{equation*}
	\end{theorem}
	\begin{remark}
		In the case $q_0<2$ (and therefore $p_0<2$), the assumption
		\begin{equation*}
		\inn{\psi(\varphi_0)+\Phi(\varphi_0)}{1}\leq C
		\end{equation*}
		imposes no restriction on the initial data, since $\psi(\cdot)$ and $\Phi(\cdot)$ are bounded in $-1$.
	\end{remark}
	\subsection{Approximation scheme}
	In the following let $\delta\in (0,1]$. We introduce a positive mobility $m_{\delta}$ by
	\begin{equation*}
	m_{\delta}(\varphi)\coloneqq \begin{cases}
	m(-1+\delta) & \text{for } \varphi\leq -1+\delta,\\
	m(\varphi) &\text{for } \varphi> -1+\delta, 
	\end{cases}
	\end{equation*}
	and we define $\Phi_{\delta}$ such that $\Phi_{\delta}''(\varphi)=\frac{1}{m_{\delta}(\varphi)}$ and $\Phi_{\delta}'(0)=\Phi_{\delta}(0)=0$. In particular, we have $\Phi_{\delta}(\varphi)=\Phi(\varphi)$ for $\varphi\geq -1+\delta$.
	The modified potential $\psi_{\delta}\colon \R\to\R$ is defined by $\psi_{\delta}\coloneqq \psi_{\delta}^1 + \psi^2$ where
	\begin{equation*}
	\left(\psi_{\delta}^1\right)''(\varphi)\coloneqq \begin{cases}
	\left(\psi^1\right)''(-1+\delta) & \text{for } \varphi\leq -1+\delta,\\
	\left(\psi^1\right)''(\varphi)& \text{for } \varphi>-1+\delta, 
	\end{cases}
	\end{equation*}
	and $\psi_{\delta}^1(0)=\psi^1(0)$, $\left(\psi_{\delta}^1\right)'(0)=\left(\psi^1\right)'(0)$. As for $\Phi$ we get $\psi_{\delta}(\varphi)=\psi(\varphi)$ if $\varphi\geq -1+\delta$. Furthermore, we extend $\psi^2$ to a function on all $\R$ such that $\norm{\psi^2}_{C^2(\R)}\leq C$.\\ [1ex]
	With these choices for $m_{\delta}$ and $\psi_{\delta}$, by Lemma \ref{7_Theorem_non_deg_mobility_approx} there exists a weak solution (which will be denoted by $(\varphi_{\delta},\mu_{\delta}, \sigma_{\delta}, \v _{\delta}, p_{\delta}))$ of \eqref{7_state_eq_1a}, \eqref{7_state_eq_1b_approx}, \eqref{7_state_eq_1c}-\eqref{7_state_eq_1e} and \eqref{7_boundary_cond_1a}-\eqref{7_boundary_cond_1b}, \eqref{7_boundary_cond_1c_approx} with $m(\cdot)$ and $\psi(\cdot)$ replaced by $m_{\delta}(\cdot)$ and $\psi_{\delta}(\cdot)$.
	
	\begin{remark}
		Due to \eqref{7_weak_formulation_eq_approx_1b}, we see that
		\begin{equation*}
		(\mu_\delta+\chi_{\varphi}\sigma_{\delta})\grad\varphi_\delta = \grad\left(\tfrac{\epsilon}{2}|\grad\varphi_\delta|^2 + \epsilon^{-1}\psi_{\delta}(\varphi_\delta) \right)-\divergence(\epsilon\grad\varphi_\delta\otimes \grad\varphi_{\delta}).
		\end{equation*}
		Therefore, \eqref{7_weak_formulation_eq_approx_1d} is equivalent to
		\begin{equation}
		\label{7_WF_DEGMOB_PROOF_4a} \delta\inn{\delt\v_{\delta}}{\mathbf{u}} + 2\eta\inn{\D\v_{\delta}}{\D\mathbf{u}} + \nu\inn{\v_{\delta}}{\mathbf{u}}=  \epsilon\inn{\grad\varphi_{\delta}\otimes \grad\varphi_{\delta}}{\grad\mathbf{u}}
		\end{equation}
		for a.\,e.\ $t\in(0,T)$ and for all $\mathbf{u} \in \mathbf{V}$.
	\end{remark}
	\subsubsection{Some preliminary results}
	The following lemma will be important to estimate the source terms independently of $\delta\in (0,1]$.
	\begin{lemma}\label{7_LEM_PROD_PSI_h} For all $s\in\R$ it holds that
		\begin{equation*}
		|h(s)(\psi_{\delta}^1)'(s)|+|h(s)\Phi_{\delta}'(s)|\leq C(1+|s|)
		\end{equation*}
		with a constant $C$ independent of $\delta\in (0,1]$.
	\end{lemma}
	\begin{proof}
		Let $\delta\in (0,1]$ be arbitrary. In the following we will frequently use the assumptions on $h(\cdot),\,F(\cdot)$ and $(\psi_{\delta}^1)'(0)= \Phi_{\delta}'(0)=0$. We consider only the case $p_0 = 2$, which corresponds to the highest degree of singularity of $(\psi_{\delta}^1)''$ and $(\Phi_{\delta}^1)''$. By $C$ we denote a generic constant independent of $\delta\in (0,1]$. We distinguish different cases. 
		\begin{enumerate}
			\item[(i)] For $s \leq -1$ we have due to \eqref{7_Assumption_on_fg} that $h(s)(\psi_{\delta}^1)'(s) = 0$.
			\item[(ii)] If $s\in (-1,-1+\delta)$, it holds
			\begin{align*}
			|h(s)(\psi_{\delta}^1)'(s )| & = \left| h(s)\left(\int_{s}^{-1+\delta}4F(-1+\delta)\delta^{-2}\dt + \int_{-1+\delta}^{0 }4F(t)(1+t)^{-2}\dt \right)\right|\\
			&\leq  4F_0 h(s)\left( -1 + \delta^{-1} + \delta^{-2}|s -(-1+\delta)| \right)\\
			&\leq C,
			\end{align*}
			where we used that $|s -(-1+\delta)|\delta^{-2}\leq \delta^{-1}$ and $0\leq h(s)\leq C_7\delta$.
			\item[(iii)] In the case $s \in (-1+\delta,0)$, an easy computation shows
			\begin{equation*}
			|h(s)(\psi_{\delta}^1)'(s)|	\leq h(s)\left|\int_{s}^{0 }4F_0(1+t)^{-2}\dt\right| = 4F_0h(s)\left(-1+(1+s)^{-1}\right).
			\end{equation*}
			Since $\tfrac{h(s)}{1+s}\leq C_7$ for $s\in [-1,1]$, this implies that $|h(s)(\psi_{\delta}^1)'(s)|\leq C$.
			\item[(iv)] For $s\geq 0$, the assumptions on $h(\cdot)$ and $\psi_{\delta}^1(\cdot)$ guarantee that $|h(s)(\psi_{\delta}^1)'(s)|\leq C(1+|s|)$.
		\end{enumerate}
		In summary, this shows that 
		\begin{equation*}
		|h(s)(\psi_{\delta}^1)'(s)|\leq C(1+|s|)\quad\forall\, s\in \R.
		\end{equation*}
		Using the assumptions on $\bar{m}(\cdot)$, with exactly the same arguments it follows that $|h(s)\Phi_{\delta}'(s)|\leq C(1+|s|)$ for all $s\in\R$, which completes the proof.
	\end{proof}
	The following lemma summarises uniform estimates for the approximating solutions.
	\begin{lemma}[a priori estimates] \label{7_LEM_APRI_DEG} There exists a $\delta_0$ such that for all $0<\delta\leq\delta_0$ the following estimates hold with a constant $C$ independent of $\delta$:
		\begin{subequations}
			\begin{align}
			\nonumber&\esssup_{0\leq t\leq T}\left(\normh{1}{\varphi_{\delta}(t)}^2 +  \norml{2}{\sigma_{\delta}(t)}^2+ \norml{1}{\psi_{\delta}(\varphi_{\delta}(t))} + \norml{1}{\Phi_{\delta}(\varphi_{\delta}(t))} + \delta\normL{2}{\v_{\delta}(t)}^2\right)\\
			\nonumber&\quad +\intT \normL{2}{\sqrt{m_{\delta}(\varphi_{\delta})}\grad\mu_{\delta}}^2 + \normL{2}{\grad\sigma_{\delta}}^2+ \norml{2}{\laplace\varphi_{\delta}}^2 + \normL{2}{\sqrt{(\psi_{\delta}^1)''(\varphi_{\delta})}\grad\varphi_{\delta}}^2\dt\\
			\label{7_APRI_DMOB_1}&\quad +\intT\normH{1}{\v_{\delta}}^2\dt \leq C,\\
			\label{7_APRI_DMOB_2}&\esssup_{0\leq t\leq T}\intO (-\varphi_{\delta}(t)-1)_{+}^2\dx\leq C\delta,\\
			\label{7_APRI_DMOB_3}&\intT \normL{2}{\mathbf{J}_{\delta}}^2\dt\leq C\text{  where  }\mathbf{J}_{\delta}\coloneqq m_{\delta}(\varphi_{\delta})\grad\mu_{\delta}.
			\end{align}	
		\end{subequations}
	\end{lemma}
	\begin{proof}
		In the following we denote by $C$ a generic positive constant independent of $\delta\in (0,1]$, which may change its value even within one line. Furthermore, we will frequently use H\"older's and Young's inequalities. \\ [1ex]
		\textbf{Step 1:} First of all, multiplying $\eqref{7_boundary_cond_1d_approx}_1$ with $\sigma_{0,\delta}$, integrating over $\Omega$ and by parts and using $\eqref{7_boundary_cond_1d_approx}_2$, we obtain
		\begin{equation}
		\label{7_Proof_deg_mob_eq_2}\norml{2}{\sigma_{0,\delta}}\leq C\norml{2}{\sigma_0}.
		\end{equation}
		Using that $\psi_{\delta}(\cdot)$ is a quadratic perturbation of a convex functional and invoking \cite[Lemma 4.1]{RoccaSchimperna}, for almost every $t\in(0,T)$ it holds
		\begin{align*}
		&	 \inndual{\delt\varphi_{\delta}}{-\epsilon\laplace\varphi_{\delta}+\epsilon^{-1}\psi_{\delta}'(\varphi_{\delta}) + \varphi_{\delta}}_{H^1} \nonumber \\ & \qquad = \tfrac{\d}{\dt}\left( \tfrac{1}{2}\norml{2}{\varphi_{\delta}}^2+\tfrac{\epsilon}{2}\normL{2}{\grad\varphi_{\delta}}^2 + \epsilon^{-1}\norml{1}{\psi_{\delta}(\varphi_{\delta})}\right).
		\end{align*}
		Then, with exactly the same arguments as in the proof of Lemma \ref{7_Theorem_non_deg_mobility_approx}, we get
		\begin{align}
		\nonumber &\tfrac{\d}{\dt}\left(\tfrac{1}{2}\norml{2}{\varphi_{\delta}}^2+\tfrac{\epsilon}{2}\normL{2}{\grad\varphi_{\delta}}^2+\epsilon^{-1}\norml{1}{\psi_{\delta}(\varphi_{\delta})}+\tfrac{D}{2}\norml{2}{\sigma_{\delta}}^2 + \tfrac{\delta}{2}\normL{2}{\v_{\delta}}^2\right) \\
		\nonumber	&\qquad + \normL{2}{\sqrt{m_{\delta}(\varphi_{\delta})}\grad\mu_{\delta}}^2+D\chi_{\sigma}\normL{2}{\grad \sigma_{\delta}}^2+ 2\eta\normL{2}{\D\v_{\delta}}^2+\nu\normL{2}{\v_{\delta}}^2\\
		\nonumber &\quad=   -\inn{m_{\delta}(\varphi_{\delta})\grad\mu_{\delta}}{\grad(\chi_{\varphi}\sigma_{\delta}+\varphi_{\delta})}+\inn{h(\varphi_{\delta})}{g(\varphi_{\delta},\sigma_{\delta})\varphi_{\delta}- Df(\varphi_{\delta},\sigma_{\delta})\sigma_{\delta}} \\
		\nonumber&\qquad  +  D\chi_{\varphi}\inn{\grad\varphi_{\delta}}{\grad\sigma_{\delta}}+   \inn{g(\varphi_{\delta},\sigma_{\delta})h(\varphi_{\delta})}{-\epsilon\laplace\varphi_{\delta} + \epsilon^{-1}\psi_{\delta}'(\varphi_{\delta})} \\
		\label{7_Proof_deg_mob_eq_1}&\quad\eqqcolon I_1+I_2+I_3+I_4
		\end{align}
		for $D>0$ to be specified and for almost every $t\in (0,T]$, where we used \eqref{7_weak_formulation_eq_approx_1c} for $\mu_{\delta}+\chi_{\varphi}\sigma_{\delta}$ and \eqref{apriori_approx_2}.
		The assumptions on $\bar{m}(\cdot)$ guarantee that
		\begin{align*}
		\nonumber\left|I_1\right|\leq \tfrac{1}{4} \normL{2}{\sqrt{m_{\delta}(\varphi_{\delta})}\grad\mu_{\delta}}^2 + 2M_0\left(\chi_{\varphi}^2\normL{2}{\grad\sigma_{\delta}}^2+\normL{2}{\grad\varphi_{\delta}}^2\right).
		\end{align*}
		Furthermore, it holds that
		\begin{equation*}
		\left|I_3\right|\leq \frac{D\chi_{\sigma}}{2}\normL{2}{\grad\sigma_{\delta}}^2 + \frac{D\chi_{\varphi}^2}{2\chi_{\sigma}}\normL{2}{\grad\varphi_{\delta}}^2.
		\end{equation*}
		With similar arguments as in the proof of Lemma \ref{7_Theorem_non_deg_mobility_approx} we deduce
		\begin{align*}
		\left|I_2\right|\leq C_D\left(1+\norml{2}{\varphi_{\delta}}^2+\norml{2}{\sigma_{\delta}}^2\right).
		\end{align*}
		Finally,  due to the assumptions on $\psi^2(\cdot)$ and using Lemma \ref{7_LEM_PROD_PSI_h} for $\psi_{\delta}^1$ along with \eqref{7_Assumption_on_fg}, we obtain
		\begin{equation*}
		|I_4|\leq \gamma \norml{2}{\laplace\varphi_{\delta}}^2+ C_{\gamma}\left(1+\norml{2}{\varphi_{\delta}}^2+\norml{2}{\sigma_{\delta}}^2\right)
		\end{equation*}
		with $\gamma>0$ to be chosen later.
		Employing the last four inequalities in \eqref{7_Proof_deg_mob_eq_1} and choosing $D=\max\left(1,(1+4M_0\chi_{\varphi}^2)\chi_{\sigma}^{-1}\right)$ gives
		
		\begin{align}
		\nonumber &\tfrac{\d}{\dt}\left(\tfrac{1}{2}\norml{2}{\varphi_{\delta}}^2+\tfrac{\epsilon}{2}\normL{2}{\grad\varphi_{\delta}}^2+\epsilon^{-1}\norml{1}{\psi_{\delta}(\varphi_{\delta})}+\tfrac{1}{2}\norml{2}{\sigma_{\delta}}^2 + \tfrac{\delta}{2}\normL{2}{\v_{\delta}}^2\right) \\
		\nonumber	&\qquad +\tfrac{1}{2}\normL{2}{\sqrt{m_{\delta}(\varphi_{\delta})}\grad\mu_{\delta}}^2+\tfrac{1}{2}\normL{2}{\grad \sigma_{\delta}}^2+ 2\eta\normL{2}{\D\v_{\delta}}^2+\nu\normL{2}{\v_{\delta}}^2\\
		\label{7_Proof_deg_mob_eq_12} &\quad\leq C_{\gamma}\left(1+\norml{2}{\varphi_{\delta}}^2 + \normL{2}{\grad\varphi_{\delta}}^2+ \norml{2}{\sigma_{\delta}}^2\right) + \gamma\norml{2}{\laplace\varphi_{\delta}}^2.
		\end{align}
		\textbf{Step 2:}
		In the following we aim to derive an estimate for $\laplace\varphi_{\delta}$ in order to absorb the last term on the right hand side of \eqref{7_Proof_deg_mob_eq_12}. First, we note that integration by parts and $\v_{\delta}\in L^2(\mathbf{V})$ implies
		\begin{equation*}
		\inn{\grad\varphi_{\delta}\cdot \v_{\delta}}{\Phi_{\delta}'(\varphi_{\delta})} =  \inn{\grad \left(\Phi_{\delta}(\varphi_{\delta})\right)}{\v_{\delta}}= 0.
		\end{equation*}
		Consequently, choosing $\Phi_{\delta}'(\varphi_{\delta})\in L^2(H^1)$ as a test function in \eqref{7_weak_formulation_eq_approx_1a}, invoking \cite[Lemma 4.1]{RoccaSchimperna} and the identity $\Phi_{\delta}''(\varphi_{\delta}) = \frac{1}{m_{\delta}(\varphi_{\delta})}$, with similar arguments as in \cite{GarckeElliot} we obtain
		\begin{align*}
		&\tfrac{\d}{\dt} \norml{1}{\Phi_{\delta}(\varphi_{\delta})} + \epsilon\norml{2}{\laplace\varphi_{\delta}}^2 + \epsilon^{-1}\normL{2}{\sqrt{(\psi_{\delta}^1)''(\varphi_{\delta}})\grad\varphi_{\delta}}^2  \\
		&\quad=  \chi_{\varphi}\inn{\grad\varphi_{\delta}}{\grad\sigma_{\delta}} - \epsilon^{-1}\normL{2}{\sqrt{(\psi^2)''(\varphi_{\delta})}\grad\varphi_{\delta}}^2  + \inn{g(\varphi_{\delta},\sigma_{\delta})h(\varphi_{\delta})}{\Phi_{\delta}'(\varphi_{\delta})}
		\end{align*} 
		for almost every $t\in (0,T)$. 
		Using the assumptions on $\psi^2(\cdot)$, \eqref{7_Assumption_on_fg} and Lemma \ref{7_LEM_PROD_PSI_h}, with similar arguments as above we can bound the right hand side of this identity to obtain
		\begin{align}
		\nonumber&\tfrac{\d}{\dt} \norml{1}{\Phi_{\delta}(\varphi_{\delta})} + \epsilon\norml{2}{\laplace\varphi_{\delta}}^2 + \epsilon^{-1}\normL{2}{\sqrt{(\psi_{\delta}^1)''(\varphi_{\delta}})\grad\varphi_{\delta}}^2\\
		\label{7_Proof_deg_mob_eq_14}&\quad\leq C\left(1+\norml{2}{\varphi_{\delta}}^2+\normL{2}{\grad\varphi_{\delta}}^2 + \norml{2}{\sigma_{\delta}}^2\right)+\tfrac{1}{4}\normL{2}{\grad\sigma_{\delta}}^2
		\end{align}
		for almost every $t\in (0,T)$. 
		Next, we notice that $\Phi_{\delta}(u)\leq \Phi(u)$, $\psi_{\delta}^1(u)\leq \psi^1(u)$ for $\delta$ sufficiently small. Using \eqref{7_Proof_deg_mob_eq_2} and the Sobolev embedding $H^1\subset L^6$ along with the assumptions on $\varphi_0$ and $\sigma_0$, we know that
		\begin{equation}
		\label{7_Proof_deg_mob_eq_15}\tfrac{1}{2}\norml{2}{\varphi_0}^2+\tfrac{\epsilon}{2}\normL{2}{\grad\varphi_0}^2+\epsilon^{-1}\norml{1}{\psi_{\delta}(\varphi_0)}+ \norml{1}{\Phi_{\delta}(\varphi_0)}+\tfrac{1}{2}\norml{2}{\sigma_{0,\delta}}^2\leq C.
		\end{equation}
		Adding up \eqref{7_Proof_deg_mob_eq_12} and \eqref{7_Proof_deg_mob_eq_14}, choosing $\gamma = \frac{\epsilon}{2}$, integrating in time from $0$ to $t\in (0,T]$ and using \eqref{7_Proof_deg_mob_eq_15} together with Korn's inequality (see, e.\,g., \cite[Sec. 6.3]{Ciarlet}), an application of Gronwall's lemma implies \eqref{7_APRI_DMOB_1}. \\ [1ex]
		\textbf{Step 3:}
		We now prove \eqref{7_APRI_DMOB_2}. First observe that the convexity of $\Phi_{\delta}(\cdot)$ and $\Phi_{\delta}(0)= \Phi_{\delta}'(0)=0$ imply
		\begin{equation*}
		\Phi_{\delta}(-1+\delta)\geq 0,\quad \Phi_{\delta}'(-1+\delta)\leq 0.
		\end{equation*}
		Recalling the assumptions on $\bar{m}(\cdot)$ and using $\delta^{p_0}\leq \delta$, we can follow the arguments in \cite{GarckeElliot} to obtain
		\begin{equation*}
		(-z-1)^2\leq C\delta \Phi_{\delta}(z)\quad\text{for all }z\leq -1\text{ and }\delta<1.
		\end{equation*}
		Employing \eqref{7_APRI_DMOB_1} we conclude
		\begin{equation*}
		\esssup_{0\leq t\leq T}\intO (-\varphi_{\delta}(s)-1)_{\text{+}}^2\dx\leq C\delta\, \esssup_{0\leq t\leq T}\intO \Phi_{\delta}(\varphi_{\delta}(s))\dx\leq C\delta
		\end{equation*}
		which implies \eqref{7_APRI_DMOB_2}. Finally, because of \eqref{7_APRI_DMOB_1}, an easy computation shows that
		\begin{equation*}
		\intT  \normL{2}{m_{\delta}(\varphi_{\delta})\grad\mu_{\delta}}^2\dt\leq C\intT\normL{2}{ \sqrt{m_{\delta}(\varphi_{\delta})}\grad\mu_{\delta}}^2\dt\leq C,
		\end{equation*}
		and the proof is complete.
	\end{proof}
	The following lemma will be applied to pass to the limit in the approximative system \eqref{7_weak_formulation_eq_approx}.
	\begin{lemma}
		\label{7_LEM_CONV_TERMS} Let $\delta\in (0,\delta_0]$ and assume the assumptions of Theorem \ref{7_THM_DEG_MOB} are fulfilled. Then, it holds that
		\begin{align}
		\nonumber&\norm{\varphi_{\delta}}_{H^1((H^1)^*)\cap L^{\infty}(H^1)\cap L^2(H^2)}+\norm{\sigma_{\delta}}_{H^1((H^1)^*)\cap L^{\infty}(L^6)\cap L^2(H^1)}+\norm{\v_{\delta}}_{L^2(\H^1)}\\
		\label{7_LEM_CONV_TERMS_EST}&\quad + \sqrt{\delta}\norm{\v_{\delta}}_{L^{\infty}(\L^2)}+\norm{\divergence(\varphi_{\delta}\v_{\delta})}_{L^2(\L^{\frac{3}{2}})}+\norm{\divergence(\sigma_{\delta}\v _{\delta})}_{L^2((H^1)^*)}\leq C
		\end{align}
		with a positive constant $C$ independent of $\delta\in (0,\delta_0]$. Furthermore, as $\delta\to 0$ we have (at least for a non-relabelled subsequence)
		\begin{subequations}
			\begin{alignat}{4}
			\label{7_LEM_CONV_TERMS_CONVa}\varphi_{\delta}&\to\varphi&&\quad\text{weakly-star in}&&\quad H^1((H^1)^*)\cap L^{\infty}(H^1)\cap L^2(H^2),\\
			\label{7_LEM_CONV_TERMS_CONVb}\sigma_{\delta}&\to\sigma&&\quad\text{weakly-star in}&&\quad H^1((H^1)^*)\cap L^{\infty}(L^6)\cap L^2(H^1),\\
			\label{7_LEM_CONV_TERMS_CONVc}\v _{\delta}&\to\v &&\quad\text{weakly in}&&\quad L^2(\H^1),\\
			\label{7_LEM_CONV_TERMS_CONVd}\divergence(\varphi_{\delta}\v _{\delta})&\to \divergence(\varphi\v) &&\quad\text{weakly in}&&\quad L^2(L^{\frac{3}{2}}),\\
			\label{7_LEM_CONV_TERMS_CONVe}\divergence(\sigma_{\delta}\v_{\delta})&\to  \divergence(\sigma \v) &&\quad\text{weakly in}&&\quad L^2((H^1)^*),\\
			\label{7_LEM_CONV_TERMS_CONVf}\textbf{J}_{\delta}&\to\textbf{J}&&\quad\text{weakly in }&&\quad L^2(\L^2),
			\end{alignat}
			and
			\begin{alignat}{4}
			\label{7_LEM_CONV_TERMS_CONVg}\varphi_{\delta}&\to\varphi&&\quad\text{strongly in }\,C^0([0,T];L^r)\cap L^2(W^{1,r})&\quad\text{ and a.\,e.\ in }Q,\\
			\label{7_LEM_CONV_TERMS_CONVh}\sigma_{\delta}&\to\sigma&&\quad\text{strongly in }\,C^0([0,T];(H^1)^*)\cap L^p(L^r)&\quad\text{ and a.\,e.\ in }Q
			\end{alignat}
			for any $r\in[1,6)$ and $p\in [1,\infty)$.
		\end{subequations}
	\end{lemma}
	\begin{proof}
		In the following we denote by $C$ a generic constant independent of $\delta\in (0,\delta_0]$.
		Using \eqref{7_APRI_DMOB_1} and elliptic regularity theory, it follows that
		\begin{equation}
		\label{7_LEM_CONV_TERMS_1}\norm{\varphi_{\delta}}_{L^{\infty}(H^1)\cap L^2(H^2)}\leq C.
		\end{equation}
		Due to Korn's inequality and \eqref{7_APRI_DMOB_1} we have
		\begin{equation}
		\label{7_LEM_CONV_TERMS_2}\norm{\v_{\delta}}_{L^2(\H^1)}+ \sqrt{\delta}\norm{\v_{\delta}}_{L^{\infty}(\L^2)}\leq C.
		\end{equation}
		Next, multiplying $\eqref{7_weak_formulation_eq_approx_1c}_1$ with $\sigma_{\delta}^5$, integrating by parts and using that 
		\begin{equation*}
		\inn{\grad\sigma_{\delta}\cdot\v_{\delta}}{\sigma_{\delta}^5} = \tfrac{1}{6} \inn{\grad(|\sigma_{\delta}|^6)}{\v_{\delta}} = -\tfrac{1}{6} \inn{|\sigma_{\delta}|^6}{\divergence(\v_{\delta})}=0\quad\text{f.\,a.\,e. }t\in (0,T),
		\end{equation*}
		we obtain
		\begin{equation*}
		\tfrac{\d}{\dt} \tfrac{1}{6}\norml{6}{\sigma_{\delta}}^6 +  5\chi_{\sigma}\inn{\sigma_{\delta}^2\,\grad\sigma_{\delta}}{\sigma_{\delta}^2\,\grad\sigma_{\delta}} =  5\chi_{\varphi}\inn{\grad\varphi_{\delta}}{\grad\sigma_{\delta}|\sigma_{\delta}|^4}- \inn{f(\varphi_{\delta},\sigma_{\delta})h(\varphi_{\delta})}{\sigma_{\delta}^5}. 
		\end{equation*}
		Using the continuous embedding $H^1\subset L^6$, the assumptions on $h$,  $f$, and \eqref{7_LEM_CONV_TERMS_1}, we can bound the right hand side by
		\begin{equation*}
		|\text{RHS}|\leq C\left(1+\normh{2}{\varphi_{\delta}}^2\right)\norml{6}{\sigma_{\delta}}^6+ 2\chi_{\sigma}\inn{\sigma_{\delta}^2\,\grad\sigma_{\delta}}{\sigma_{\delta}^2\,\grad\sigma_{\delta}},
		\end{equation*}
		and therefore
		\begin{equation}
		\label{7_LEM_CONV_TERMS_3}\tfrac{\d}{\dt} \tfrac{1}{6}\norml{6}{\sigma_{\delta}}^6 \leq C\left(1+\normh{2}{\varphi_{\delta}}^2\right)\norml{6}{\sigma_{\delta}}^6.
		\end{equation}
		Now, multiplying \eqref{7_boundary_cond_1d_approx} with $\sigma_{0,\delta}^5$, integrating by parts and neglecting the non-negative term $5\delta (\sigma_{0,\delta}^2\,\grad\sigma_{0,\delta}\mc\sigma_{0,\delta}^2\,\grad\sigma_{0,\delta})$, we obtain
		\begin{equation*}
		\norml{6}{\sigma_{0,\delta}}^6\leq  \inn{\sigma_0}{\sigma_{0,\delta}^5}\leq \tfrac{1}{2}\norml{6}{\sigma_{0,\delta}}^6 + C\norml{6}{\sigma_{0}}^6\Longrightarrow \norml{6}{\sigma_{0,\delta}}\leq C\norml{6}{\sigma_{0}}\leq C.
		\end{equation*}
		Hence, integrating \eqref{7_LEM_CONV_TERMS_3} in time from $0$ to $t\in (0,T)$ and using \eqref{7_LEM_CONV_TERMS_1}, a Gronwall argument gives
		\begin{equation*}
		\norm{\sigma_{\delta}}_{L^{\infty}(L^6)}\leq C.
		\end{equation*}
		Together with \eqref{7_LEM_CONV_TERMS_1}-\eqref{7_LEM_CONV_TERMS_2} and using similar arguments as in, e.\,g., \cite{EbenbeckGarcke,GarckeLam1}, we obtain
		\eqref{7_LEM_CONV_TERMS_EST}.\\ [1ex]
		Recalling \eqref{7_APRI_DMOB_1}, \eqref{7_LEM_CONV_TERMS_EST}, and using a generalised version of H\"older's inequality, by standard compactness arguments we obtain \eqref{7_LEM_CONV_TERMS_CONVa}-\eqref{7_LEM_CONV_TERMS_CONVc} and \eqref{7_LEM_CONV_TERMS_CONVf}-\eqref{7_LEM_CONV_TERMS_CONVh}. The argument for \eqref{7_LEM_CONV_TERMS_CONVd}-\eqref{7_LEM_CONV_TERMS_CONVe} is slightly different. Indeed, applying \eqref{7_LEM_CONV_TERMS_EST} and reflexive weak compactness arguments, we infer that
		\begin{equation*}
		\divergence(\varphi_{\delta}\v_{\delta})\to \theta\quad\text{weakly in }L^2(L^{\frac{3}{2}})
		\end{equation*}
		for some limit function $\theta\in L^2(L^{\frac{3}{2}})$. Integrating by parts, we obtain
		\begin{equation*}
		\normL{2}{\grad\varphi_{\delta}-\grad\varphi}^4\leq C\norml{2}{\varphi_{\delta}-\varphi}^2\norml{2}{\laplace(\varphi_{\delta}-\varphi)}^2.
		\end{equation*}
		Integrating this inequality in time from $0$ to $T$, using \eqref{7_LEM_CONV_TERMS_EST}, \eqref{7_LEM_CONV_TERMS_CONVg} and weak(-star) lower semicontinuity of norms, this leads to
		\begin{align*}
		\intT \normL{2}{\grad\varphi_{\delta}-\grad\varphi}^4\dt 
		\leq C\norm{\varphi_{\delta}-\varphi}_{L^{\infty}(L^2)}^2\norm{\varphi_{\delta}-\varphi}_{L^2(H^2)}^2\to 0\quad\text{as }\delta\to 0.
		\end{align*}
		By the product of weak-strong convergence and \eqref{7_LEM_CONV_TERMS_CONVc}, this yields 
		\begin{equation*}
		\divergence(\varphi_{\delta}\v_{\delta})\to \divergence(\varphi\v)\quad \text{weakly in }L^{\frac{4}{3}}(L^{\frac{3}{2}})\quad\text{as }\delta\to 0.
		\end{equation*}
		Consequently, by uniqueness of limits we obtain $\divergence(\varphi\v) = \theta\in L^2(L^{\frac{3}{2}})$. For \eqref{7_LEM_CONV_TERMS_CONVe} one can use similar arguments as in \cite{EbenbeckGarcke,GarckeLam1}, which completes the proof.
	\end{proof}
	\subsubsection{Proof of Theorem \ref{7_THM_DEG_MOB}}
	We divide the analysis into several steps:\\ [1ex]
	\textbf{Step 1:}
	Passing to the limit in \eqref{7_APRI_DMOB_2} and using \eqref{7_LEM_CONV_TERMS_CONVg}, we conclude that
	\begin{equation*}
	\varphi\geq -1 \quad\text{a.\,e.\ in }Q.
	\end{equation*}
	Recalling \eqref{7_WF_DEGMOB_PROOF_4a}, the quadruplet $(\varphi_{\delta},\mu_{\delta},\sigma_{\delta},\v_{\delta})$ fulfils
	\begin{align*}
	0&= \intT  \hspace{-1pt}\inndual{\del_t\varphi_{\delta}}{\xi }_{H^1}  \splus \inn{\grad\varphi_{\delta}\cdot\v_{\delta}\sminus g(\varphi_{\delta},\sigma_{\delta})h(\varphi_{\delta})}{\xi} \splus \inn{m_{\delta}(\varphi_{\delta})\grad\mu_{\delta}}{\grad\xi} \dt,\\
	0&= \intT \zeta\, \big(2\eta\inn{\D \v_{\delta}}{\D\mathbf{u}} \splus \nu\inn{\v_{\delta}}{\mathbf{u}} \sminus \epsilon\inn{\grad\varphi_{\delta}\otimes\grad\varphi_{\delta}}{\grad\mathbf{u}}\big)\sminus\zeta'\delta\inn{\v_{\delta}}{\mathbf{u}}\dt \\
	0&= \intT \zeta\,\big( \inn{\del_t\sigma_{\delta}\splus f(\varphi_{\delta},\sigma_{\delta})h(\varphi_{\delta})}{\phi} \splus\inn{\chi_{\sigma}\grad\sigma_{\delta}\sminus\chi_{\varphi}\grad\varphi_{\delta}\sminus \sigma_{\delta}\v_{\delta}}{\grad \phi}\big)\dt
	\end{align*}
	for all $\zeta\in C_0^{\infty}(0,T)$, $\xi\in L^2(H^1)$, $\phi\in H^1$ and $\mathbf{u}\in \mathbf{V}$, where $\mu_{\delta}$ is given by
	\begin{equation*}
	\mu_{\delta} = -\epsilon^{-1}\laplace\varphi_{\delta}+\epsilon\psi_{\delta}'(\varphi_{\delta})-\chi_{\varphi}\sigma_{\delta}\qquad\text{a.\,e. in }Q.
	\end{equation*}
	Using Lemma \ref{7_LEM_CONV_TERMS}, with similar arguments as in, e.\,g., \cite{EbenbeckGarcke}, it follows that 
	\begin{align*}
	\intT   \inndual{\del_t\varphi}{\xi}_{H^1}\dt &= \intT \inn{\mathbf{J}}{\grad\xi} - \inn{\grad\varphi\cdot\v}{\xi} + \inn{g(\varphi,\sigma)h(\varphi)}{\xi}\dt,\\
	\inndual{\del_t\sigma}{\phi}_{H^1} &= - \inn{\chi_{\sigma}\grad\sigma-\chi_{\varphi}\grad\varphi-\sigma\v}{\grad\phi}  - \inn{f(\varphi,\sigma)h(\varphi)}{\phi}
	\end{align*}
	for almost all $t\in (0,T)$ and all $\xi\in L^2(H^1)$, $\phi\in H^1$. Due to \eqref{7_LEM_CONV_TERMS_EST} and the continuous embedding $L^{\infty}(\L^2)\cap L^2(\H^1)\hookrightarrow L^4(\L^3)$, we have that
	\begin{equation*}
	\norm{\grad\varphi_{\delta}\otimes \grad\varphi_{\delta}}_{L^{\frac{4}{3}}(\L^2)}\leq C.
	\end{equation*}
	Using reflexive weak compactness arguments, this means that $\grad\varphi_{\delta}\otimes \grad\varphi_{\delta}\rightharpoonup \boldsymbol{\theta}$ in $L^{\frac{4}{3}}(\L^2)$ for some $\boldsymbol{\theta}\in L^{\frac{4}{3}}((L^2)^{d\times d})$. Applying \eqref{7_LEM_CONV_TERMS_CONVa} and \eqref{7_LEM_CONV_TERMS_CONVg}, by the product of weak strong convergence we obtain
	\begin{equation*}
	\grad\varphi_{\delta}\otimes \grad\varphi_{\delta}\to \grad\varphi\otimes \grad\varphi\quad\text{weakly in }L^{\frac{4}{3}}(\L^p)\quad\forall \, p\in (1,2).
	\end{equation*}
	Then, by uniqueness of weak limits we deduce that $\boldsymbol{\theta} = \grad\varphi\otimes\grad\varphi$. Then, using the boundedness of $\sqrt{\delta}\v_{\delta}\in L^{\infty}(\L^2)$ and using $\zeta\grad\mathbf{u}\in C^0([0,T];\L^2)$, we infer that
	\begin{equation*}
	0=   2\eta\inn{\D \v}{\D\mathbf{u}} + \nu\inn{\v}{\mathbf{u}} - \epsilon\inn{\grad\varphi\otimes\grad\varphi}{\grad\mathbf{u}}
	\end{equation*}
	for almost all $t\in (0,T)$ and all $\mathbf{u}\in\mathbf{V}$.\\ [1ex]
	\textbf{Step 2:}
	In order to identify \textbf{J}, straightforward modifications of the arguments in \cite{GarckeElliot} can be applied.
	We remark that $m'$ is given by
	\begin{equation*}
	m'(u) = \begin{cases}
	0 & \text{for } u<-1,\\
	q_0\tfrac{1}{2^{q_0}}(1+u)^{q_0-1}\bar{m}(u) + \left(\tfrac{1}{2}(1+u)\right)^{q_0}\bar{m}'(u) &\text{for } u\in (-1,1),\\
	\bar{m}'(u)&\text{for }u>1,
	\end{cases}
	\end{equation*}
	and thus we observe that $m'(\cdot)$ may be discontinuous in $1$, and $m'(\cdot)$ is discontinuous in $-1$ if $q_0=1$ and $\bar{m}(-1)\neq 0$. 
	Therefore, we conclude that \eqref{7_WF_DEGMOB_5} holds.\\[1ex]
	\textbf{Step 3:} Attainment of initial conditions follows with standard arguments, see, e.\,g., \cite{EbenbeckGarcke}. We notice that $\sigma(0)$ is well-defined due to the continuous embedding $H^1((H^1)^*)\cap L^2(H^1)\hookrightarrow C^0([0,T];L^2)$.
	Moreover, the uniform estimates and weak(-star) lower semi-continuity of norms imply that
	\begin{equation*}
	\mathcal{S}\coloneqq  -\divergence(2\eta\D\v -\epsilon\left(\grad\varphi\otimes\grad\varphi\right)) + \nu\v \in L^{\frac{4}{3}}(\mathbf{V}^*).
	\end{equation*}
	Hence,  there exists a unique pressure $p\in L^{\frac{4}{3}}(L_0^2)$ satisfying $-\grad p = \mathcal{S}$ in the sense of distributions, see \cite[Lem. II.2.2.2]{Sohr} for details, which completes the proof.
	
	\section{Numerical results}
	In this section, we show several numerical simulations for the tumour
	growth
	model derived in the previous sections, in the case $d=2$.
	We consider the system
	\begin{subequations}\label{3_NUMS_Model_eq}
		\begin{alignat}{3}
		\label{3_NUMS_Model_eq_1a}\divergence(\v)&=
		\alpha\tfrac{1}{2}(\mathcal{P}\sigma-\mathcal{A})(\varphi+1)&&\qquad\text{in
		}Q,\\
		\label{3_NUMS_Model_eq_1b}-\divergence(\T(\varphi,\v,p))+\nu\v&=
		(\mu+\chi_{\varphi}\sigma)\grad\varphi&&\qquad\text{in }Q,\\
		\label{3_NUMS_Model_eq_1c}\delt\varphi +
		\divergence(\varphi\v)&= \divergence(m(\varphi)\grad\mu) +
		\rho_S\tfrac{1}{2}(\mathcal{P}\sigma-\mathcal{A})(\varphi+1)&&\qquad\text{in
		}Q,\\
		\label{3_NUMS_Model_eq_1d}\mu &=
		\tfrac{\beta}{\epsilon}\psi'(\varphi)-\beta\epsilon\laplace\varphi-\chi_{\varphi}\sigma&&\qquad\text{in
		}Q,\\
		\label{3_NUMS_Model_eq_1e}0&=
		\mathcal{D}\divergence(\grad\sigma-\chi\grad\varphi)-\tfrac{1}{2}\mathcal{C}\sigma(\varphi+1)&&\qquad\text{in
		}Q,
		\end{alignat}
	\end{subequations}
	where
	\begin{equation*}
	\T(\varphi,\v,p) = 2\eta(\varphi)\D\v + \lambda(\varphi)\divergence(\v)\I-p\I,
	\end{equation*}
	and with mobilities of the form \eqref{3_mobility}, that means
	\begin{equation}
	\label{3_mobility_NUMS}
	\text{(i)}\, m(\varphi)=m_0,\qquad \text{(ii)}\, m(\varphi)=\epsilon
	m_0,\qquad \text{(iii)}\, m(\varphi)=m_0\tfrac{1}{2}(1+\varphi)^2.
	\end{equation}
	We supplement the system with initial and boundary conditions of the
	form
	\begin{subequations}\label{3_NUMS_Model_IBC}
		\begin{alignat}{3}
		\label{3_NUMS_Model_IBC_1a}\grad\mu\cdot\n &=
		\grad\varphi\cdot\n = 0,\quad \sigma= \sigma_B&&\quad\text{on
		}\oldSigma,\\
		\label{3_NUMS_Model_IBC_1c}\T(\varphi,\v,p)\n&= 0 \quad\text{on
		}\del_1\Omega\times(0,T),\qquad \v=\mathbf{0}&&\quad\text{on
		}\del_2\Omega\times (0,T),\\
		\label{3_NUMS_Model_IBC_1d}\varphi(0)&=\varphi_0&&\quad\text{in
		}\Omega,
		\end{alignat}
	\end{subequations}
	where $\sigma_B$ is a given function and $\del_1\Omega$,
	$\del_2\Omega\subset\del\Omega$, are measurable, relatively open such
	that
	\begin{equation*}
	\overline{\del_1\Omega\cup\del_2\Omega} =
	\del\Omega\quad\text{and}\quad\del_1\Omega\cap\del_2\Omega = \emptyset.
	\end{equation*}
	In \eqref{3_NUMS_Model_eq} we denote by $\mathcal{P}$, $\mathcal{A}$ and
	$\mathcal{C}$ the proliferation, apoptosis and consumption rate.
	Moreover, the parameters $\mathcal{D}$, $\chi_{\varphi}$, $\chi$ and
	$\beta$ are related to nutrient diffusion, chemotaxis, active transport
	and cell-cell adhesion. The remaining variables and parameters are defined
	as before. In the case \eqref{3_mobility_NUMS}\textnormal{(ii)} we
	always set $\rho_S=\alpha$ in order to fulfil
	\eqref{3_outer_equations_1c_ii}. We remark that setting
	$\eta(\cdot)=\lambda(\cdot)\equiv 0$ leads to a Cahn--Hilliard--Darcy
	model.\newline
	\def\boldS{\mathbf S}
	\subsection{Finite element approximation}
	Let $\mathcal{T}$ be a
	regular triangulation of $\Omega$ into disjoint open simplices,
	associated with $\mathcal{T}$ is the piecewise polynomial finite element
	spaces
	\begin{align*}
	S^{h}_k :=  \left \{ \varphi \in C^{0}(\overline\Omega)
	\Big| \,  \varphi_{|_{T}} \in P_{k}(T) \; \forall ~T \in \mathcal{T}
	\right \}
	\subset H^1(\Omega),\quad k \in \mathbb{N},
	\end{align*}
	where we denote by $P_{k}(T)$ the space of polynomials of degree $k$ on
	$T$,
	and extend them naturally to the vector-valued spaces
	$\boldS^h_k$, $k\in \mathbb{N}$.
	Moreover, we define
	\begin{align*}
	K^{h} & := \{ \chi \in S^{h}_1 |~ |\chi| \leq 1 \}, \quad
	S^{h, \alpha}_1 := \{ \chi \in S^{h}_1 |~\chi = \alpha \text{ on }
	\partial \Omega \},\ \alpha \in \mathbb{R}, \nonumber \\
	\boldS^{h,0}_2 & := \{ \boldsymbol\chi \in \boldS^{h}_2
	|~\boldsymbol\chi = \mathbf{0}
	\text{ on } \partial_2 \Omega \},
	\end{align*}
	and let $I^h_k : C(\overline\Omega) \to S^h_k$ denote the standard
	interpolation operators.
	Let $(\,\cdot\mc\cdot\,)_h$ denote the mass-lumped $L^2$ inner product
	on $\Omega$ induced by $\mathcal{T}$, so that, for
	$v,w\in C(\overline\Omega)$ it holds that
	$(v,w)_h = (1,I^h_1[vw])$.
	We now introduce a finite element approximation of the tumour model
	\eqref{3_NUMS_Model_eq}-\eqref{3_NUMS_Model_IBC} with the obstacle
	potential
	\eqref{3_def_double_obstacle}.
	For simplicity we assume that $\sigma_B\in \mathbb{R}$.
	Let $\varphi_h^0 = I^h_1[\varphi_0]$, $\mu_h^0=0$, $\sigma_h^0 =
	\sigma_B$
	and fix a time step size $\tau > 0$.
	Then, for $n\geq1$, find $\v_h^{n} \in \boldS^{h,0}_2$, $p_h^n \in
	S^h_1$,
	$\varphi_{h}^{n} \in K^{h}$, $\mu_{h}^{n} \in S^{h}_1$,
	$\sigma_{h}^{n} \in S^{h,\sigma_B}_1$,
	such that for all $\boldsymbol{\xi}_h \in \boldS^{h,0}_2$, $\chi_h \in
	S^h_1$,
	$\phi_{h} \in S^{h}_1$, $\zeta_{h} \in K^{h}$ and
	$\xi_{h} \in S^{h,0}_1$
	\begin{subequations} \label{eq:chbfea}
		\begin{align}
		&
		2\inn{\eta(\varphi_h^{n-1}) \D( \v_h^{n})}{\D(
			\boldsymbol{\xi}_h)}
		+
		\inn{\lambda(\varphi_h^{n-1})\divergence(\v_h^n)-p_h^{n}}{\divergence(\boldsymbol{\xi}_h)}
		+ \nu\inn{\v_h^n}{\boldsymbol{\xi}_h}
		\nonumber \\ & \qquad\qquad
		=
		\inn{(\mu_h^{n-1}+\chi_{\varphi}\sigma_h^{n-1})\grad\varphi_h^{n-1}}{\boldsymbol{\xi}_h},
		\label{eq:chbfea_a}\\
		& \inn{\divergence(\v_h^{n})}{\chi_h} =
		\tfrac12\alpha
		\inn{(\mathcal{P}\sigma_h^{n-1}-\mathcal{A})(\varphi_h^{n-1}+1)}{\chi_h}_h,
		\label{eq:chbfea_b} \\
		& \tfrac{1}{\tau} \inn{\varphi_{h}^{n} -
			\varphi_{h}^{n-1}}{\phi_{h}}_{h}
		+ \inn{\v^{n} \cdot\grad\varphi_h^{n-1}}{\phi_h}
		+ \inn{m(\varphi_h^{n-1}) \grad\mu_{h}^{n}}{\nabla \phi_{h}}_h
		\nonumber \\ & \qquad\qquad
		= \tfrac12 \inn{(\rho_S - \alpha\varphi_h^{n-1})
			(\mathcal{P}\sigma_h^{n-1}-\mathcal{A})(\varphi_h^{n-1}+1)}{\phi_h}_h
		, \label{eq:chbfea_c} \\
		& \inn{\mu_{h}^{n} + \tfrac{\beta}{\epsilon} \varphi_{h}^{n-1}
			+ \chi_{\varphi} \sigma_{h}^{n-1}}{\zeta_{h} -
			\varphi_{h}^{n}}_{h} \leq \beta \epsilon
		\inn{\grad \varphi_{h}^{n}}{\grad (\zeta_{h} -
			\varphi_{h}^{n})},
		\label{eq:chbfea_d} \\
		& \mathcal{D}\inn{\nabla \sigma_{h}^{n}}{\nabla \xi_h}
		+\tfrac{1}{2} \mathcal{C} \inn{\sigma_{h}^{n}(\varphi_{h}^{n} +
			1)}{\xi_h}_h
		= \mathcal{D}\chi \inn{\grad \varphi_{h}^{n}}{\nabla \xi_h}
		. \label{eq:chbfea_e}
		\end{align}
	\end{subequations}
	We implement \eqref{eq:chbfea} within the finite element package
	Alberta,
	\cite{Alberta}, and use adaptive meshes that are refined in the
	interfacial
	region, where $|\varphi_h^{n-1}| < 1$. In particular, away from the
	interface a
	coarse mesh corresponding to a uniform $N_c \times N_c$ grid is used,
	while the interfacial region is resolved with a mesh size corresponding
	to a
	uniform $N_f \times N_f$ grid. The precise strategy is described in
	\cite{voids}.
	We note that the time discretization in \eqref{eq:chbfea} is chosen such
	that
	the overall system decouples into three independent systems:
	the linear discrete Stokes problem
	\eqref{eq:chbfea_a}-\eqref{eq:chbfea_b},
	featuring the LBB stable lowest order Taylor--Hood element,
	the nonlinear discrete Cahn--Hilliard equation
	\eqref{eq:chbfea_c}-\eqref{eq:chbfea_d}, with the discrete variational
	inequality \eqref{eq:chbfea_d} due to the chosen obstacle potential,
	and the linear equation \eqref{eq:chbfea_e} for the nutrient
	approximation.
	In practice, for each time step,
	we first solve \eqref{eq:chbfea_a}-\eqref{eq:chbfea_b}
	with the help of a preconditioned GMRES iteration,
	followed by solving \eqref{eq:chbfea_c}-\eqref{eq:chbfea_d} with the
	Uzawa
	solver from \cite{vch}, see also \cite{voids3d}, before solving
	\eqref{eq:chbfea_e} with a direct solver.
	Here all the occuring linear problems, e.g.\ as part of the above
	iterative solvers and preconditioners, are solved with the help of the
	sparse factorization packages LDL, AMD (\cite{AmestoyDD04,Davis05}) or
	UMFPACK (\cite{Davis04}), depending on whether the systems are
	symmetric or not.
	\subsection{Results}
	Throughout we let $\Omega = (-3,3)^2$.
	As initial data we choose $\varphi_0 \in C^0(\overline \Omega)$ defined
	as
	\begin{equation} \label{eq:phi0}
	\varphi_0(\mathbf{x}) = \begin{cases}
	1 & r \leq -\tfrac12\pi\epsilon,\\
	- \sin( \frac{r(\mathbf{x})}\epsilon ) & |r(\mathbf{x})| <
	\tfrac12\pi\epsilon,\\
	-1 & r \geq \tfrac12\pi\epsilon,
	\end{cases}
	\end{equation}
	where
	\begin{align*} 
	r(\mathbf x) & = |\mathbf x| - \left(\tfrac12 +
	\tfrac1{40}\cos(2\theta)\right)
	\end{align*}
	and $\mathbf{x} = |\mathbf{x}|(\cos\theta, \sin\theta)^\intercal$.
	The first  initial profile related to $r$ is shown in Figure \ref{3_FIG_0}.
	Unless otherwise stated, we will always use the following set of
	parameters
	\begin{equation}\label{3_NUMS_Parameters}
	\begin{aligned}
	&\epsilon = 0.02,\quad \alpha = 0.5,\quad\rho_S = 2,\quad
	\mathcal{P}=0.1,\quad \mathcal{A}=0,\quad \mathcal{C}=2,\quad
	\chi_{\varphi}=5,   \\
	&\mathcal{D} = 1,\quad \sigma_B = 1,\quad\chi = 0.02,\quad\lambda =
	0,\quad\nu = 100,\quad \del_1\Omega = \Omega.
	\end{aligned}
	\end{equation}
	For the discretization parameters we always choose $N_c = 16$, $N_f = 1024$ and $\tau = 10^{-4}$.
	\begin{figure}[!h]
		\centering
		\includegraphics[angle=-0,width=0.31\textwidth]{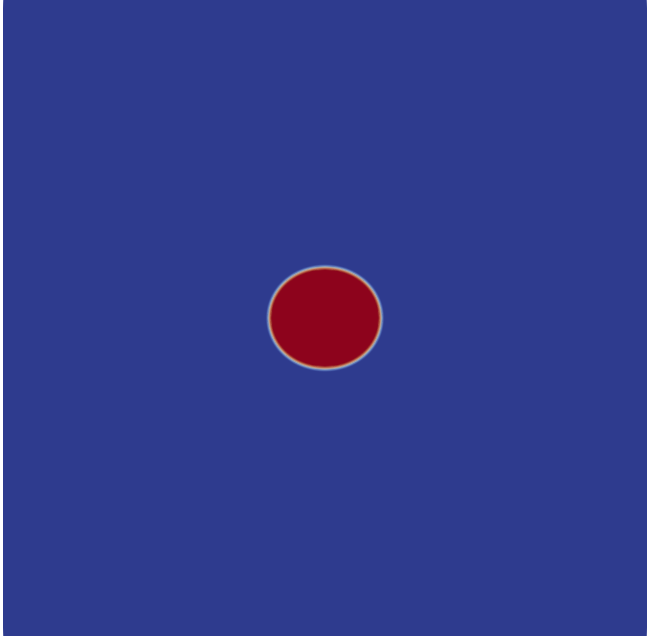}
		\caption{Initial tumour size for initial data $r$: A slightly perturbed sphere.}
		\label{3_FIG_0}
	\end{figure}
	
	\noindent We will now systematically interpret the influence of
	different parameters in our model.
	\subsection{Brinkman's and Darcy's law}
	In the following we investigate the relation of the
	Cahn--Hilliard--Brinkman (CHB) and Cahn--Hilliard--Darcy (CHD) models.
	For small viscosities we expect a similar qualitative behaviour of
	solutions to the corresponding systems.
	For the mobility we take $m(s)=\frac{1}{2}(1+s)^2$, which corresponds to
	\eqref{3_mobility_NUMS}\textnormal{(iii)} with $m_0=1$.
	In Figure \ref{3_FIG_1} we show the tumour for both the CHD and CHB model for
	$\eta=10^{-5}$ at time $t=12$.
	We see that the qualitative behaviour for both models is similar for low
	viscosities.
	\begin{figure}[!h]
		\centering
		\includegraphics[width =
		0.3\textwidth]{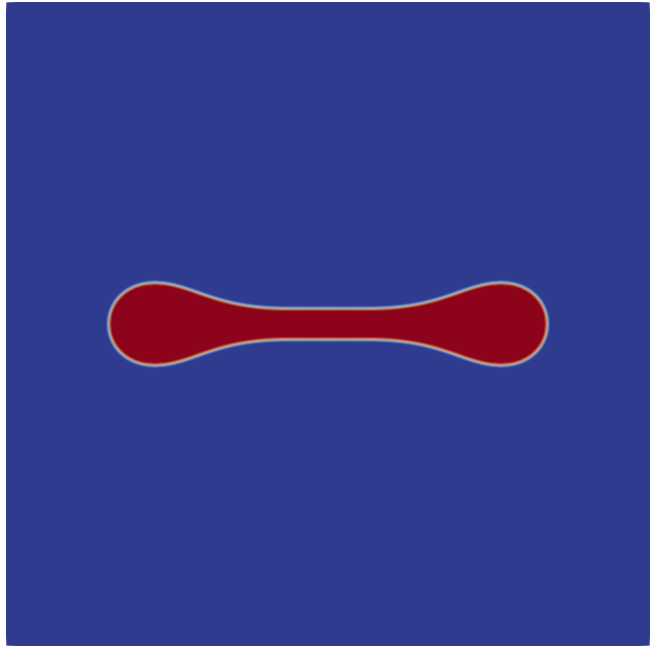}\hspace*{+10pt}
		\includegraphics[width
		=0.345\textwidth]{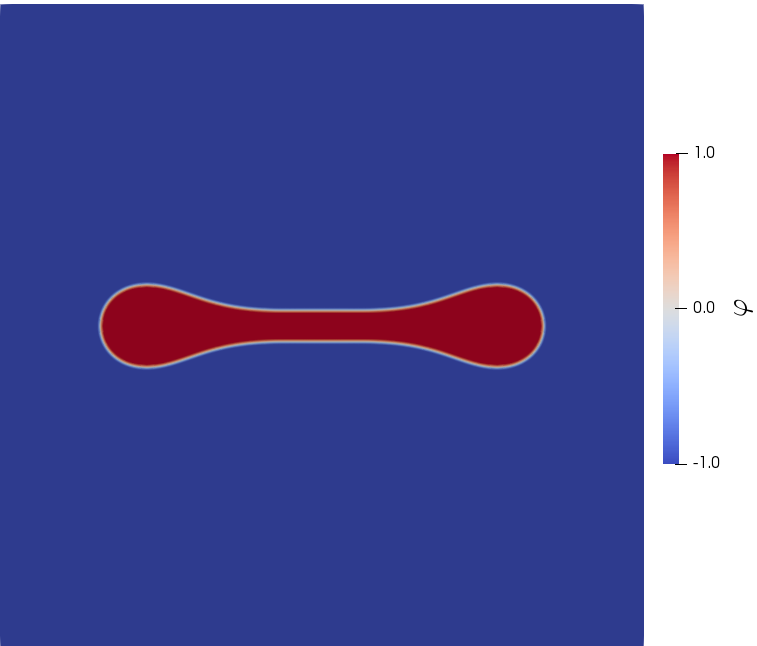}
		\caption{Comparison of Cahn--Hilliard--Darcy and Cahn--Hilliard--Brinkman models: Tumour at time $t=12$ for $\beta=0.1$, left side for
			the CHD model, right side for the CHB model with $\eta = 10^{-5}$.}
		\label{3_FIG_1}
	\end{figure}
	
	\subsection{Influence of mobility and adhesion}
	We now investigate the influence of the mobility and the cell-cell adhesion.
	In Figure \ref{3_FIG_2} we show the evolutions with $\eta=10^{-5}$
	and for different mobilities. The formal asymptotic analysis in the
	previous section indicates that the mobility
	\eqref{3_mobility_NUMS}\textnormal{(ii)}, corresponds to a free boundary
	problem where the interface is transported solely by the fluid velocity.
	\begin{figure}[!h]
		\centering
		\includegraphics[width =
		0.301\textwidth]{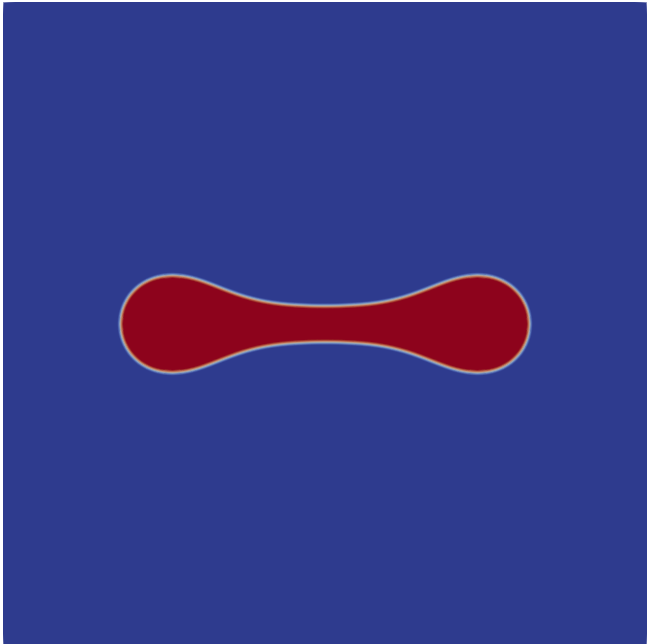}\hspace*{+10pt}
		\includegraphics[width =
		0.3\textwidth]{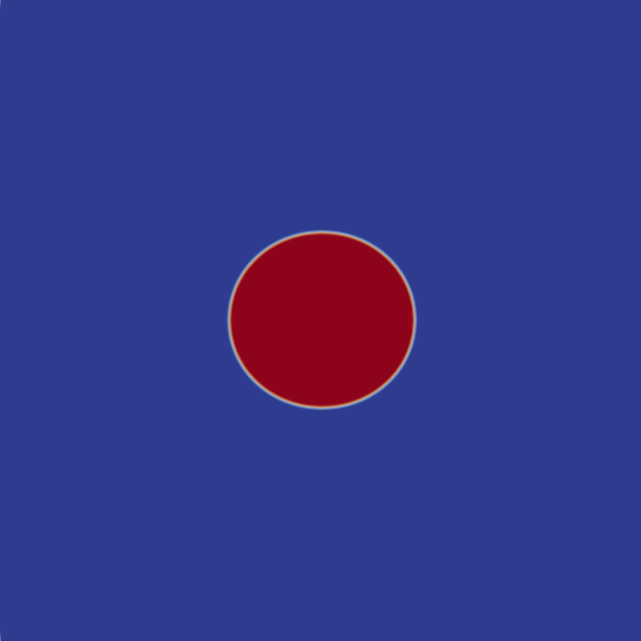}\hspace*{+10pt}
		\includegraphics[width =
		0.36\textwidth]{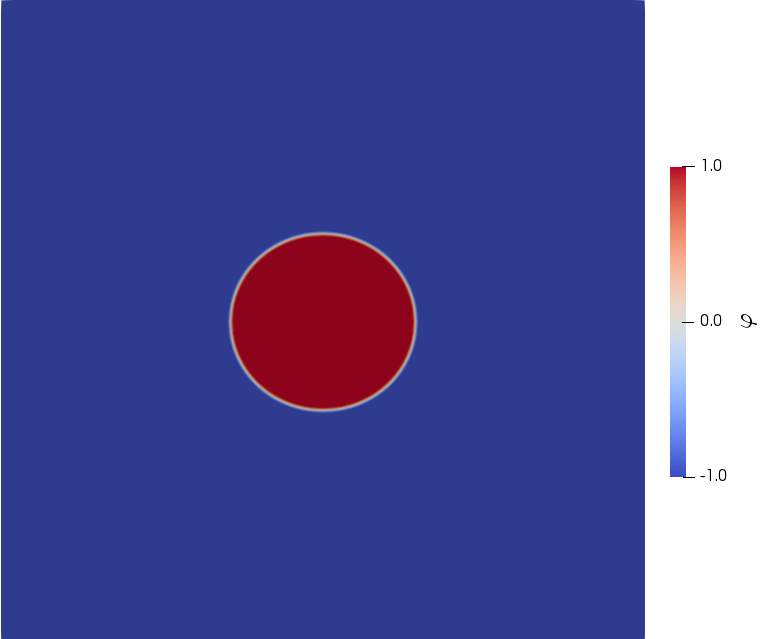}
		\caption{Influence of different mobilities: Tumour at time $t=9$ for $\eta=10^{-5}$, $\beta=0.1$
			and $\alpha=\rho_S=2$, but with different mobilities, left
			$m(\varphi)=\frac{1}{2}(1+\varphi)^2$, middle $m(\varphi) = \epsilon$,
			right $m(\varphi)=10^{-3}\epsilon$.}
		\label{3_FIG_2}
	\end{figure}
	\newline Thus, we see that a one-sided degenerate mobility causes
	instabilities while pure transport by the velocity stabilises the
	interface. Moreover, having a closer look we see that the thickness of
	the interface is smaller for the mobility
	$m(\varphi)=10^{-3}\epsilon$.\newline
	As the Ginzburg--Landau energy models adhesion forces, it can be
	expected that a reduction of the parameter $\beta>0$ reduces adhesion
	forces and leads to instabilities. In Figure \ref{3_FIG_3}, we compare
	the tumour evolutions for $\beta\in \{0.1,0.01\}$ with $ \eta = 0.1$ and
	for the mobility $m(\varphi)=\tfrac{1}{2}(1+\varphi)^2$. We see that the
	instabilities are more pronounced for $\beta=0.01$ and the fingers are
	longer and thinner.
	\begin{figure}[!h]
		\centering
		\includegraphics[width =
		0.22\textwidth]{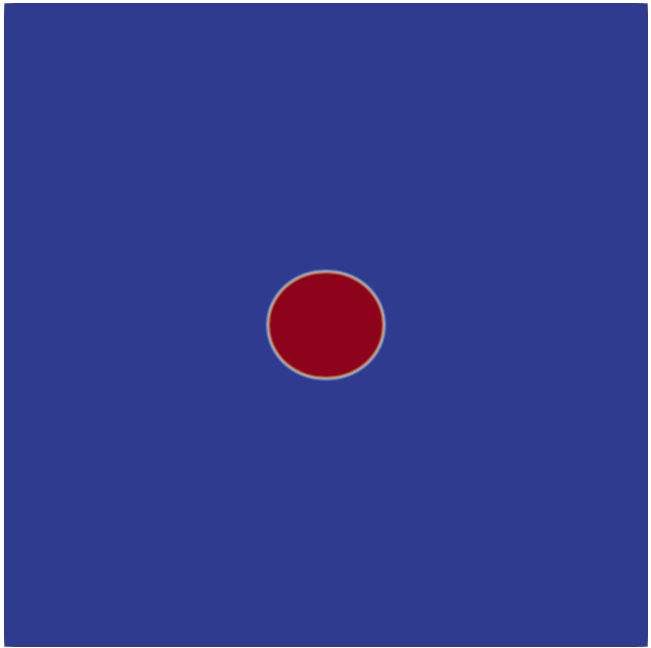}\hspace*{+10pt}
		\includegraphics[width =
		0.22\textwidth]{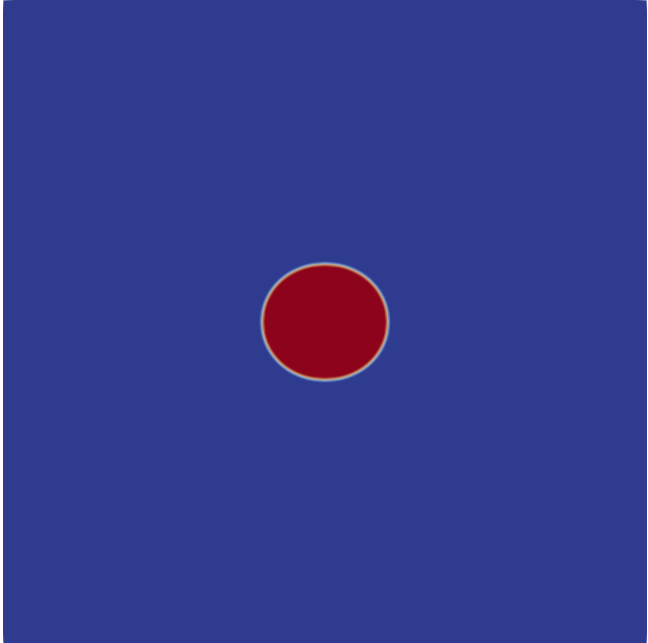}\hspace*{+10pt}
		\includegraphics[width =
		0.22\textwidth]{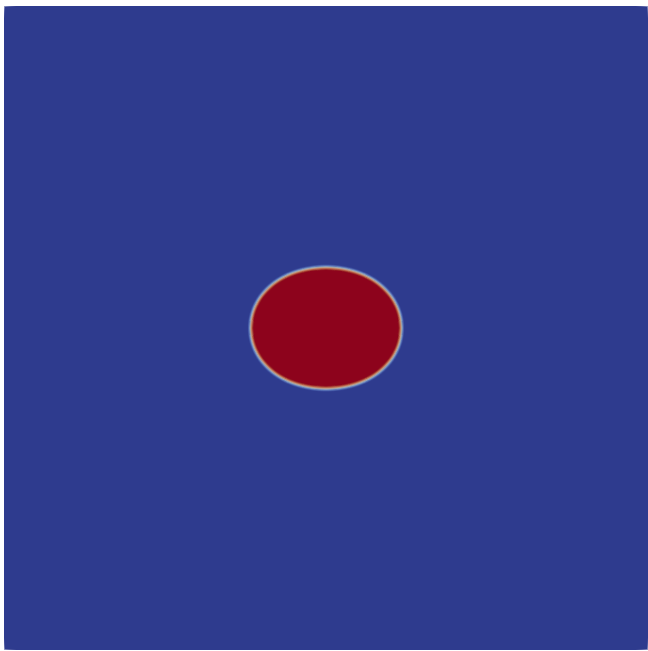}\hspace*{+10pt}
		\includegraphics[width =
		0.22\textwidth]{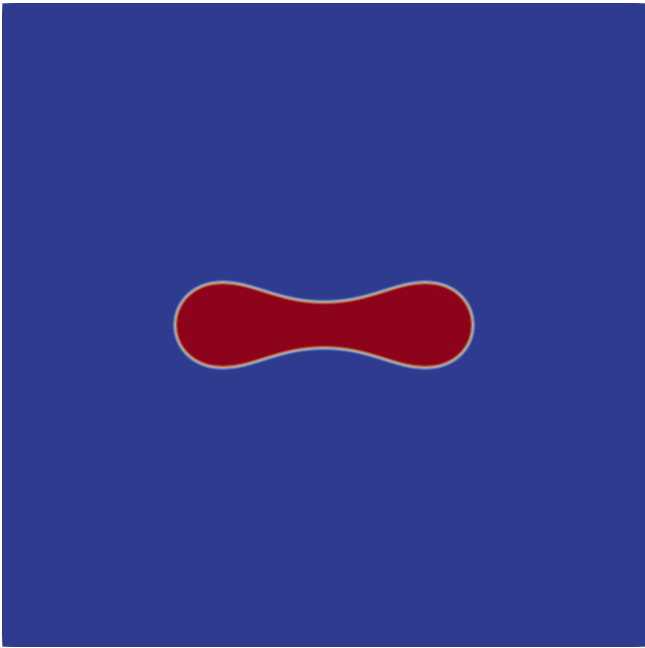}\\
		\vspace*{+10pt}
		\includegraphics[width =
		0.22\textwidth]{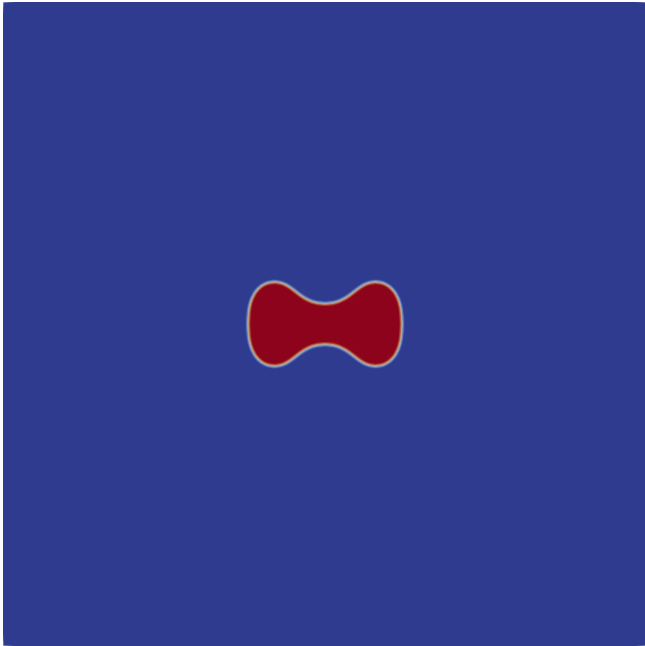}\hspace*{+10pt}
		\includegraphics[width =
		0.22\textwidth]{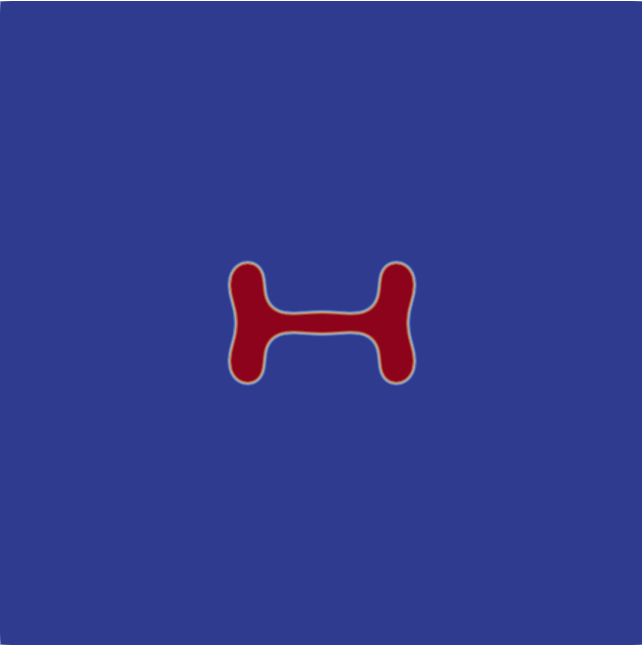}\hspace*{+10pt}
		\includegraphics[width =
		0.22\textwidth]{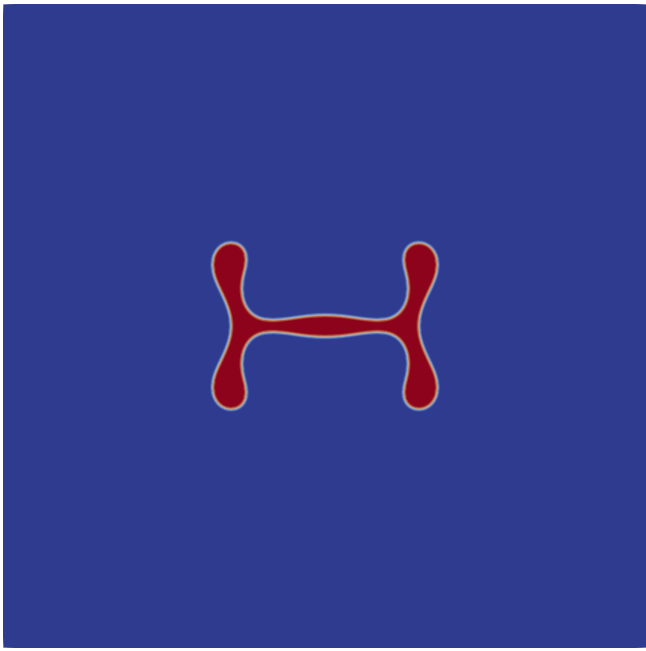}\hspace*{+10pt}
		\includegraphics[width =
		0.22\textwidth]{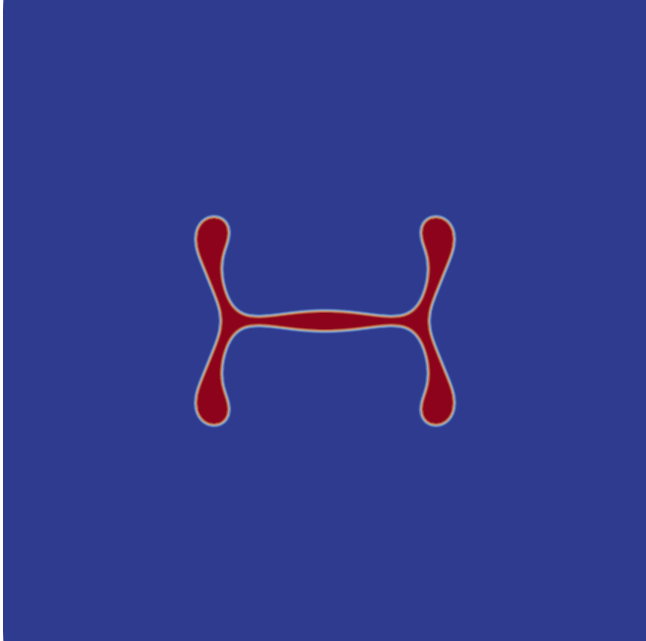}
		\caption{Influence of the adhesion parameter $\beta$: Evolution of the tumour with
			$m(\varphi)=\tfrac{1}{2}(1+\varphi)^2$ and $\eta=0.1$, above for $\beta
			= 0.1$  at time $t=1,3,6,10$, below for $\beta = 0.01$ at time
			$t=1,1.5,2,2.5$.}
		\label{3_FIG_3}
	\end{figure}
	
	\subsection{Influence of the viscosity}
	Next we investigate the influence of the viscosity and we always take
	the one-sided degenerate mobility
	$m(\varphi)=\tfrac{1}{2}(1+\varphi)^2$. \newline In Figure \ref{3_FIG_5}, we
	compare the tumour at time $t=2.5$ for constant viscosities
	$\eta\in\{0.1,100\}$ and the Neumann boundary condition for the stress
	tensor. We see that the results look nearly identical. We also plot the
	velocity magnitude which is slightly bigger for $\eta=0.1$. Thus, it
	seems that the influence of viscosity in the case of stress free
	boundary conditions is rather low.
	\begin{figure}[!h]
		\centering
		\hspace*{-20pt}\includegraphics[width =
		0.3\textwidth]{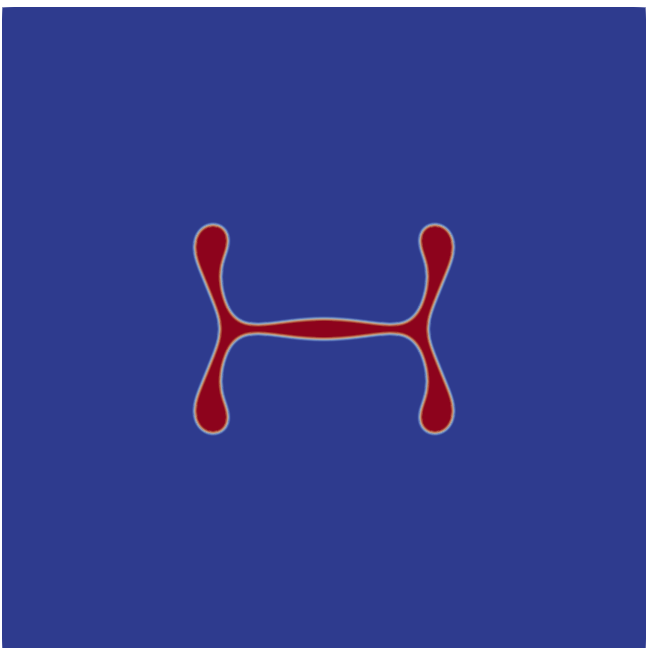}\hspace*{+39pt}
		\includegraphics[width =
		0.35\textwidth]{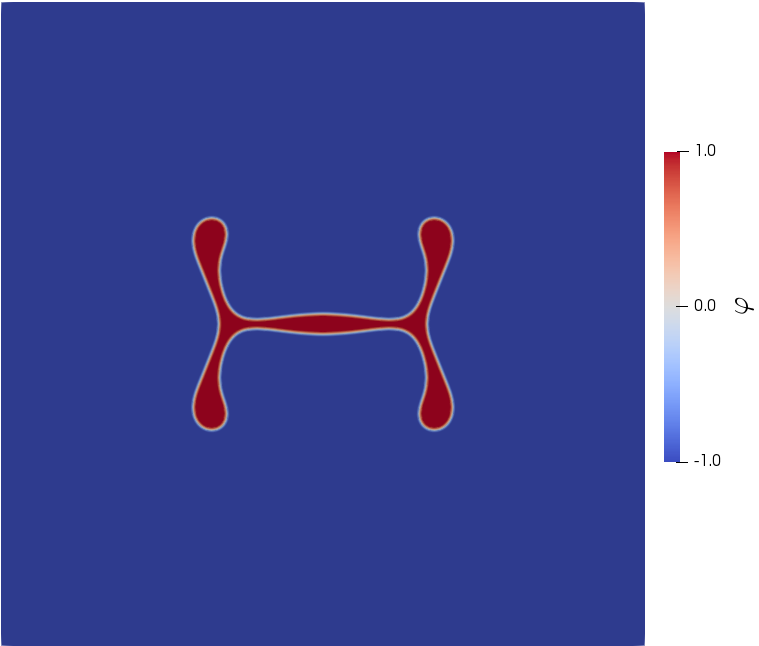}\\
		\vspace*{+10pt}
		\includegraphics[width =
		0.37\textwidth]{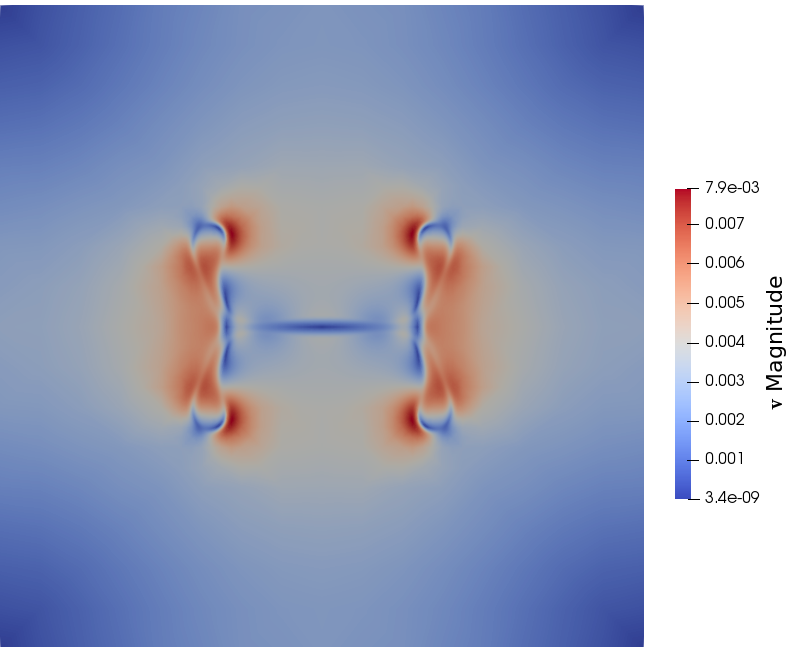}\hspace*{+10pt}
		\includegraphics[width =
		0.37\textwidth]{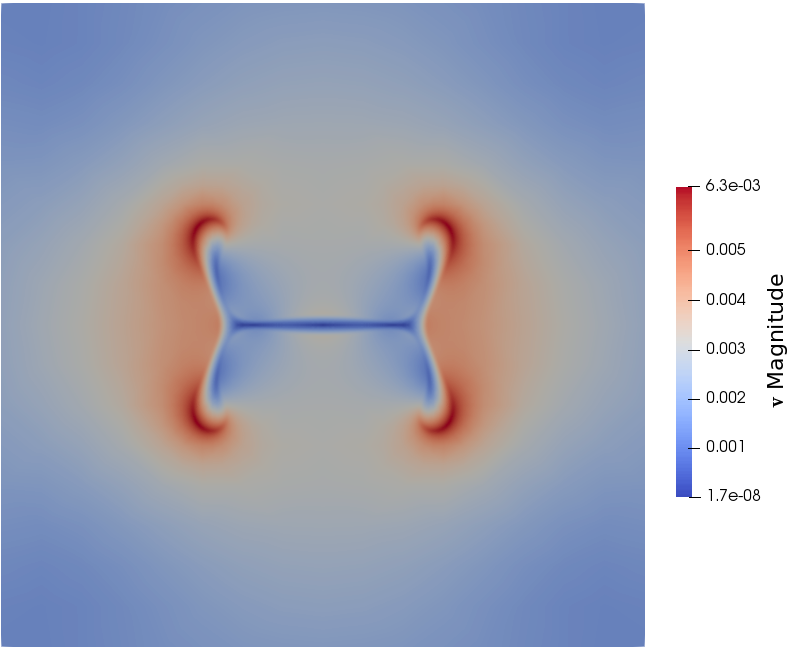}\hspace*{+10pt}
		\caption{Influence of viscosity I: Tumour and velocity for $\beta=0.01$ at time $t=2.5$,
			left for $\eta=0.1$, right for $\eta=100$, on top the tumour and below
			the velocity magnitude.}
		\label{3_FIG_5}
	\end{figure}
	\newline In the case of no-slip conditions on one part of the boundary
	we observe a different situation. In Figure \ref{3_FIG_6}, we plot the
	evolution for $\eta\in \{ 0.1,10\}$ with $\nu=0$, $\beta=0.1$ and a
	no-slip boundary condition on the left boundary, i.\,e., $\del_2\Omega =
	\{-3\}\times (-3,3)$. We see that for low viscosity the tumour evolves
	radially symmetric whereas instabilities appear if the viscosity is
	higher.
	\begin{figure}[!h]
		\centering
		\includegraphics[width =
		0.22\textwidth]{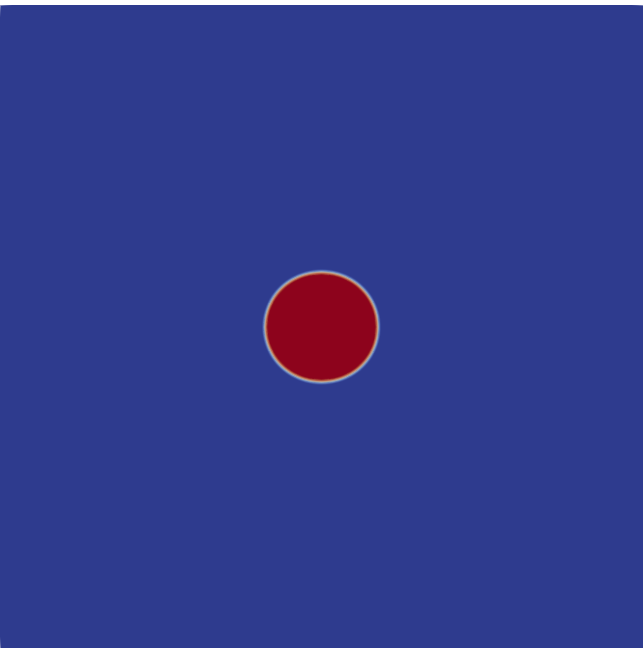}\hspace*{+10pt}
		\includegraphics[width =
		0.22\textwidth]{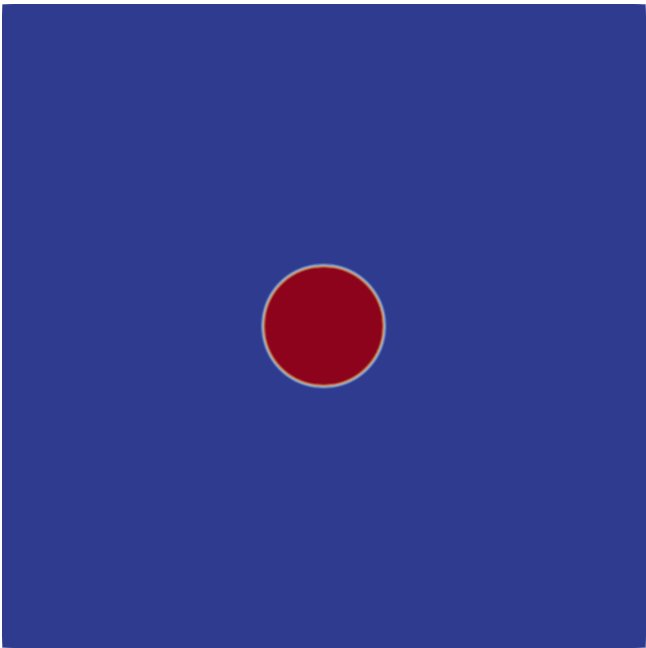}\hspace*{+10pt}
		\includegraphics[width =
		0.22\textwidth]{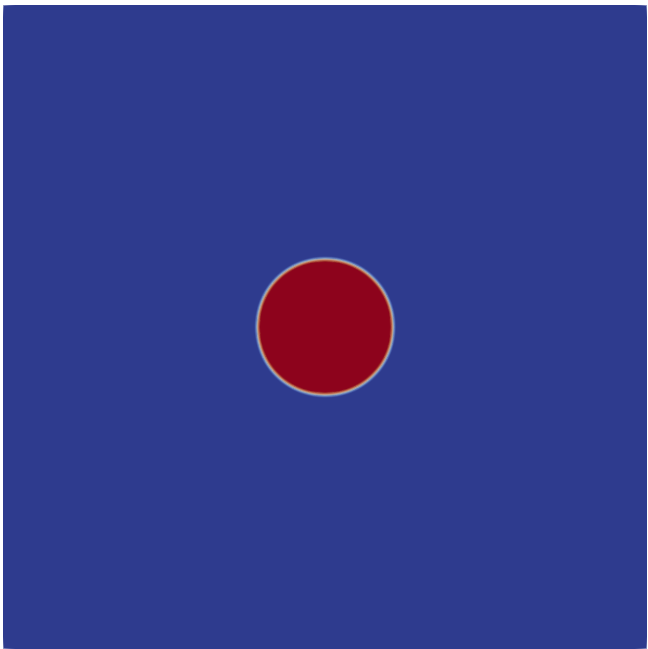}\hspace*{+10pt}
		\includegraphics[width =
		0.22\textwidth]{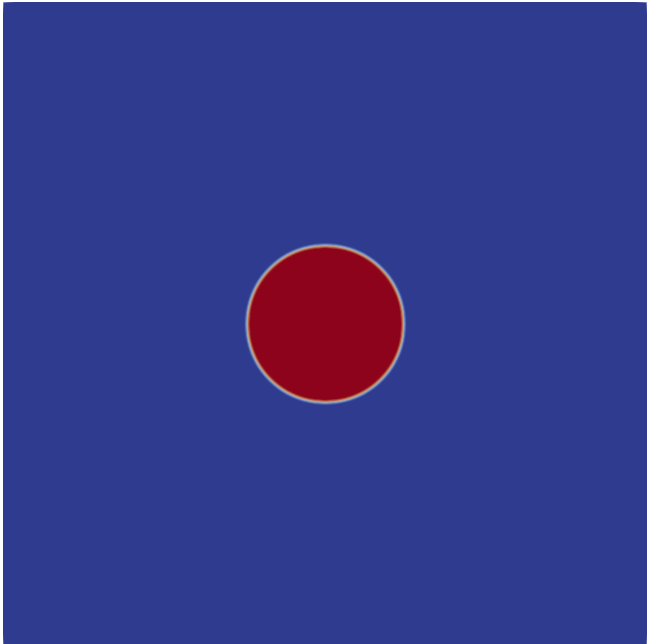}\\
		\vspace*{+10pt}
		\includegraphics[width =
		0.22\textwidth]{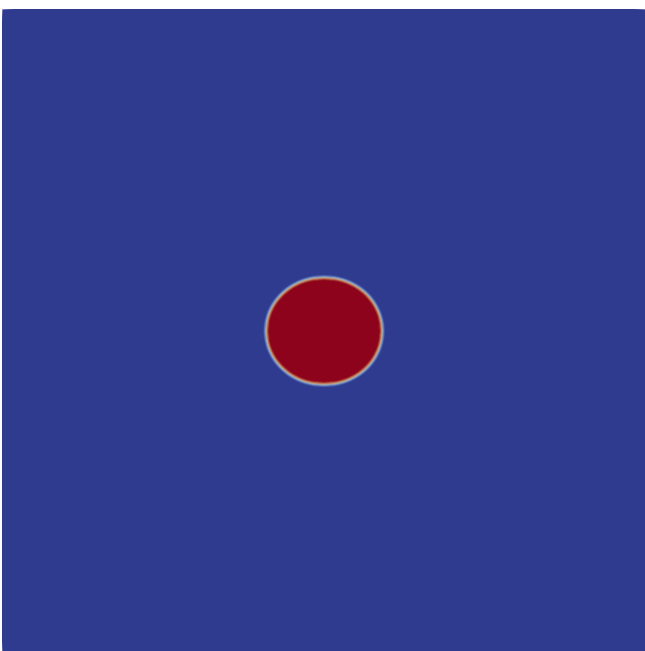}\hspace*{+10pt}
		\includegraphics[width =
		0.22\textwidth]{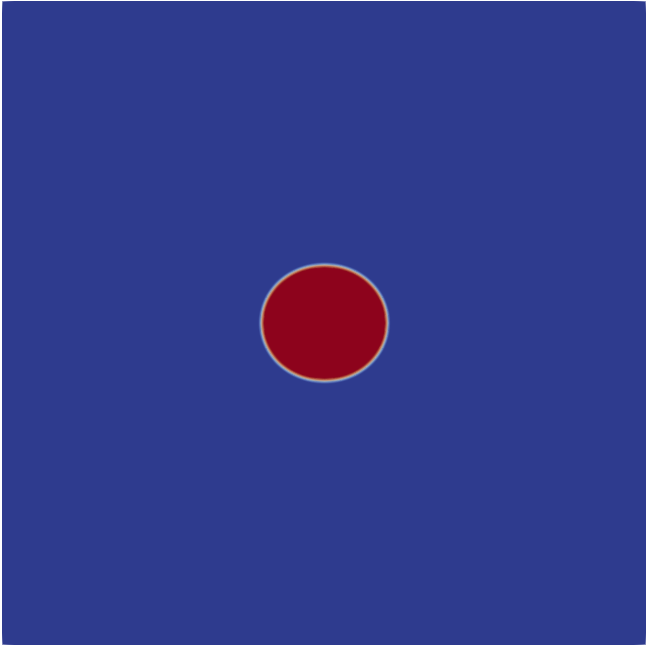}\hspace*{+10pt}
		\includegraphics[width =
		0.22\textwidth]{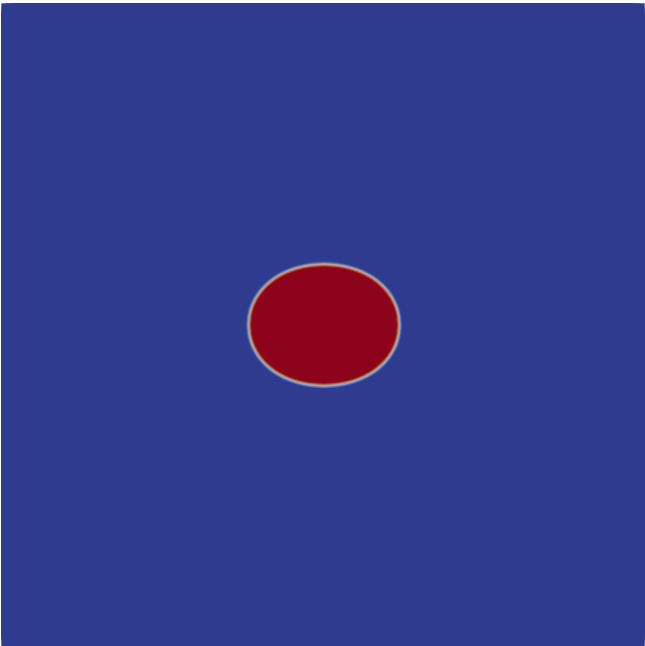}\hspace*{+10pt}
		\includegraphics[width =
		0.22\textwidth]{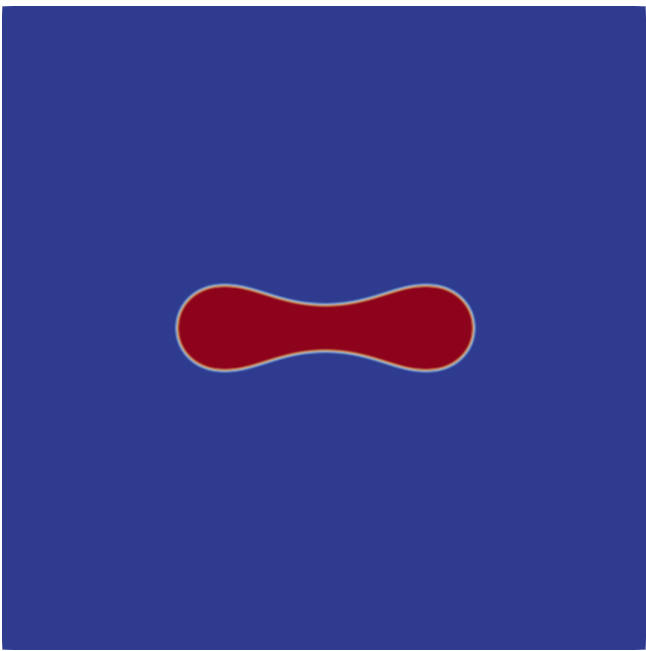}
		\caption{Influence of viscosity II: Evolution of the tumour at time $t=1,3,6,10$ with
			$\beta=0.1$, $\nu=0$ and a no-slip boundary condition on the left
			boundary, on top for $\eta=0.1$ and below for $\eta=10$.}
		\label{3_FIG_6}
	\end{figure}
	
	\noindent We also show the velocity magnitudes at $t=10$ in
	Figure \ref{3_FIG_7}. Although the maximal magnitudes are almost the same, we
	see more regions with high velocity if the viscosity is bigger, that
	means for $\eta=10$. It is also worth noticing that the velocity field
	is no longer symmetric as observed in Figure \ref{3_FIG_5} which is due to the
	no-slip boundary condition.
	\begin{figure}[!h]
		\centering
		\includegraphics[width =
		0.35\textwidth]{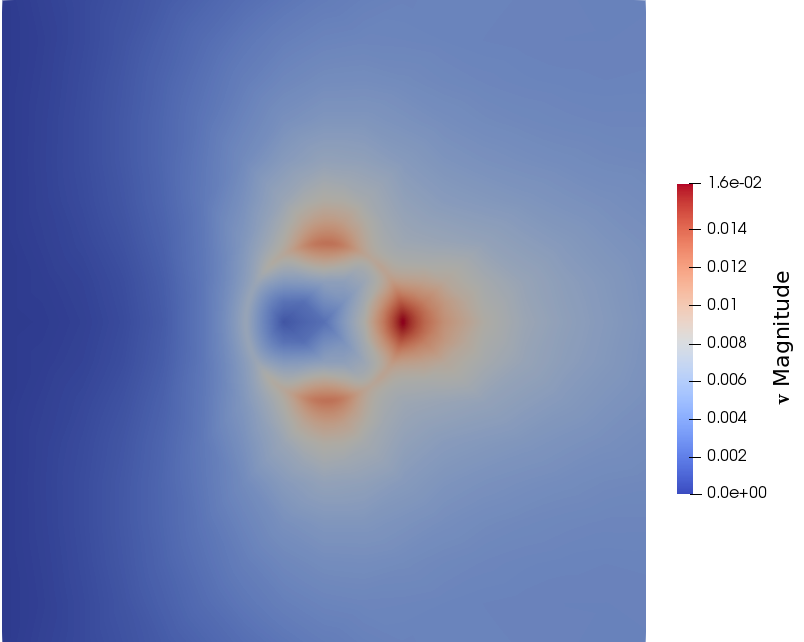}\hspace*{+10pt}
		\includegraphics[width =
		0.35\textwidth]{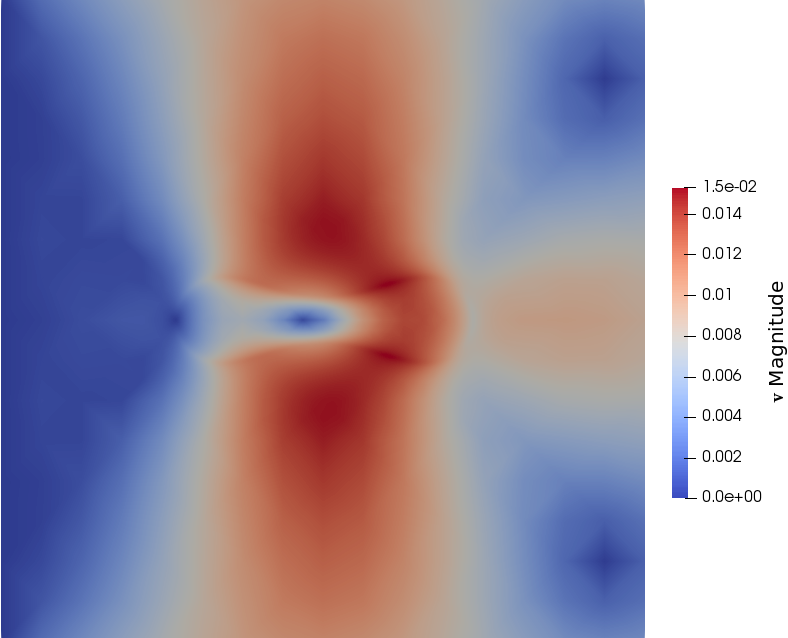}
		\caption{Velocity profiles for different viscosities: The velocity magnitude at time $t=10$ with $\beta=0.1$,
			$\nu=0$ and a no-slip boundary condition on the left boundary, left for
			$\eta=0.1$, right for $\eta=10$.}
		\label{3_FIG_7}
	\end{figure}
	
	 We also investigate the influence of different viscosities
	for the no-slip boundary condition. We denote by $\eta_+\coloneqq \eta(1)$ and $\eta_-\coloneqq \eta(-1)$
	the viscosities in the tumour and healthy phase, respectively. In
	Figure \ref{3_FIG_9}, we show the tumour at time $t=10$ for different cases.
	\begin{figure}[!h]
		\centering
		\includegraphics[width =
		0.22\textwidth]{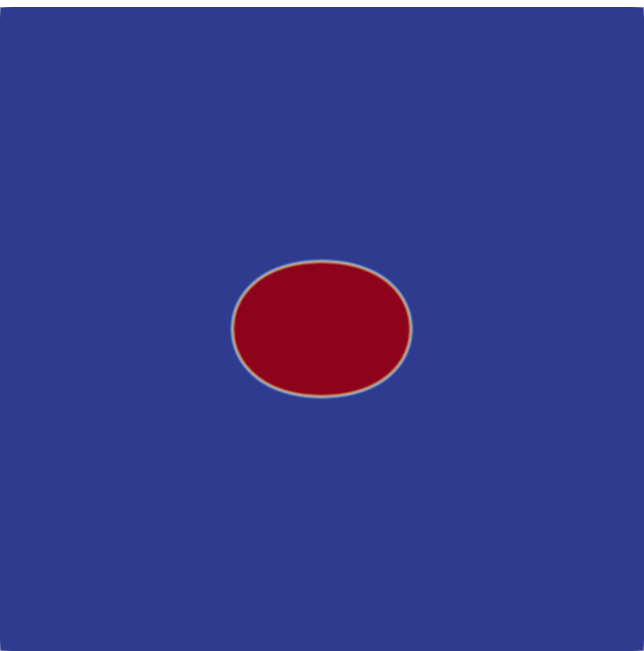}\hspace*{+10pt}
		\includegraphics[width =
		0.22\textwidth]{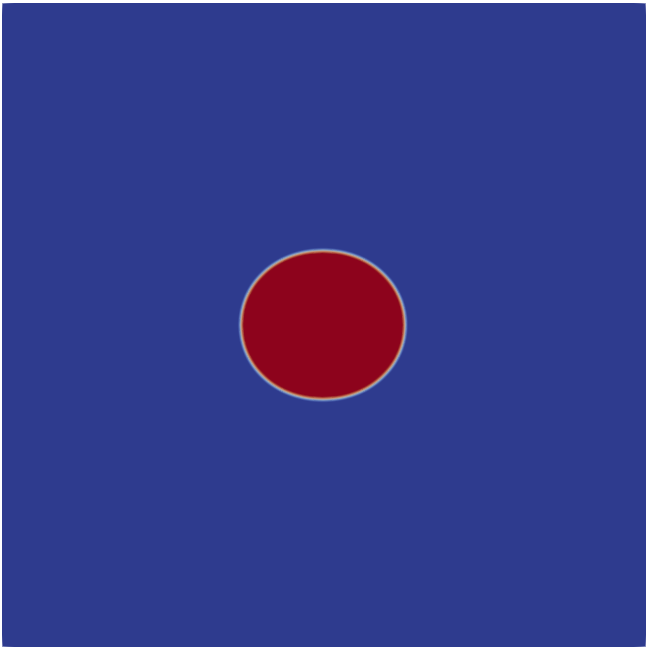}\hspace*{+10pt}
		\includegraphics[width =
		0.22\textwidth]{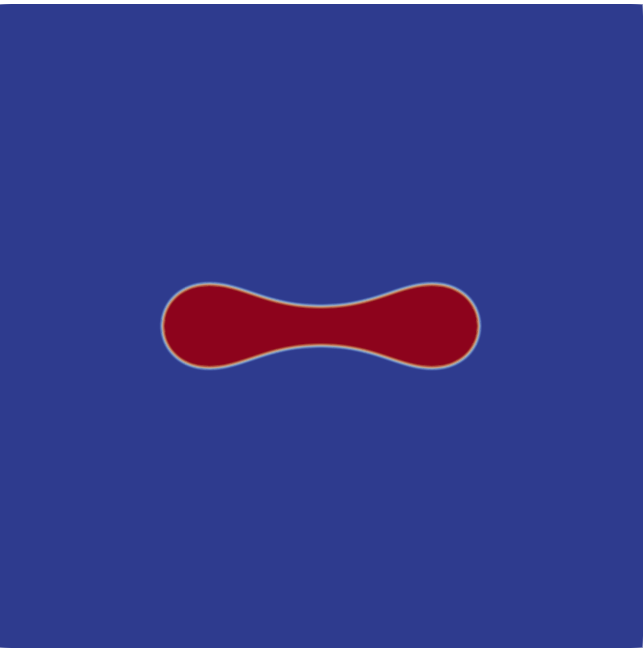}\hspace*{+10pt}
		\includegraphics[width =
		0.22\textwidth]{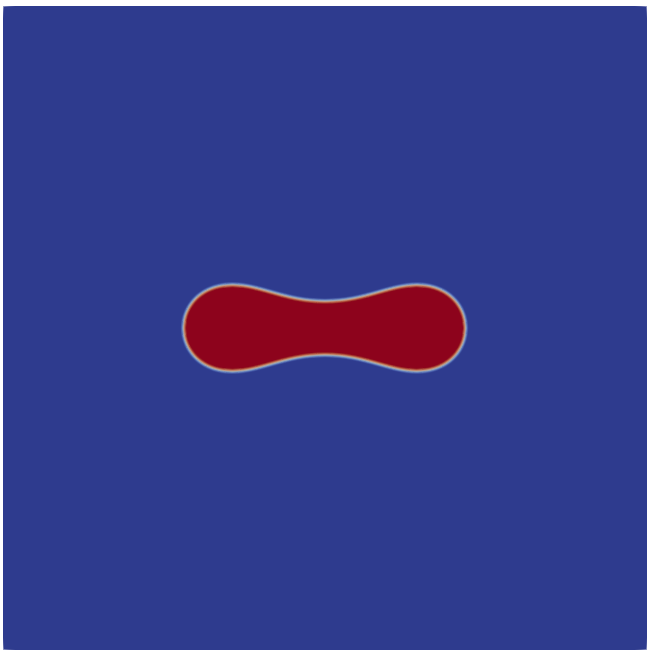}
		\caption{Influence of viscosity contrast: Tumour at time $t=10$ with $\beta=0.1$, $\nu=0$ and a
			no-slip b.\,c. on the left boundary, with $\eta_- = 0.01$, $\eta_+ = 1$;
			$\eta_- = 1$, $\eta_+ = 0.01$; $\eta_- = 0.01$, $\eta_+ = 10$; $\eta_- =
			10$, $\eta_+ = 0.01$.}
		\label{3_FIG_9}
	\end{figure}
	\newline It can be seen that a large difference between the viscosities
	leads to a more interesting evolution. Moreover, instabilities are more
	pronounced if the viscosity in the surroundings is lower than in the
	tumour tissue. Thus, the tumour tends to grow towards directions with
	least resistance. This effect has also been observed in a theoretical
	analysis in \cite{FranksKing2}.
	\subsection{Influence of different initial profiles}
	Here we want to study the influence of different initial profiles. In particular, we will see that some modes of the perturbation of a sphere are stable while other modes are unstable. We always choose $\lambda = 0.02$, $\beta=0.01$ and leave the remaining parameters as in \eqref{3_NUMS_Parameters}. As initial data we choose \eqref{eq:phi0} with $r$ replaced by the following different choices
	\begin{align*}
	r_1(\mathbf x) &= |\mathbf x| - (\tfrac12 +
	\tfrac1{40}\cos(6\theta)),\\
	r_2(\mathbf x) &= |\mathbf x| - (\tfrac12 +
	\tfrac1{40}\cos(12\theta-\tfrac{\pi}{9})),\\
	r_3(\mathbf x) &= |\mathbf x| - (\tfrac12 + 10^{-3}[\cos(2\theta)
	+ \tfrac54\cos(6\theta - \tfrac\pi{12})
	+ \tfrac34\cos(8\theta - \tfrac\pi7)]),\\
	r_4(\mathbf x) & = |\mathbf x| - (\tfrac12 + 10^{-3}[\cos(12\theta)
	+ \tfrac54\cos(7\theta - \tfrac\pi{12})
	+ \tfrac34\cos(8\theta - \tfrac\pi7)]).
	\end{align*}
	We show the evolution for the initial profile with $r_1(\cdot)$ in Figure \ref{Figure_noise6}, where we see that a 6-fold perturbation leads to six enhanced fingers.
	\begin{figure}[!h]
		\centering
		\includegraphics[width =
		0.22\textwidth]{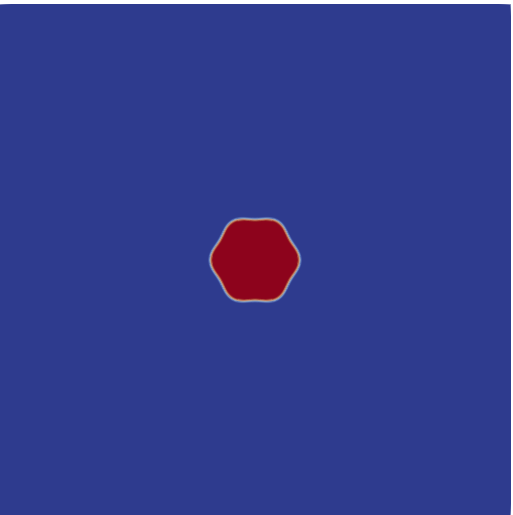}\hspace*{+10pt}
		\includegraphics[width =
		0.22\textwidth]{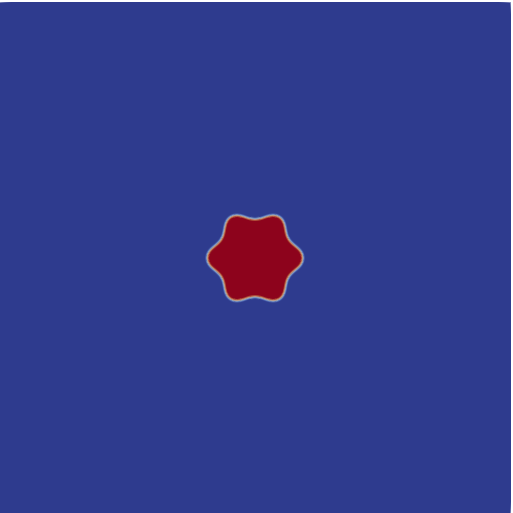}\hspace*{+10pt}
		\includegraphics[width =
		0.22\textwidth]{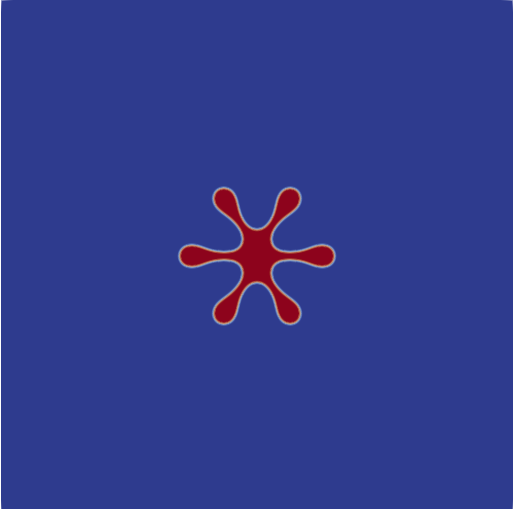}\hspace*{+10pt}
		\includegraphics[width =
		0.22\textwidth]{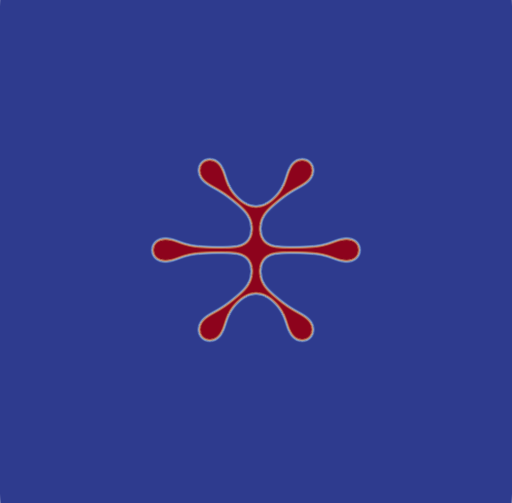}
		\caption{Influence of initial profile I: Tumour at time $t=0,0.3,1,1.6$ with $\eta=100$ and with the initial profile corresponding to $r_1$.}
		\label{Figure_noise6}
	\end{figure}
	
	The evolution for the initial profile $r_2(\cdot)$ is shown in
	Figure \ref{Figure_noise12}. The 12-fold perturbation is damped and the
	tumour region becomes nearly round. Finally, an instability with  four
	enhanced fingers arises.
	\begin{figure}[!h]
		\centering
		\includegraphics[width =
		0.22\textwidth]{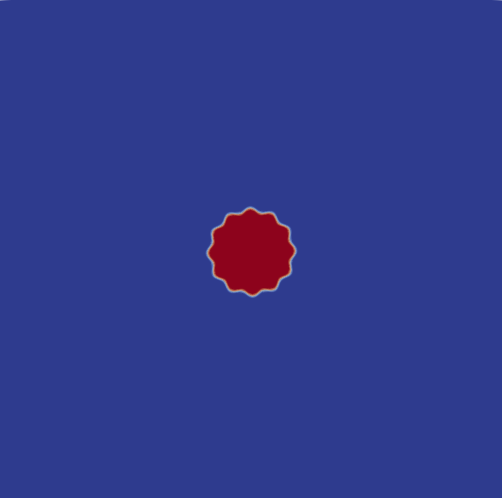}\hspace*{+10pt}
		\includegraphics[width =
		0.22\textwidth]{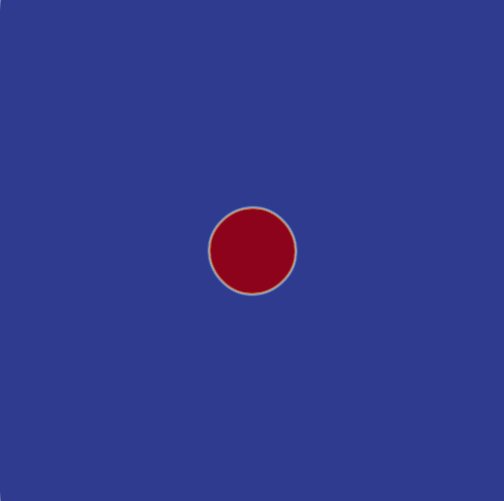}\hspace*{+10pt}
		\includegraphics[width =
		0.22\textwidth]{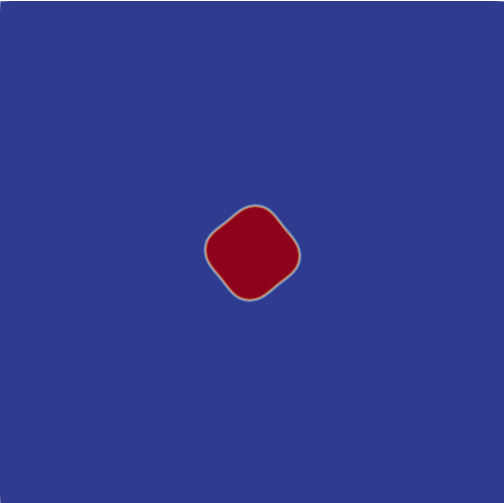}\hspace*{+10pt}
		\includegraphics[width =
		0.22\textwidth]{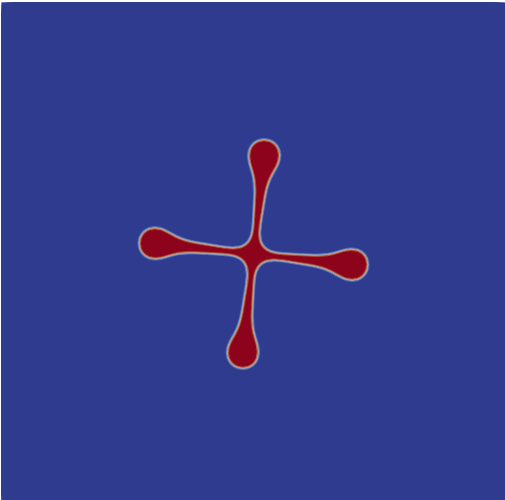}
		\caption{Influence of initial profile II: Tumour at time $t=0,0.7,1.2,2.6$ with $\eta=100$ and with the initial profile corresponding to $r_2$.}
		\label{Figure_noise12}
	\end{figure}
	
	Next, we show the evolution for the initial profile $r_3(\cdot)$ in
	Figure \ref{Figure_noise-2}. Here, six enhanced fingers  evolve and the final tumour is asymmetric.
	\begin{figure}[!h]
		\centering
		\includegraphics[width =
		0.22\textwidth]{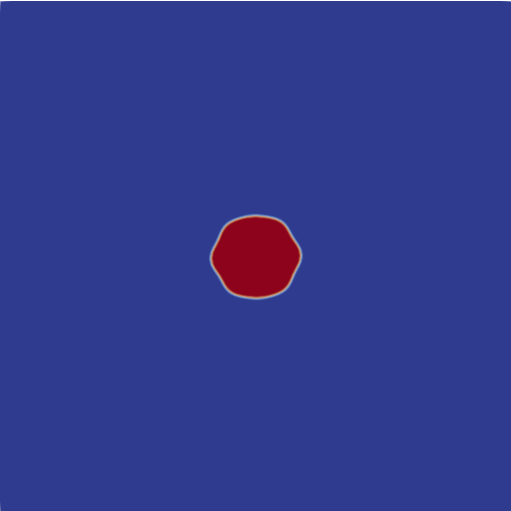}\hspace*{+10pt}
		\includegraphics[width =
		0.22\textwidth]{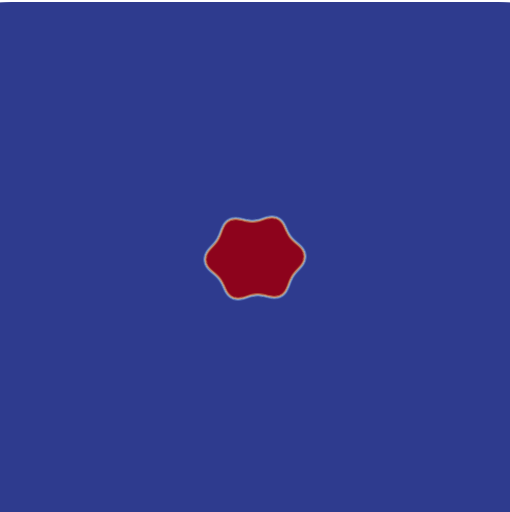}\hspace*{+10pt}
		\includegraphics[width =
		0.22\textwidth]{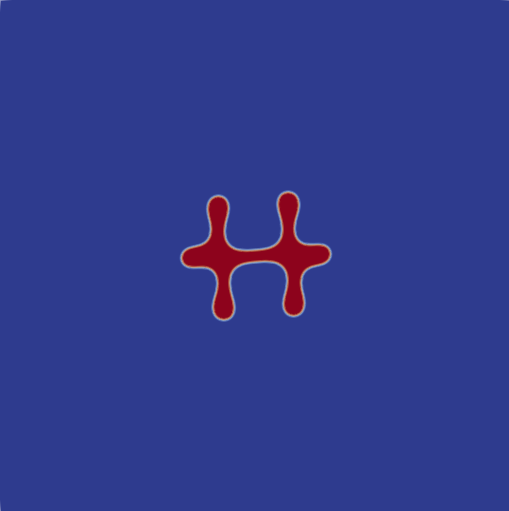}\hspace*{+10pt}
		\includegraphics[width =
		0.22\textwidth]{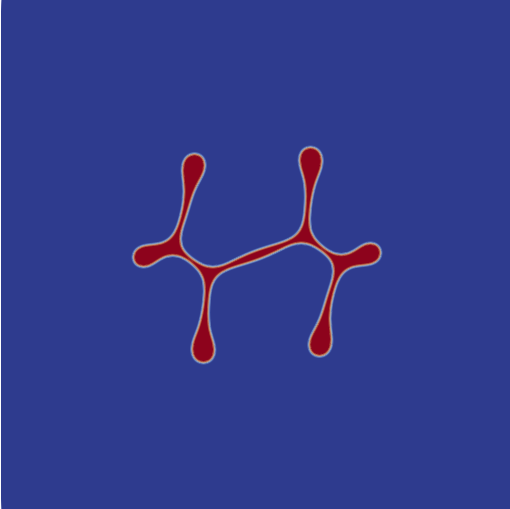}
		\caption{Influence of initial profile III: Tumour at time $t=0,0.5,1.2,2.4$ with $\eta=100$ and with the initial profile corresponding to $r_3$.}
		\label{Figure_noise-2}
	\end{figure}
	
	Finally, we show the evolution corresponding to $r_4(\cdot)$ in Figure
	\ref{Figure_noise-3}. Similar as in Figure \ref{Figure_noise12}, four
	fingers evolve and two of them are more elongated, and the final tumour
	is quite  asymmetric.
	\begin{figure}[!h]
		\centering
		\includegraphics[width =
		0.22\textwidth]{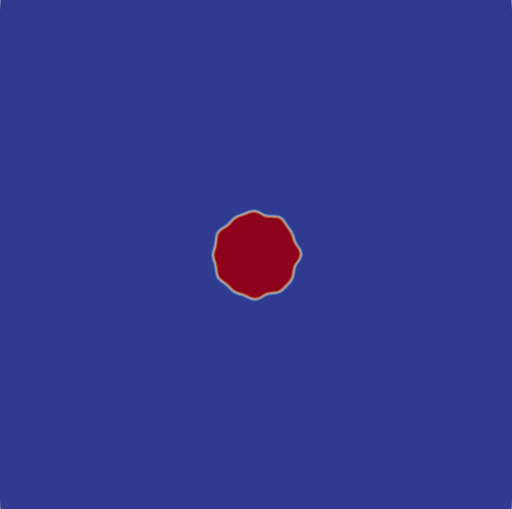}\hspace*{+10pt}
		\includegraphics[width =
		0.22\textwidth]{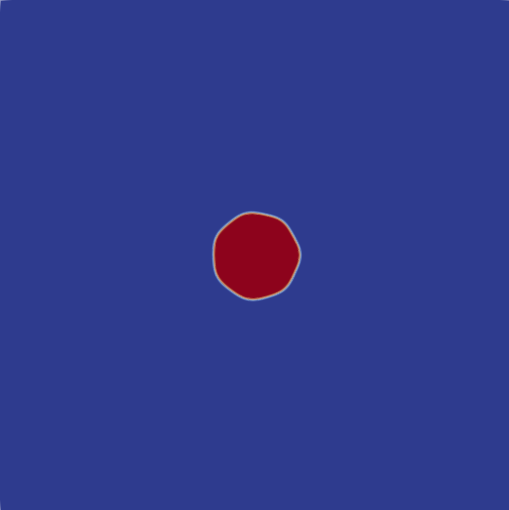}\hspace*{+10pt}
		\includegraphics[width =
		0.22\textwidth]{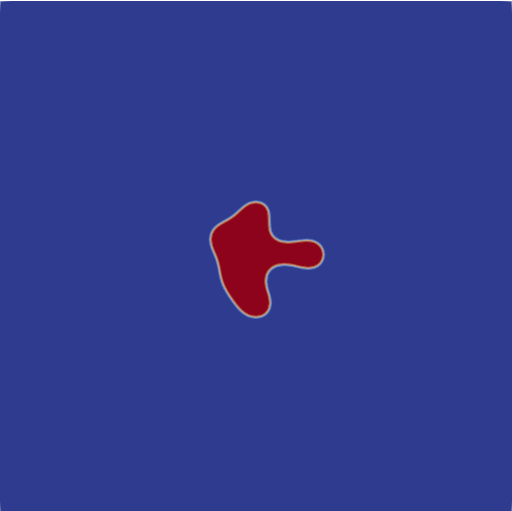}\hspace*{+10pt}
		\includegraphics[width =
		0.22\textwidth]{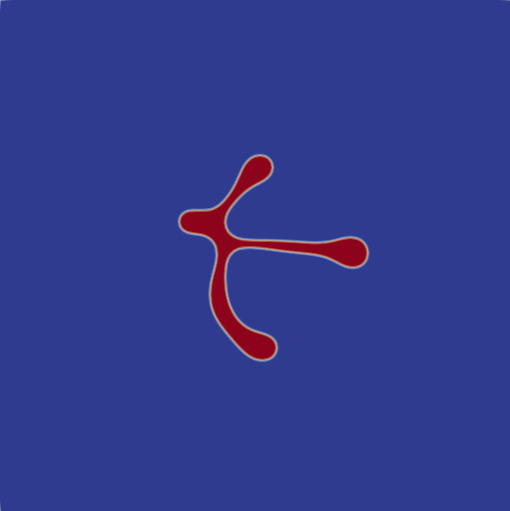}
		\caption{Influence of initial profile IV: Tumour at time $t=0,0.3,1.3,2.3$ with $\eta=0.01$ and with the initial profile corresponding to $r_4$.}
		\label{Figure_noise-3}
	\end{figure}
	\section*{Acknowledgments} The authors gratefully acknowledge the
	support by the RTG 2339 ``Interfaces, Complex Structures, and Singular
	Limits'' of the German Science Foundation (DFG) and by the Regensburger
	Universit\"atsstiftung Hans Vielberth.
	
	\newpage
	\printbibliography

\end{document}